\documentclass[a4paper, intlimits, reqno]{amsart}

\usepackage[english]{babel}
\usepackage[T1]{fontenc}
\usepackage[utf8]{inputenc}

\usepackage{amsmath}
\usepackage{amssymb}
\usepackage{MnSymbol}
\usepackage{amsthm}
\allowdisplaybreaks
\usepackage{amsfonts}
\usepackage{mathrsfs} 
\usepackage{mathtools}
\usepackage{enumitem}
\usepackage{multicol}
\usepackage{dsfont}
\usepackage[numbers,sort&compress]{natbib}
\usepackage{doi}
\usepackage{prettyref}
\usepackage{orcidlink}

\newrefformat{defn}{Definition \ref{#1}}
\newrefformat{rem}{Remark \ref{#1}}
\newrefformat{sect}{Section \ref{#1}}
\newrefformat{prop}{Proposition \ref{#1}}
\newrefformat{thm}{Theorem \ref{#1}}
\newrefformat{cor}{Corollary \ref{#1}}
\newrefformat{ex}{Example \ref{#1}}

\swapnumbers
\newtheoremstyle{dotless}{}{}{\itshape}{}{\bfseries}{}{}{}
\theoremstyle{dotless}
\theoremstyle{plain}
\newtheorem{thm}{Theorem}[section]

\newtheorem{prop}[thm]{Proposition}
\newtheorem{cor}[thm]{Corollary}
\theoremstyle{definition}
\newtheorem{defn}[thm]{Definition}
\newtheorem{rem}[thm]{Remark}
\newtheorem{exa}[thm]{Example}

\newcommand{\N} {\mathbb{N}}
\newcommand{\Z} {\mathbb{Z}}

\newcommand{\R} {\mathbb{R}}
\newcommand{\C} {\mathbb{C}}
\newcommand{\D} {\mathbb{D}}
\newcommand{\e}{\mathrm{e}}
\newcommand{\p}{\operatorname{p}}
\newcommand{\ob}{\operatorname{b}}
\newcommand{\apr}{\operatorname{a}}
\newcommand{\ap}{\operatorname{ap}}
\newcommand{\res}{\operatorname{r}}
\newcommand{\tp}{\operatorname{t}}
\newcommand{\alg}{\operatorname{alg}}
\newcommand{\bap}{\operatorname{bap}}
\newcommand{\seq}{\operatorname{seq}}
\newcommand{\s}{\operatorname{s}}
\newcommand{\ev}{\operatorname{ev}}
\newcommand{\cd}{\mathord{\,\cdot\,}} 
\newcommand{\cL} {\mathcal{L}}
\newcommand{\ran}{\operatorname{ran}}

\DeclareMathOperator{\id}{id}
\DeclareMathOperator{\re}{Re}
\providecommand{\differential}{\mathrm{d}}
\renewcommand{\d}{\differential}
\newcommand\rlim{
\mathchoice{\vcenter{\hbox{${\scriptstyle{+}}$}}}
{\vcenter{\hbox{$\scriptstyle{+}$}}}
{\vcenter{\hbox{$\scriptscriptstyle{+}$}}}
{\vcenter{\hbox{$\scriptscriptstyle{+}$}}}}

\makeatletter
\newcommand{\fakephantomsection}{%
  \Hy@GlobalStepCount\Hy@linkcounter%
  \Hy@MakeCurrentHref{\@currenvir.\the\Hy@linkcounter}
  \Hy@raisedlink{\hyper@anchorstart{\@currentHref}\hyper@anchorend}%
}
\makeatother

\begin{document}

\title[Spectral theory for semigroups on locally convex spaces]{Spectral theory for semigroups on locally convex spaces}
\author[K.~Kruse]{Karsten Kruse\,\orcidlink{0000-0003-1864-4915}}
\address{University of Twente, Department of Applied Mathematics, P.O.~Box 217, 7500 AE Enschede, The Netherlands}
\email{k.kruse@utwente.nl}



\date{\today}
\begin{abstract}
In this paper we provide spectral inclusion and mapping theorems for strongly continuous locally 
equicontinuous semigroups on Hausdorff locally convex spaces. Our results extend the classical spectral inclusion 
and mapping theorems for strongly continuous semigroups on Banach spaces.
\end{abstract}
\maketitle

\section{Introduction}

The spectral theory for strongly continuous semigroups on Banach spaces is well developed, 
in particular spectral inclusion and mapping theorems are available, 
see e.g.~\cite[Chap.~IV]{engel_nagel2000}, \cite[Chap.~XVI]{hillephillips1957}, \cite[Chap.~2.2]{pazy1983} and 
\cite[Chap.~2]{vanneerven1996}. For a strongly continuous semigroup $(T(t))_{t\geq 0}$ with generator 
$(A,D(A))$ on a Banach space $X$ an identity like 
\begin{equation}\label{eq:full_spec_map}
\sigma(T(t))\setminus\{0\}=\e^{t\sigma(A)},\quad t\geq 0,
\end{equation}
is called a \emph{spectral mapping theorem} where $\sigma(\cd)$ in \eqref{eq:full_spec_map} denotes the spectrum of the 
corresponding operator. Looking at the abstract Cauchy problem
\[
u'(t)=Au(t),\quad t\geq 0,\quad u(0)=x_0\in X,
\]
which has the function $t\mapsto T(t)x_0$ as a (mild) solution, the spectral mapping theorem connects the spectral behaviour 
of the generator $A$ and of the solution of the abstract Cauchy problem induced by the semigroup $(T(t))_{t\geq 0}$. 
This is important since in concrete problems we often have a good characterisation of the generator but 
the semigroup is not explicitly available. 
For instance, this allows us to study the long-term, or asymptotic, behaviour of the 
non-explicit solution of the abstract Cauchy problem by studying the properties of the generator, 
see e.g.~\cite[Chap.~V]{engel_nagel2000}, in particular \cite[Chap.~V, 1.7 Proposition, 1.9 Lemma, p.~299--301]{engel_nagel2000}, 
and \prettyref{rem:spectral_bound} \ref{it:spectral_bound_2}. 

It is well-known that the spectral mapping theorem \eqref{eq:full_spec_map} does not hold in general, 
see \cite[p.~270--275]{engel_nagel2000}. However, it holds for eventually uniformly continuous semigroups. 
Further, if we replace the spectrum in \eqref{eq:full_spec_map} by the point or residual spectrum, then this adjusted 
spectral mapping theorem holds for all strongly continuous semigroups on Banach spaces.  

If we want to go beyond the realm of Banach spaces, a natural more general setting is to consider strongly continuous 
semigroups on Hausdorff locally convex spaces. The general theory of strongly continuous semigroups on such spaces 
is rather well developed, see 
e.g.~\cite{albanese2016b,babalola1974,choe1985,dembart1974,frerick2014,jacob2015,komatsu1964,
komura1968,kraaij2016,kruse_schwenninger2025,miyadera1959,moore1971a,moore1971b,ouchi1973,wegner2014,yosida1968}.
More recently, also spectral theory for closed linear operators and strongly continuous semigroups 
on Hausdorff locally convex spaces $X$ has gotten more attention, see for example 
\cite{albanese2010a,albanese2010b,albanese2012,albanese2013,albanese2014,albanese2016a,albanese2016c,wegner2016}, 
even though according to \cite[p.~254]{albanese2013}, for ``$X$ non-normable, the spectral theory of 
closed operators $A$ is much less developed.'' (cf.~\cite[p.~922]{albanese2016b}). 
In that regard, the purpose of our paper is to advance the spectral theory of strongly continuous semigroups 
on Hausdorff locally convex spaces by providing spectral inclusion and spectral mapping theorems, 
which was to the best of our knowledge not done before. To do so, we deeply analyse the corresponding proofs 
in the case of Banach spaces given in \cite{engel_nagel2000} and \cite{vanneerven1996} and modify them to our needs.  

Let us outline the content of our paper. In \prettyref{sect:notions} we recall some notions and results 
related to linear operators and semigroups on Hausdorff locally convex spaces. 
\prettyref{sect:spec_lin_op} is dedicated to different types of spectra of linear operators and how they are related. 
Then we turn to periodic semigroups in \prettyref{sect:periodic_sg} and analyse their spectral behaviour. 
In our final and main section we prove spectral inclusion theorems in \prettyref{thm:spec_incl} 
for the different types of spectra from \prettyref{sect:spec_lin_op}. 
Using our results on periodic semigroups, we show in \prettyref{thm:spec_point} 
that the spectral mapping theorem for the point spectrum of strongly continuous locally equicontinuous semigroups 
$(T(t))_{t\geq 0}$ with generator $(A,D(A))$ on sequentially complete Hausdorff locally convex spaces $X$ holds. Further, we describe the relation between the eigenspaces of $A$ and $T(t)$ in 
\prettyref{thm:spec_point}. Then we turn our attention to the residual spectrum and prove in \prettyref{thm:spec_res}
that the spectral mapping theorem for the residual spectrum of strongly continuous locally equicontinuous 
semigroups holds under some completeness assumptions on $X$, at least if the (algebraic) resolvent set of $A$ 
is non-empty. Finally, we focus on the bounded (sequential) approximate point spectrum. 
We show that for a strongly continuous locally equicontinuous semigroup $(T(t))_{t\geq 0}$ on 
a sequentially complete Hausdorff locally convex space the spectral mapping theorems for those spectra 
hold if $X$ is a generalised Schwartz space (see \prettyref{cor:spec_bounded_ap}) or the semigroup is 
eventually uniformly continuous (see \prettyref{cor:spec_bounded_seq_ap}). 

\section{Notions and preliminaries}
\label{sect:notions}

For a Hausdorff locally convex space $(X,\tau_X)$ we always denote by $\Gamma_{X}$ a fundamental system of seminorms 
that induces the Hausdorff locally convex topology $\tau_X$ on $X$. Further, all Hausdorff locally convex spaces that 
we consider have the complex numbers $\C$ as their scalar field, and if no confusion seems to be likely, 
we just write $X$ instead of $(X,\tau_X)$. We denote by $\cL(X)$ the space of continuous linear operators from $X$ to $X$, 
and by $X'$ the topological dual space of $X$. We write $\cL_{\s}(X)$ for the space $\cL(X)$ equipped with the topology of uniform convergence on finite subsets of $X$, and $\cL_{\ob}(X)$ for the space $\cL(X)$ equipped with the topology of 
uniform convergence on bounded subsets of $X$. On $X'$ we denote the corresponding topologies by $\sigma(X',X)$ and 
$\beta(X',X)$, respectively. For other unexplained notions on the theory of Hausdorff locally convex spaces we refer the reader 
to \cite{jarchow1981,kaballo2014,meisevogt1997,bonet1987}.

We write in short that $(A,D(A))$ is a linear operator on a linear space $X$ over the scalar field $\C$ 
if $D(A)$ is a linear subspace of $X$ and $A\colon D(A)\to X$ a linear operator. 
For a linear operator $(A,D(A))$ on $X$ and $\lambda\in\C$ we 
write $\lambda-A\coloneqq \lambda\id -A$ where $\id$ is the identity map on $X$.

\begin{defn}[{\cite[p.~258]{albanese2013}, \cite[Definition 3.5, p.~923]{albanese2016b}, 
\cite[Chap.~IV, Definition, p.~60]{engel_nagel2000}}]
Let $X$ be a Hausdorff locally convex space and $(A,D(A))$ a linear operator on $X$.
\begin{enumerate}[label=\upshape(\alph*), leftmargin=*]
\item $(A,D(A))$ is called \emph{closed} if for each net $(x_{i})_{i\in I}\subseteq D(A)$ 
satisfying $x_{i}\to x$ and $Ax_{i}\to y$ for some $x,y\in X$, we have $x\in D(A)$ and $Ax=y$. 
If $D(A)=X$, then we just write that $A$ is closed instead of $(A,X)$ closed.
\item $(A,D(A))$ is called \emph{sequentially closed} if for each sequence $(x_{i})_{i\in \N}\subseteq D(A)$ 
satisfying $x_{i}\to x$ and $Ax_{i}\to y$ for some $x,y\in X$, we have $x\in D(A)$ and $Ax=y$. 
\item $(A,D(A))$ is called \emph{densely defined} if $D(A)$ is dense in $X$.
\item Let $(A,D(A))$ be densely defined. The \emph{dual operator} $(A',D(A'))$ of $(A,D(A))$ on $X'$ 
is defined by setting
\[
D(A')\coloneqq\{x'\in X'\;|\;\exists\;y'\in X'\;\forall\;x\in D(A):\;\langle x', Ax\rangle=\langle y', x \rangle\}
\]
and $A'x'\coloneqq y'$ for $x'\in D(A')$.
\item Let $Y$ be a linear subspace of $X$. The \emph{part} $(A_{\mid Y},D(A_{\mid Y}))$ of $(A,D(A))$ in $Y$ 
is defined by $A_{\mid Y}y\coloneqq Ay$ for $y\in D(A_{\mid Y})$ with
\[
D(A_{\mid Y})\coloneqq\{y\in D(A)\cap Y\;|\;Ay\in Y\}.
\]
\end{enumerate}
\end{defn}

Next, we recall some notions in the context of semigroups. 

\begin{defn}[{\cite[p.~143]{albanese2010a}, \cite[p.~294]{choe1985}, \cite[Definition 1.1, p.~259]{komura1968}}]\label{defn:semigroup}
Let $X$ be a Hausdorff locally convex space. A family $(T(t))_{t\geq 0}$ in $\mathcal{L}(X)$ is called
\begin{enumerate}[label=\upshape(\roman*), leftmargin=*, widest=iii]
\item a \emph{semigroup} on $X$ if $T(t+s)=T(t)T(s)$ and $T(0)=\id$ for all $t,s\geq 0$,
\item \emph{strongly continuous} if the map $[0,\infty)\to\cL_{\s}(X)$, $t\mapsto T(t)$, is continuous, 
\item \emph{eventually uniformly continuous} on $X$ if $(T(t))_{t\geq 0}$ is strongly continuous and 
there is $t_{\ev}\geq 0$ such that the map $[t_{\ev},\infty)\to\cL_{\ob}(X)$, $t\mapsto T(t)$, is continuous. 
If $t_{\ev}=0$, then $(T(t))_{t\geq 0}$ is called \emph{uniformly continuous}.
\item \emph{locally equicontinuous} if for a fundamental system of seminorms $\Gamma_{X}$ it holds
\[
\forall\;q\in\Gamma_{X},\,t_{0}\geq 0\;\exists\;p\in\Gamma_{X},\,C\geq 0\;
\forall\;t\in [0,t_{0}],\,x\in X:\;q(T(t)x)\leq Cp(x),
\] 
\item \label{it:quasi-equi} \emph{quasi-equicontinuous} if for a fundamental system of seminorms $\Gamma_{X}$ it holds
\[
\exists\;\omega\in\R\;\forall\;q\in\Gamma_{X}\;\exists\;p\in\Gamma_{X},\,C\geq 0\;
\forall\;t\geq 0,\,x\in X:\;q(\mathrm{e}^{-\omega t}T(t)x)\leq Cp(x). 
\] 
If $\omega=0$, then $(T(t))_{t\geq 0}$ is called \emph{equicontinuous}.
\end{enumerate}
\end{defn}

In the case that $X$ is a Banach space the definition of eventual uniform continuity is for example given 
in \cite[p.~35]{vanneerven1992}. 
We recall some observations from \cite[p.~48--49, 79]{kruse_schwenninger2025} regarding the notions in \prettyref{defn:semigroup}. 
We note that the definitions of local equicontinuity and quasi-equicontinuity do not depend on the choice of 
$\Gamma_{X}$. Clearly, quasi-equicontinuity, which is sometimes also called \emph{exponential equicontinuity} 
(see \cite[Definition 2.1, p.~255--256]{albanese2013}), implies local equicontinuity. Moreover, some results 
on automatic local equicontinuity are known. For instance, every strongly continuous semigroup on a barrelled or strong 
Mackey space $X$ is locally equicontinuous by \cite[Proposition 1.1, p.~259]{komura1968} and 
\cite[Lemma 3.2, p.~160]{kraaij2016}. Hence on Fr\'echet spaces every strongly continuous 
semigroup is already locally equicontinuous but there exist strongly continuous semigroups on Fr\'echet spaces which are 
not quasi-equicontinuous by \cite[Remark 2.2 (iii), p.~256]{albanese2013}. 
In contrast, on Banach spaces every strongly continuous semigroup is already quasi-equicontinuous 
by \cite[Chap.~I, 5.5 Proposition, p.~39]{engel_nagel2000}. 
The same is true for so-called bi-continuous semigroups on certain sequentially complete Saks spaces 
w.r.t.~mixed topology by \cite[Theorem 7.4, p.~180]{kraaij2016} (cf.~\cite[Theorem 3.17 (a), p.~13]{kruse_schwenninger2022})
A \emph{Saks space} is a triple $(X,\|\cd\|,\tau)$ where $(X,\|\cd\|)$ is a normed space, 
$\tau$ is a Hausdorff locally convex topology which is coarser than the $\|\cd\|$-topology $\tau_{\|\cd\|}$ and fulfils that 
$\{x\in X\;|\;\|x\|\leq 1\}$ is $\tau$-closed 
(see \cite[I.3.2 Definition, p.~27--28]{cooper1978} and \cite[Section 2.1]{wiweger1961}). 
The mixed topology $\gamma\coloneqq \gamma(\|\cd\|,\tau)$ is then the finest linear topology on $X$ 
that coincides with $\tau$ on $\|\cd\|$-bounded sets and such that $\tau\leq \gamma \leq \tau_{\|\cd\|}$. 
The mixed topology $\gamma$ is Hausdorff locally convex and 
the definition given here is equivalent to the one from the literature \cite[Section 2.1]{wiweger1961} 
due to \cite[Lemmas 2.2.1, 2.2.2, p.~51]{wiweger1961}. 

Moreover, we recall from \cite[p.~260]{komura1968} that the \emph{generator} $(A,D(A))$ 
of a strongly continuous semigroup $(T(t))_{t\geq 0}$ on a Hausdorff locally convex space $X$ is defined by 
\[
D(A)\coloneqq \Bigl\{x\in X\;|\;\lim_{t\to 0\rlim}\frac{T(t)x-x}{t}\;\text{exists in }X\Bigr\}
\] 
and 
\[
Ax\coloneqq \lim_{t\to 0\rlim}\frac{T(t)x-x}{t},\quad x\in D(A).
\]
If $X$ is sequentially complete, then $D(A)$ is dense in $X$ 
by \cite[Proposition 1.3, p.~261]{komura1968}, so $(A,D(A))$ is densely defined in this case. 
If $(T(t))_{t\geq 0}$ is locally equicontinuous, then the generator $(A,D(A))$ is closed by 
\cite[Proposition 1.4, p.~262]{komura1968}. 

\begin{prop}[{\cite[Lemma 1, p.~450]{wegner2014}}]\label{prop:rescale_sg}
Let $\lambda\in\C$, $c>0$, $X$ be a Hausdorff locally convex space and $(T(t))_{t\geq 0}$ a strongly continuous semigroup on 
$X$ with generator $(A,D(A))$. Then the family $(S(t))_{t\geq 0}$ defined by $S(t)\coloneqq \e^{-\lambda t}T(ct)$, 
$t\geq 0$, is a strongly continuous semigroup on $X$ with generator $(B,D(B))$ where $B=cA-\lambda$ and $D(B)=D(A)$. 
In addition, if $(T(t))_{t\geq 0}$ is locally (or quasi-)\linebreak[0]{}equicontinuous, 
then so is $(S(t))_{t\geq 0}$.
\end{prop}

$(S(t))_{t\geq 0}$ is called a \emph{rescaled semigroup}. 
\prettyref{prop:rescale_sg} is stated in \cite{wegner2014} for $c=1$, complete Hausdorff locally convex spaces $X$ and 
strongly continuous quasi-equicontinuous semigroups $(T(t))_{t\geq 0}$. However, looking at the proof of 
\cite[Lemma 1, p.~450]{wegner2014} it is easily adjustable to the case $c>0$, the assumption of completeness 
is not needed and it also holds for locally equicontinuous $(T(t))_{t\geq 0}$ with the only difference 
that the rescaled semigroup $(S(t))_{t\geq 0}$ is then also only locally equicontinuous. 

The next two identities generalise \cite[Chap.~II, 1.9 Lemma, p.~55]{engel_nagel2000} from Banach spaces to 
sequentially complete Hausdorff locally convex spaces. 

\begin{prop}\label{prop:rescale_identities}
Let $X$ be a sequentially complete Hausdorff locally convex space and $(T(t))_{t\geq 0}$ a strongly continuous 
semigroup on $X$ with generator $(A,D(A))$. Then for all $\lambda\in\C$, $t\geq 0$ and $x\in X$ we have 
$\int_{0}^{t}\e^{-\lambda s}T(s)x \d s\in D(A)$ and the following identities hold
\begin{align}
  \e^{-\lambda t}T(t)x-x
&=(A-\lambda)\int_{0}^{t}\e^{-\lambda s}T(s)x \d s &&\text{if }x\in X, \label{eq:rescale_iden_1}\\
&=\int_{0}^{t}\e^{-\lambda s}T(s)(A-\lambda)x \d s &&\text{if }x\in D(A), \label{eq:rescale_iden_2}
\end{align}
where the integrals above are Riemann integrals.
\end{prop}
\begin{proof}
Let $\lambda\in\C$ and $x\in X$. The claim that $\int_{0}^{t}\e^{-\lambda s}T(s)x \d s\in D(A)$ 
and the two identities follow from \cite[Corollary, p.~261]{komura1968} and 
\cite[Proposition 1.2 (2), p.~260]{komura1968} applied to the rescaled semigroup 
$(S(t))_{t\geq 0}$ from \prettyref{prop:rescale_sg} given by $S(t)\coloneqq \e^{-\lambda t}T(t)$, $t\geq 0$.
\end{proof}

Let $X$ be a Hausdorff locally convex space. We call a Hausdorff locally convex space $Y$ \emph{continuously embedded} 
(in $X$) if there is an injective continuous linear map $j\colon Y\to X$. 
In this case we write $Y\hookrightarrow X$ for short. Let $(T(t))_{t\geq 0}$ be a strongly continuous semigroup on 
$X$. We call a continuously embedded space $Y$ \emph{$(T(t))_{t\geq 0}$-invariant} if $T(t)j(y)\in j(Y)$ 
for all $t\geq 0$ (see e.g.~\cite[p.~43]{engel_nagel2000} in the case that $X$ is a Banach space, 
$Y$ a closed subspace and the semigroup strongly continuous). Further, we usually omit the map $j$ and just write 
$T(t)_{\mid Y}y\coloneqq T(t)y\coloneqq T(t)j(y)$ for all $y\in Y$ and $t\geq 0$ in such a case. 
The family $(T(t)_{\mid Y})_{t\geq 0}$ is then a semigroup on $Y$, which we call the \emph{restricted semigroup}, 
but it might not be strongly continuous w.r.t.~the topology of $Y$. 
Our following result generalises \cite[Chap.~II, Proposition, Corollary, p.~60--61]{engel_nagel2000}.

\begin{prop}\label{prop:restricted_sg}
Let $X$ be a Hausdorff locally convex space and $(T(t))_{t\geq 0}$ a strongly continuous 
semigroup on $X$ with generator $(A,D(A))$. Let $Y\hookrightarrow X$ be a $(T(t))_{t\geq 0}$-invariant 
Hausdorff locally convex space. Then the following assertions hold.
\begin{enumerate}[label=\upshape(\alph*), leftmargin=*]
\item \label{it:rest_sg_0} If $Y$ is sequentially complete and $(T(t)_{\mid Y})_{t\geq 0}$ 
a strongly continuous semigroup on $Y$, then the generator of $(T(t)_{\mid Y})_{t\geq 0}$ 
is the part $(A_{\mid Y},D(A_{\mid Y}))$ of $(A,D(A))$ in 
$Y$.
\end{enumerate}
Suppose for \ref{it:rest_sg_1}--\ref{it:rest_sg_3} that $Y$ is a topological subspace of $X$, i.e.~the embedding $Y\hookrightarrow X$ is a topological isomorphism to its range.
\begin{enumerate}[label=\upshape(\alph*), leftmargin=*]\setcounter{enumi}{1}
\item \label{it:rest_sg_1} $(T(t)_{\mid Y})_{t\geq 0}$ is a strongly continuous semigroup on $Y$. 
\item \label{it:rest_sg_2} If $(T(t))_{t\geq 0}$ is (locally, quasi-) equicontinuous, then $(T(t)_{\mid Y})_{t\geq 0}$ 
is (locally, quasi-) equicontinuous.
\item \label{it:rest_sg_3} If $X$ is sequentially complete and $Y$ sequentially closed, 
then the generator of $(T(t)_{\mid Y})_{t\geq 0}$ is the part $(A_{\mid Y},D(A_{\mid Y}))$ of $(A,D(A))$ in $Y$
and its domain fulfils
$
D(A_{\mid Y})=D(A)\cap Y.
$
\end{enumerate}
\end{prop}
\begin{proof}
\ref{it:rest_sg_0} Let $(C,D(C))$ be the generator of $(T(t)_{\mid Y})_{t\geq 0}$. If $y\in D(C)\subseteq Y$, 
then 
\[
 Y\ni Cy
=\lim_{t\to 0\rlim}\frac{T(t)_{\mid Y}y-y}{t}
=\lim_{t\to 0\rlim}\frac{T(t)y-y}{t}
=Ay
\]
which yields $D(C)\subseteq (D(A)\cap Y)$ and $D(C)\subseteq D(A_{\mid Y})$. 

Now, we turn to the converse inclusion. Let $y \in D(A_{\mid Y})$. Then $Ay\in Y$ and 
we note that $\int_{0}^{t}T(s)_{\mid Y}Ay\d s\in Y$ for all $t\geq 0$ by \cite[Theorem 10, p.~317]{albanese2012} (cf.~\cite[Proposition 1.1, p.~232]{komatsu1964}) since $Y$ is sequentially complete and 
$(T(s)_{\mid Y})_{s\geq 0}$ strongly continuous on $Y$. Furthermore, the map 
$f\colon [0,\infty) \to X$, $f(s)\coloneqq T(s)Ay$, is continuous as $(T(t))_{t\geq 0}$ is strongly continuous 
on $X$. Therefore the Riemann integral of $f$ on $[0,t]$ exists for all $t\geq 0$ 
in the completion $(\widehat{X},\widehat{\tau}_{X})$ of $X$ by \cite[Theorem 10, p.~317]{albanese2012} again 
where $\widehat{\tau}_{X}$ denotes the Hausdorff locally convex topology on $\widehat{X}$. 
We write $\widehat{\tau}_{X}\text{-}\int_{0}^{t}T(s)Ay\d s$ for this integral and observe that 
\[
 \int_{0}^{t}T(s)_{\mid Y}Ay\d s
=\widehat{\tau}_{X}\text{-}\int_{0}^{t}T(s)Ay\d s
=T(t)y-y
=T(t)_{\mid Y}y-y
\]
for all $t\geq 0$ by \cite[Proposition 1.2 (2), p.~260]{komura1968} (applied to $(T(s))_{s\geq 0}$ on $X$) . 
Now, \cite[Proposition 1.2 (2), p.~260]{komura1968} (applied to $(T(s)_{\mid Y})_{s\geq 0}$ on $Y$) 
implies that $y\in D(C)$ and $Cy=Ay$. 

\ref{it:rest_sg_1} and \ref{it:rest_sg_2} are obvious. 

\ref{it:rest_sg_3} We note that $Y$ is sequentially complete as a sequentially closed subspace of $X$. 
Therefore $(A_{\mid Y},D(A_{\mid Y}))$ is the generator of $(T(t)_{\mid Y})_{t\geq 0}$ by parts \ref{it:rest_sg_0} 
and \ref{it:rest_sg_1}.
Further, we have $D(A_{\mid Y})\subseteq (D(A)\cap Y)$ by definition. Let $y\in D(A)\cap Y$. 
Then $T(t)y\in Y$ for all $t\geq 0$ and 
\[
\lim_{t\to 0\rlim}\frac{T(t)y-y}{t}=Ay\in X, 
\]
which implies $Ay\in Y$ as $Y$ is sequentially closed in $X$. Hence we have $y\in D(A_{\mid Y})$ and so 
$(D(A)\cap Y)\subseteq D(A_{\mid Y})$.
\end{proof}

Now, we turn to a special case in the setting of dual semigroups. 
Let $X$ be a Hausdorff locally convex space and $(T(t))_{t\geq 0}$ a strongly continuous semigroup on $X$ 
with generator $(A,D(A))$. 
Then the family $(T'(t))_{t\geq 0}$ in $\cL(X_{\s}')\subseteq \cL(X_{\ob}')$ 
defined by $T'(t)\coloneqq T(t)'$ for $t\geq 0$ is a $\sigma(X',X)$-strongly 
continuous semigroup on $X'$. The family $(T'(t))_{t\geq 0}$ is called the \emph{dual semigroup} of 
$(T(t))_{t\geq 0}$. If $X$ is sequentially complete, then $(A,D(A))$ is densely defined and $(A',D(A'))$ is 
the generator of $(T'(t))_{t\geq 0}$ by \cite[Proposition 2.1, p.~263]{komura1968}. 
In general, $(T'(t))_{t\geq 0}$ might not be $\beta(X',X)$-strongly continuous and we define 
\[
X^{\odot}\coloneqq\{x'\in X'\;|\; T'(\cdot)x'\colon [0,\infty)\to X' \text{ is }\beta(X',X)\text{-continuous}\}.
\]
If $X$ is a Banach space, then $X^{\odot}$ is called the \emph{sun dual} (see \cite[p.~5]{vanneerven1992} and \prettyref{rem:sun_dual}).
Related to \prettyref{prop:restricted_sg} \ref{it:rest_sg_0} with $Y\coloneqq (X^{\odot},\beta(X',X))$ 
we recall the following result, which we need later on.

\begin{thm}[{\cite[Theorem 1, p.~263]{komura1968}}]\label{thm:dual_str_cont}
Let $X$ be a sequentially complete Hausdorff locally convex space such that $X_{\ob}'$ is sequentially complete 
and $(T(t))_{\geq 0}$ a strongly continuous semigroup on $X$ with generator $(A,D(A))$. 
Then it holds that $X^{\odot}=\overline{D(A')}^{\beta(X',X)}$, in particular $X^{\odot}$ is $\beta(X',X)$-closed, 
and it is a $(T'(t))_{t\geq 0}$-invariant linear subspace of $X'$. 
Moreover, the restricted semigroup $(T^{\odot}(t))_{t\geq 0}\coloneqq (T'(t)_{\mid X^{\odot}})_{t\geq 0}$ is 
$\beta(X',X)$-strongly continuous on $X^{\odot}$. Its generator $(A^{\odot},D(A^{\odot}))$ coincides with the part 
$(A_{\mid X^{\odot}}',D(A_{\mid X^{\odot}}'))$ of $(A',D(A'))$ in $X^{\odot}$ and its domain fulfils
\[
D(A^{\odot})=\{x'\in D(A')\;|\;A'x'\in X^{\odot}\}.
\]
If $(T(t))_{t\geq 0}$ is (locally, quasi-) equicontinuous, then $(T^{\odot}(t))_{t\geq 0}$ is 
(locally, quasi-) $\beta(X',X)$-equicontinuous.
\end{thm}

\begin{rem}\label{rem:sun_dual}
Let $X$ be a Hausdorff locally convex space and $(T(t))_{t\geq 0}$ a strongly continuous semigroup on $X$. Then we clearly have
\[
X^{\odot}\subseteq\{x'\in X'\;|\;\beta(X',X)\text{-}\lim_{t\to 0\rlim}T'(t)x'-x'=0\}\eqqcolon X_0.
\]
If in addition $(T(t))_{t\geq 0}$ is locally equicontinuous, then $X^{\odot}=X_0$. 
Indeed, the local equicontinuity of $(T(t))_{t\geq 0}$ implies that $(T'(t))_{t\geq 0}$ is locally $\beta(X',X)$-equicontin\-uous 
and so the inclusion $X^{\odot}\subseteq X_0$ follows from the proof of \cite[Remark 1 (iii), p.~302]{albanese2012}.
\end{rem}

\section{Spectra of linear operators}
\label{sect:spec_lin_op}

In this section we introduce different notions of spectra of a linear operator $(A,D(A))$ and present some results 
about these spectra and their relations.

\begin{defn}[{\cite[p.~258]{albanese2013}, \cite[p.~269]{albanese2016a}}]\label{defn:spectrum}
Let $X$ be a Hausdorff locally convex space and $(A,D(A))$ a linear operator on $X$. 
If $\lambda\in\C$ is such that $\lambda-A\colon D(A)\to X$ is injective, then the linear operator 
$(\lambda-A)^{-1}$ exists and is defined on the domain $\ran(\lambda-A)\coloneqq \{(\lambda-A)x\;|\;x\in D(A)\}$, 
i.e.~the range of $\lambda-A$. The \emph{resolvent set} of $A$ is defined by 
\[
\rho(A)\coloneqq\{\lambda\in\C\;|\;\lambda-A\text{ is bijective and }(\lambda-A)^{-1}\in\cL(X)\}.
\]
If $\lambda\in\rho(A)$, we write $R(\lambda,A)\coloneqq (\lambda-A)^{-1}$ and call it the \emph{resolvent} of $A$ 
in $\lambda$. Further, we call $\sigma(A)\coloneqq\C\setminus\rho(A)$ the \emph{spectrum} of $A$. 
Moreover, we define the subset $\rho^{\ast}(A)\subseteq\rho(A)$ consisting of all $\lambda\in\rho(A)$ such that 
there is $\delta>0$ which fulfils $B(\lambda,\delta)\coloneqq\{\mu\in\C\;|\;|\mu-\lambda|<\delta\}\subseteq\rho(A)$ 
and that the set $\{R(\mu,A)\;|\;\mu\in B(\lambda,\delta)\}$ is equicontinuous in $\cL(X)$. 
In addition, we write $\sigma^{\ast}(A)\coloneqq \C\setminus \rho^{\ast}(A)$.
\end{defn}

If $(A,D(A))$ is a linear operator on a Hausdorff locally convex space $X$ such that $\rho(A)\neq\varnothing$, 
then $(A,D(A))$ is already closed by \cite[Remark 3.1 (i), p.~259]{albanese2013}. 
An example of a closed linear operator on a Banach space $X$ such that $\rho(A)=\varnothing$ is given in 
\cite[Chap.~IV, 1.5 Examples (i), p.~241]{engel_nagel2000}.
If $(A,D(A))$ generates a strongly continuous quasi-equicontinuous semigroup on a sequentially complete 
Hausdorff locally convex space $X$, then $\rho(A)\neq\varnothing$ by \cite[Corollary 4.5, p.~307]{choe1985}, 
more precisely there is $a\geq 0$ such that $\{\lambda\in\C\;|\;\re(\lambda)>a\}\subseteq\rho(A)$. 
In general, it might happen in contrast to the situation on Banach spaces 
(see \cite[Chap.~IV, 1.3 Proposition (i), p.~240]{engel_nagel2000}) that $\rho(A)$ 
is not an open subset of $\C$ even if $(A,D(A))$ generates a strongly continuous equicontinuous semigroup 
on a Fr\'echet space (see \cite[Remark 3.5 (vii), p.~265--266]{albanese2013}). On the other hand, 
$\rho^{\ast}(A)$ is an open set by definition for any linear operator $(A,D(A))$ on a Hausdorff locally convex space 
$X$, and if $\rho^{\ast}(A)\neq \varnothing$ and $X$ is sequentially complete, 
then $R(\cd,A)\colon\rho^{\ast}(A)\to\cL_{\ob}(X)$ is holomorphic by \cite[Proposition 3.4 (i), p.~260]{albanese2013}. 
In \cite[Remark 3.5 (vi), p.~264--265]{albanese2013} an example of an operator $A\in\cL(X)$ on a Fr\'echet space $X$ 
is given with a strict inclusion $\overline{\sigma(A)}\subset \sigma^{\ast}(A)$ (cf.~\cite[p.~269]{albanese2016a}). 
Whereas, if $(A,D(A))$ is a linear operator on a Banach space $X$ such that $\rho(A)\neq\varnothing$, 
then $\sigma^{\ast}(A)=\sigma(A)$ by \cite[Remark 3.5 (iii), p.~262]{albanese2013}.

\prettyref{defn:spectrum} is not the only way to generalise the notions of the resolvent (set) and the spectrum 
to the locally convex setting, see for instance \cite[Definition 3.1, p.~804--805]{wegner2016} and the discussion 
of the different types of definitions and relations there. 

\begin{rem}\label{rem:alg_resolvent}
Let $X$ be a linear space and $(A,D(A))$ a linear operator on $X$ and define the 
\emph{algebraic resolvent set} of $A$ by
\[
\rho_{\alg}(A)\coloneqq \{\lambda\in\C\;|\;\lambda-A\text{ is bijective}\}
\]
and the \emph{algebraic spectrum} of $A$ by $\sigma_{\alg}(A)\coloneqq\C\setminus \rho_{\alg}(A)$.
If $X$ is a Hausdorff locally convex space, $(A,D(A))$ is closed and 
\begin{enumerate}[label=\upshape(\roman*), leftmargin=*, widest=iii]
\item \label{it:cl_gr_1} $X$ is ultrabornological and webbed, or
\item \label{it:cl_gr_2} $X$ is barrelled and $B_{r}$-complete, or
\item \label{it:cl_gr_3} $X$ is a Mackey $L_r$-space such that $X'$ is weakly sequentially complete, or
\item \label{it:cl_gr_4} $X$ is a semireflexive Mackey gDF space, or
\item \label{it:cl_gr_5} $X$ is a semi-Montel space and the topology on $X$ coincides with a mixed topology 
$\gamma\coloneqq\gamma(\|\cd\|,\tau)$ for some Saks space $(X,\|\cd\|,\tau)$,
\end{enumerate}
then $\rho(A)=\rho_{\alg}(A)$ and  $\sigma(A)=\sigma_{\alg}(A)$. 
Indeed, we only need to prove that $\rho_{\alg}(A)\subseteq\rho(A)$ which 
follows in the listed cases from closed graph theorems. Let $\lambda\in\rho_{\alg}(A)$. 
Then $(\lambda-A)^{-1}$ is closed by \cite[p.~258]{albanese2013} and the statement follows 
in case \ref{it:cl_gr_1} from \cite[Closed graph theorem 24.31, p.~289]{meisevogt1997},
in case \ref{it:cl_gr_2} from \cite[11.1.7 Theorem (c), p.~221]{jarchow1981}, 
in case \ref{it:cl_gr_3} from \cite[Theorem 1, p.~390]{qiu1985} (and its correction 
\cite[Proposition 3.1, p.~17]{boos1993}), 
in case \ref{it:cl_gr_4} from \cite[Theorem 1 (vii), p.~398]{mcintosh1969},
\cite[12.4.2 Theorem, p.~258]{jarchow1981} and the fact that semireflexive spaces are quasi-complete, 
and in case \ref{it:cl_gr_5} from \cite[I.4.32 Proposition, p.~60]{cooper1978} 
and the fact that $(X,\gamma)$ is complete by \cite[I.1.13, I.1.14 Propositions, p.~11]{cooper1978}.
\end{rem}

In the special case of \prettyref{rem:alg_resolvent} \ref{it:cl_gr_1} that $X$ is a Fr\'echet space
this is already observed in \cite[Remark 3.1 (ii), p.~259]{albanese2013}.

\begin{defn}
Let $X$ be a Hausdorff locally convex space and $(A,D(A))$ a linear operator on $X$. 
Then the \emph{point spectrum} of $A$ is defined by 
\begin{flalign*}
&\sigma_{\p}(A)\coloneqq\{\lambda\in\C\;|\;\lambda-A\text{ is not injective}\},&
\end{flalign*}
the \emph{approximate point spectrum} by 
\begin{flalign*}
&\sigma_{\ap}(A)\coloneqq\{\lambda\in\C\;|\;\exists\text{ a net }(x_i)_{i\in I}\subseteq D(A) \text{ not converging to }0:\;
\lim_{i\in I}(A-\lambda )x_i=0\},&
\end{flalign*}
the \emph{sequential approximate point spectrum} by 
\begin{flalign*}
&\sigma_{\ap}^{\seq}(A)\coloneqq\{\lambda\in\C\;|\;\exists\text{ a seq.}~(x_i)_{i\in \N}\subseteq D(A) \text{ not converging to }0:\;
\lim_{i\to\infty }(A-\lambda )x_i=0\},&
\end{flalign*}
the \emph{approximate spectrum} by 
\begin{flalign*}
&\sigma_{\apr}(A)\coloneqq\{\lambda\in\C\;|\;\lambda-A\text{ is not injective or }\ran(\lambda-A)
\text{ is not closed in }X\},&
\end{flalign*}
the \emph{residual spectrum}\footnote{One should be aware that there are different definitions of the residual spectrum 
in the literature. For instance, in \cite[VII.5.1 Exercises, p.~580]{dunford1958} the residual spectrum of a continuous linear operator $T\in\cL(X)$ is defined 
as the set $\{\lambda\in\C\;|\;\lambda-T\text{ is injective and }\ran(\lambda-T)\text{ is not dense in }X\}$ whereas what we call the residual spectrum is often named the compression spectrum (see e.g.~\cite[p.~28]{appell2004}). 
However, we stick here to the established notion of the residual spectrum from semigroup theory.} by 
\begin{flalign*}
&\sigma_{\res}(A)\coloneqq\{\lambda\in\C\;|\;\ran(\lambda-A)\text{ is not dense in }X\},&
\end{flalign*}
and the \emph{topological spectrum} by 
\begin{flalign*}
&\sigma_{\tp}(A)\coloneqq\{\lambda\in\C\;|\;\lambda-A\text{ is bijective and }(\lambda-A)^{-1}\notin\cL(X)\}
=\sigma(A)\setminus\sigma_{\alg}(A).&
\end{flalign*}
Further, the subset of all $\lambda\in\sigma_{\ap}(A)$ such that $(x_i)_{i\in I}$ 
can be chosen bounded is denoted by $\sigma_{\bap}(A)$ and called the 
\emph{bounded approximate point spectrum}. The \emph{bounded sequential approximate point spectrum} $\sigma_{\bap}^{\seq}(A)$ is defined analogously.
\end{defn}

In the case that $X$ is a Banach space and $(A,D(A))$ a closed linear operator 
the point spectrum, the approximate spectrum and the residual spectrum are given in 
\cite[Chap.~IV, 1.6, 1.8, 1.11 Definitions, p.~241--243]{engel_nagel2000}. Moreover, in this case it holds that 
\[
\sigma_{\ap}(A)=\sigma_{\ap}^{\seq}(A)=\sigma_{\bap}(A)=\sigma_{\bap}^{\seq}(A)=\sigma_{\apr}(A)
\]
by e.g.~\prettyref{prop:seq_ap_bounded_ap}.

\begin{rem}\label{rem:point_spec_included}
Let $X$ be a Hausdorff locally convex space, $(A,D(A))$ a linear operator on $X$ and $\lambda\in\C$.
\begin{enumerate}[label=\upshape(\alph*), leftmargin=*]
\item \label{it:p_spec_included_1} Then $\lambda\in\sigma_{\p}(A)$ if and only if $\lambda$ is an \emph{eigenvalue} of $A$, 
i.e.~there is $x\in D(A)$, $x\neq 0$, such that $(\lambda-A)x=0$. Such elements $x$ are called \emph{eigenvectors} 
of $A$ (corresponding to $\lambda\in\sigma_{\p}(A)$) and the space $\ker(\lambda-A)=\ker(A-\lambda)$ 
the \emph{eigenspace}. Further, the inclusions 
$\sigma_{\p}(A)\subseteq\sigma_{\bap}^{\seq}(A)\subseteq\sigma_{\bap}(A)$, 
$\sigma_{\ap}^{\seq}(A)\subseteq\sigma_{\ap}(A)$ and $\sigma_{\p}(A)\subseteq\sigma_{\apr}(A)$ hold.
\item Let $\lambda\in\sigma_{\ap}(A)$. Then $\lambda$ is called an \emph{approximate eigenvalue} of $A$ 
and a net $(x_i)_{i\in I}\subseteq D(A)$ not converging to $0$ with $\lim_{i\in I}(A-\lambda )x_i=0$ 
is called an \emph{approximate eigenvector} of $A$ (corresponding to $\lambda$). 
If $I=\N$, then $(x_i)_{i\in I}$ is called a \emph{sequential approximate eigenvector}. 
\end{enumerate}
\end{rem}

In order to clarify the relation between the approximate spectrum and the (sequential) approximate point spectrum, 
we recall the following observations from \cite[Lemma 4.1, p.~268--269]{albanese2013}. 
Let $(X,\tau)$ be a Hausdorff locally convex space and $(A,D(A))$ a linear operator on $X$. Then the system 
of seminorms $(q^{A})_{q\in \Gamma_X }$ defined by
\[
q^{A}(x)\coloneqq q(x)+q(Ax),\quad x\in D(A),\,q\in\Gamma_X ,
\]
defines a Hausdorff locally convex topology on $D(A)$, which is called the \emph{graph topology} and denoted by $\tau^{A}$, and it does not depend on 
the choice of the fundamental system of seminorms $\Gamma_X $ which induces $\tau$.  
If $X$ is (quasi-, sequentially) complete and $(A,D(A))$ closed, then $(D(A),\tau^{A})$ is 
(quasi-, sequentially) complete. 
Moreover, if $(A,D(A))$ is closed and $\lambda\in\C$ is such that $\lambda-A$ is injective, then 
it is easily seen that $(\lambda-A)^{-1}\colon \ran(\lambda-A)\to (D(A),\tau)$ is closed 
(cf.~\cite[p.~258]{albanese2013}), and this implies that $(\lambda-A)^{-1}\colon \ran(\lambda-A)\to (D(A),\tau^{A})$ 
is also closed since $\tau$ is coarser than $\tau^{A}$ on $D(A)$. 

We also recall that a topological space $X$ is called \emph{sequential} by \cite[p.~53]{engelking1989} if every sequentially 
closed subset of $X$ is already closed. Examples of sequential spaces are first countable spaces, in particular metrisable spaces, 
by \cite[1.6.14 Theorem, p.~53]{engelking1989} and Montel DF-spaces by \cite[Theorem 4.6, p.~397]{kakolsaxon2002}.

\begin{prop}\label{prop:approx_spec}
Let $X$ be a Hausdorff locally convex space and $(A,D(A))$ a closed linear operator on $X$. 
Then the following assertions hold.
\begin{enumerate}[label=\upshape(\alph*), leftmargin=*]
\item \label{it:ap_spec_1} If $X$ is complete, then $\sigma_{\apr}(A)\subseteq\sigma_{\ap}(A)$. 
\item \label{it:ap_spec_2} If $X$ is sequentially complete and sequential, then $\sigma_{\apr}(A)\subseteq\sigma_{\ap}^{\seq}(A)$. 
\item \label{it:ap_spec_3} If for all $\lambda\in\sigma_{\ap}(A)\setminus\sigma_{\p}(A)$ such that $\ran(\lambda-A)$ is closed in $X$, 
the closed linear map $(\lambda-A)^{-1}\colon \ran(\lambda-A)\to (D(A),\tau^A )$ is continuous, 
then $\sigma_{\ap}(A)\subseteq \sigma_{\apr}(A)$.
\item \label{it:ap_spec_4} If $X$ is a Fr\'echet space, then $\sigma_{\ap}(A)=\sigma_{\ap}^{\seq}(A)=\sigma_{\apr}(A)$.
\end{enumerate}
\end{prop}
\begin{proof}
\ref{it:ap_spec_1} Let $\lambda\in\sigma_{\apr}(A)$. Due to the inclusions $\sigma_{\p}(A)\subseteq\sigma_{\ap}(A)$ 
and $\sigma_{\p}(A)\subseteq\sigma_{\apr}(A)$ we only need to consider the case that $\lambda-A$ is injective. 
Then the map $(\lambda-A)^{-1}\colon\ran (\lambda-A)\to X$ is well-defined and linear. 
Suppose $\lambda\notin\sigma_{\ap}(A)$ and let $(x_i)_{i\in I}$ be a net in $D(A)$ such that 
$\lim_{i\in I}(\lambda-A)x_i =0$. Then $\lim_{i\in I}x_i =0$ since $\lambda\notin\sigma_{\ap}(A)$. 
This implies that $(\lambda-A)^{-1}\colon\ran (\lambda-A)\to X$ is continuous. 
Now, let $(z_i)_{i\in I}$ be a net in $D(A)$ and $y\in X$ such that 
$\lim_{i\in I}(\lambda-A)z_i =y$. By the continuity of $(\lambda-A)^{-1}\colon\ran (\lambda-A)\to X$ 
we get that for all $q\in\Gamma_X$ there are $p\in\Gamma_X$ and $C\geq 0$ such that
\[
q(z_i-z_j)=q((\lambda-A)^{-1}((\lambda-A)z_i-(\lambda-A)z_j))\leq C p((\lambda-A)z_i-(\lambda-A)z_j)
\]
for all $i,j \in I$. Since $\lim_{i\in I}(\lambda-A)z_i =y$, this estimate implies that $(z_i)_{i\in I}$ is a 
Cauchy net in $X$, which converges to some $z\in X$ by the completeness of $X$. Hence the closedness of $A$ 
yields that $z\in D(A)$ and $y=(\lambda-A)z\in\ran (\lambda-A)$. Thus $\ran (\lambda-A)$ is closed, which is 
a contradiction to $\lambda\in\sigma_{\apr}(A)$.

\ref{it:ap_spec_2} Let $\lambda\in\sigma_{\apr}(A)$. Due to the inclusions 
$\sigma_{\p}(A)\subseteq\sigma_{\ap}^{\seq}(A)$ 
and $\sigma_{\p}(A)\subseteq\sigma_{\apr}(A)$ we only need to consider the case that $\lambda-A$ is injective. 
Suppose $\lambda\notin\sigma_{\ap}^{\seq}(A)$. Looking at the proof of part \ref{it:ap_spec_1}, it follows 
that $(\lambda-A)^{-1}\colon\ran (\lambda-A)\to X$ is sequentially continuous. 
Now, let $(z_i)_{i\in \N}$ be a sequence in $D(A)$ and $y\in X$ such that 
$\lim_{i\to\infty}(\lambda-A)z_i =y$. Then $((\lambda-A)z_i)_{i\in\N}$ is a Cauchy sequence in $\ran (\lambda-A)$ 
and 
\[
z_i=(\lambda-A)^{-1}((\lambda-A)z_i)
\]
for all $i\in\N$. Since sequentially continuous linear operators map Cauchy sequences to Cauchy sequences 
by \cite[Proposition 3.2, p.~1135]{federico2020}, we get that $(z_i)_{i\in\N}$ is a 
Cauchy sequence in $X$, which converges to some $z\in X$ by the sequential completeness of $X$. 
Hence the closedness of $A$ yields that $z\in D(A)$ and $y=(\lambda-A)z\in\ran (\lambda-A)$. 
Thus $\ran (\lambda-A)$ is sequentially closed. Due to $X$ being sequential this implies that the 
linear subspace $\ran (\lambda-A)$ is closed, which is a contradiction to $\lambda\in\sigma_{\apr}(A)$. 

\ref{it:ap_spec_3} Let $\lambda\in\sigma_{\ap}(A)$. Again, we only need to consider the case that 
$\lambda-A$ is injective. Suppose that $\lambda\notin\sigma_{\apr}(A)$. 
Then $(\lambda-A)^{-1}\colon \ran(\lambda-A)\to (D(A),\tau^{A} )$ is continuous by our assumption. 
Now, let $(x_i)_{i\in I}$ be a net in $D(A)$ such that $\lim_{i\in I}(\lambda-A)x_i =0$. 
By the continuity of $(\lambda-A)^{-1}\colon \ran(\lambda-A)\to (D(A),\tau^{A} )$ we obtain that
$\tau^{A}$-$\lim_{i\in I}x_i =0$ and so $\lim_{i\in I}x_i =0$ in the topology of $X$. However, this 
is a contradiction to $\lambda\in\sigma_{\ap}(A)$.

\ref{it:ap_spec_4} Let $X$ be a Fr\'echet space. Then $X$ is complete and sequential because every 
metrisable space is sequential.
If $\lambda\in\C$ is such that $\ran(\lambda-A)$ is closed, 
then $\ran(\lambda-A)$ is also a Fr\'echet space because closed subspaces of Fr\'echet spaces are Fr\'echet spaces. 
Moreover, $(D(A),\tau^A )$ is complete and metrizable, so a Fr\'echet space, 
by our observations above \prettyref{prop:approx_spec} since $\Gamma_X$ and thus $(q^{A})_{q\in \Gamma_X }$ 
can be chosen as a countable system of seminorms. 
Therefore the closed linear map $(\lambda-A)^{-1}\colon \ran(\lambda-A)\to (D(A),\tau^A )$ is continuous for any 
$\lambda\in\C\setminus\sigma_{\p}(A)$ such that $\ran(\lambda-A)$ is closed 
by \cite[Closed graph theorem 24.31, p.~289]{meisevogt1997}. 
We conclude our statement from parts \ref{it:ap_spec_2}, \ref{it:ap_spec_3} and the inclusion 
$\sigma_{\ap}^{\seq}(A)\subseteq\sigma_{\ap}(A)$. 
\end{proof}

\prettyref{prop:approx_spec} \ref{it:ap_spec_4} generalises \cite[Chap.~IV, 1.9 Lemma, p.~242]{engel_nagel2000} from 
Banach spaces $X$ to Fr\'echet spaces. 

\begin{rem}\label{rem:trap}
Let $X$ be a sequentially complete Hausdorff locally convex space. 
Looking at the proof of \prettyref{prop:approx_spec} \ref{it:ap_spec_2}, 
we note that it still holds if $(A,D(A))$ is a sequentially closed linear operator on $X$, and $X$ is such that 
every sequentially closed linear subspace of $X$ is already closed.  
\end{rem}

We know that metrisable Hausdorff locally convex spaces and Montel DF-spaces are sequential. 
We observe that the first ones belong to the class of bornological spaces, the latter ones to the class of barrelled spaces. 
However, there are sequential Hausdorff locally convex spaces outside of these classes, namely among the Saks spaces. 
Indeed, if $(X,\|\cd\|,\tau)$ is a Saks space and $\tau\neq\tau_{\|\cd\|}$, 
then $(X,\gamma)$ is neither bornological nor barrelled by \cite[I.1.15 Proposition, p.~12]{cooper1978}. 

\begin{prop}
Let $(X,\|\cd\|,\tau)$ be a Saks space such that $\tau$ is metrisable on $B_{\|\cd\|}\coloneq\{x\in X\;|\;\|x\|\leq 1\}$ and $B_{\|\cd\|}$ is $\tau$-compact. Then $(X,\gamma)$ is sequential. 
\end{prop} 
\begin{proof}
Let $M\subseteq X$ be sequentially $\gamma$-closed. For $k\in\N$ let $B_{k}\coloneqq k B_{\|\cdot\|}$ and denote by $\tau_{k}$ 
the topology $\tau$ restricted to $B_{k}$. Since $\tau$ is metrisable on the $\tau$-compact set $B_{\|\cd\|}$, we have that 
$\tau_{k}$ is metrisable and $B_{k}$ $\tau$-closed. 
Let $(x_{n})_{n\in\N}$ be a sequence in $M\cap B_{k}$ that $\tau$-converges to some $x\in X$. 
In particular, this implies that $x\in B_{k}$ because $B_{k}$ is $\tau$-closed. 
By \cite[I.1.10 Proposition, p.~9]{cooper1978} we also have that $(x_n)_{n\in\N}$ $\gamma$-converges to $x$ as this sequence is 
$\tau$-convergent to $x$ and $\|\cd\|$-bounded. Thus the sequential $\gamma$-closedness of $M$ yields that $x\in M$. 
Therefore $x\in M\cap B_{k}$ and so $M\cap B_{k}$ is sequentially $\tau_{k}$-closed. As $\tau_{k}$ is metrisable, 
this means that $M\cap B_{k}$ is $\tau_{k}$-closed. Due to \cite[I.4.2 Corollary, p.~38]{cooper1978} with 
$\mathcal{B}\coloneqq\{B_{k}\;|\;k\in\N\}$ we obtain that $M$ is $\gamma$-closed because $B_{\|\cd\|}$ is $\tau$-compact.
\end{proof}

A concrete example of such a Saks space is given in \prettyref{ex:periodic_sg}. More examples can be found in 
\cite[Example 3.11 (b)--(d), p.~10, Remark 3.19, p.~14]{kruse_schwenninger2022} and 
\cite[3.5 Corollary (c), p.~9, \& Section 4]{kruse2024a}. 
We also make the following observation that the bounded (sequential) approximate point spectrum coincides 
with the (sequential) approximate point spectrum on Banach spaces.

\begin{prop}\label{prop:seq_ap_bounded_ap}
Let $(X,\|\cd\|)$ be a Banach space and $(A,D(A))$ a closed linear operator on $X$. Then
\[
\sigma_{\ap}(A)=\sigma_{\ap}^{\seq}(A)=\sigma_{\bap}(A)=\sigma_{\bap}^{\seq}(A)=\sigma_{\apr}(A).
\]
\end{prop}
\begin{proof}
Due to \prettyref{prop:approx_spec} and the inclusion 
$\sigma_{\bap}^{\seq}(A)\subseteq \sigma_{\bap}(A)\subseteq \sigma_{\ap}(A)$ we only need to prove the inclusion 
$\sigma_{\ap}^{\seq}(A)\subseteq\sigma_{\bap}^{\seq}(A)$. Let $\lambda\in\sigma_{\ap}^{\seq}(A)$. 
Then there exists a sequence $(x_i)_{i\in \N}$ in $D(A)$ which does not converge to $0$ and fulfils
$\lim_{i\to\infty}(A-\lambda )x_i=0$. Since $(x_i)_{i\in\N}$ does not converge to $0$, 
there is $\varepsilon >0$ such that for all $j\in\N$ there is $i\in\N$, $i\geq j$, 
with $\|x_i\|\geq\varepsilon$. By passing to a subsequence, which we still denote by $(x_i)_{i\in\N}$, 
we may assume that $\|x_i\|\geq\varepsilon$ for all $i\in\N$. 
Now, we set $y_{i}\coloneqq\frac{x_i}{\|x_i\|}$ for $i\in\N$. Then $(y_i)_{i\in\N}$ is $\|\cd\|$-bounded 
and does not converge to $0$ since $\|y_i\|=1$ for all $i\in\N$, and 
\[
\|(\lambda-A)y_i\|=\frac{1}{\|x_{i}\|}\|(\lambda-A)x_i\|\leq \frac{1}{\varepsilon}\|(\lambda-A)x_i\|
\] 
for all $i\in\N$, which implies that $\lim_{i\to\infty}(\lambda-A)y_i=0$. 
We deduce that $\lambda\in\sigma_{\bap}^{\seq}$.
\end{proof}

We are now able to give some decompositions of the spectrum.

\begin{prop}\label{prop:decomp_spec}
Let $X$ be a Hausdorff locally convex space and $(A,D(A))$ a linear operator on $X$. 
Then the following assertions hold.
\begin{enumerate}[label=\upshape(\alph*), leftmargin=*]
\item \label{it:decomp_spec_1} $\sigma_{\alg}(A)=\sigma_{\apr}(A)\cup\sigma_{\res}(A)$, 
$\sigma(A)=\sigma_{\alg}(A)\cup\sigma_{\tp}(A)$ and $\sigma_{\ap}(A)\subseteq\sigma(A)$.
\item \label{it:decomp_spec_2} If $X$ is complete and $(A,D(A))$ closed, then
$\sigma(A)=\sigma_{\ap}(A)\cup\sigma_{\res}(A)\cup\sigma_{\tp}(A)$. 
\item \label{it:decomp_spec_3} If $X$ is sequentially complete and such that 
every sequentially closed linear subspace of $X$ is already closed, and 
$(A,D(A))$ sequentially closed, then $\sigma(A)=\sigma_{\ap}^{\seq}(A)\cup\sigma_{\res}(A)\cup\sigma_{\tp}(A)$. 
\end{enumerate}
\end{prop}
\begin{proof}
\ref{it:decomp_spec_1} ``$\sigma_{\alg}(A)\subseteq\sigma_{\apr}(A)\cup\sigma_{\res}(A)$'' 
Let $\lambda\in\sigma_{\alg}(A)$. If $\lambda-A$ is not injective, then $\lambda\in\sigma_{\apr}(A)$. 
So, let $\lambda-A$ be injective but not surjective. If $\lambda\notin \sigma_{\apr}(A)$, then $\ran (\lambda-A)$ 
is closed in $X$. Suppose that $\lambda\notin\sigma_{\res}(A)$. 
Then $\ran (\lambda-A)$ is dense and closed in $X$, which implies that $\ran (\lambda-A)=X$ 
and so contradicts that $\lambda-A$ is not surjective. Hence we have $\lambda\in\sigma_{\res}(A)$.

``$\sigma_{\apr}(A)\cup\sigma_{\res}(A)\subseteq\sigma_{\alg}(A)$'' This inclusion is obvious since $\lambda-A$ for 
$\lambda\in\C$ is not surjective if $\ran (\lambda-A)$ is not closed or not dense in $X$. 

The identity $\sigma(A)=\sigma_{\alg}(A)\cup\sigma_{\tp}(A)$ follows by the definitions of the sets involved.

``$\sigma_{\ap}(A)\subseteq\sigma(A)$'' Let $\lambda\in\sigma_{\ap}(A)$. 
If $\lambda\in\sigma_{\alg}(A)$, then $\lambda\in \sigma(A)$. If $\lambda\notin\sigma_{\alg}(A)$, 
then $(\lambda-A)^{-1}\colon X\to X$ is well-defined and linear. Since $\lambda\in\sigma_{\ap}(A)$, 
there exists a net $(x_i)_{i\in I}$ in $D(A)$ which does not converge to $0$ and fulfils
$\lim_{i\in I}(A-\lambda )x_i=0$. Writing $y_i\coloneqq(A-\lambda )x_i \in X$, $i\in I$, 
this means that the net $(y_i)_{i\in I}$ converges to $0$ in $X$ but the net $((\lambda-A)^{-1}y_i)_{i\in I}$ 
does not converge to $0$ in $X$ as $(\lambda-A)^{-1}y_i=-x_i$ for all $i\in I$. 
Thus $(\lambda-A)^{-1}\notin\cL(X)$ and so $\lambda\in\sigma_{\tp}(A)$.

\ref{it:decomp_spec_2} This statement follows from part \ref{it:decomp_spec_1} and 
\prettyref{prop:approx_spec} \ref{it:ap_spec_1}.

\ref{it:decomp_spec_3} This statement follows from part \ref{it:decomp_spec_1}, 
\prettyref{prop:approx_spec} \ref{it:ap_spec_2} and \prettyref{rem:trap}.
\end{proof}

Next, we turn to spectral mapping theorems for resolvents where we slightly refine 
\cite[Theorem 1.1, p.~269]{albanese2016a}.

\begin{thm}\label{thm:spec_map_res}
Let $X$ be a Hausdorff locally convex space and $(A,D(A))$ a linear operator on $X$. 
Then the following assertions hold for every $\lambda\in\rho(A)$.
\begin{enumerate}[label=\upshape(\alph*), leftmargin=*]
\item \label{it:spec_map_res_1} $\sigma(R(\lambda,A))\setminus\{0\}=\{\frac{1}{\lambda-\mu}\;|\;\mu\in\sigma(A)\}$,
\item \label{it:spec_map_res_2} $\sigma_{\alg}(R(\lambda,A))\setminus\{0\}=\{\frac{1}{\lambda-\mu}\;|\;\mu\in\sigma_{\alg}(A)\}$,
\item \label{it:spec_map_res_3} $\sigma_{\p}(R(\lambda,A))\setminus\{0\}=\{\frac{1}{\lambda-\mu}\;|\;\mu\in\sigma_{\p}(A)\}$,
\item \label{it:spec_map_res_4} $\sigma_{\ap}(R(\lambda,A))\setminus\{0\}=\{\frac{1}{\lambda-\mu}\;|\;\mu\in\sigma_{\ap}(A)\}$,
\item \label{it:spec_map_res_4a} $\sigma_{\ap}^{\seq}(R(\lambda,A))\setminus\{0\}=\{\frac{1}{\lambda-\mu}\;|\;\mu\in\sigma_{\ap}^{\seq}(A)\}$,
\item \label{it:spec_map_res_4b} $\sigma_{\bap}(R(\lambda,A))\setminus\{0\}=\{\frac{1}{\lambda-\mu}\;|\;\mu\in\sigma_{\bap}(A)\}$,
\item \label{it:spec_map_res_4c} $\sigma_{\bap}^{\seq}(R(\lambda,A))\setminus\{0\}=\{\frac{1}{\lambda-\mu}\;|\;\mu\in\sigma_{\bap}^{\seq}(A)\}$,
\item \label{it:spec_map_res_5} $\sigma_{\apr}(R(\lambda,A))\setminus\{0\}=\{\frac{1}{\lambda-\mu}\;|\;\mu\in\sigma_{\apr}(A)\}$,
\item \label{it:spec_map_res_6} $\sigma_{\res}(R(\lambda,A))\setminus\{0\}=\{\frac{1}{\lambda-\mu}\;|\;\mu\in\sigma_{\res}(A)\}$,
\item \label{it:spec_map_res_7} $\sigma_{\tp}(R(\lambda,A))\setminus\{0\}=\{\frac{1}{\lambda-\mu}\;|\;\mu\in\sigma_{\tp}(A)\}$,
\item \label{it:spec_map_res_8} $\sigma^{\ast}(R(\lambda,A))\setminus\{0\}=\{\frac{1}{\lambda-\mu}\;|\;\mu\in\sigma^{\ast}(A)\}$.
\end{enumerate}
\end{thm}
\begin{proof}
If $\rho(A)=\varnothing$, then the stated identities are trivially true. 
If $\rho(A)\neq\varnothing$, then $(A,D(A))$ is closed by \cite[Remark 3.1 (i), p.~259]{albanese2013} 
and it is shown in \cite[Theorem 1.1, p.~269]{albanese2016a} that the statements \ref{it:spec_map_res_1}, 
\ref{it:spec_map_res_3} and \ref{it:spec_map_res_8} hold. 
Further, the proof of \cite[Theorem 1.1, p.~269]{albanese2016a} relies on the equations 
\cite[Eq.~(2.1)--(2.4), p.~270]{albanese2016a} for $\lambda\in\rho(A)$ and 
$\eta\in\C$, $\eta\neq 0$, which are
\begin{equation}\label{eq:spec_map_res_1}
(\eta-R(\lambda,A))x=\eta\Bigl(\bigl(\lambda-\frac{1}{\eta}\bigr)-A\Bigr)R(\lambda,A)x,\quad x\in X,
\end{equation}
and
\begin{equation}\label{eq:spec_map_res_2}
(\eta-R(\lambda,A))x=R(\lambda,A)\eta\Bigl(\bigl(\lambda-\frac{1}{\eta}\bigr)-A\Bigr)x,\quad x\in D(A),
\end{equation}
implying
\begin{equation}\label{eq:spec_map_res_3}
\ker(\eta-R(\lambda,A)) = \ker\Bigl(\bigl(\lambda-\frac{1}{\eta}\bigr)-A\Bigr)
\end{equation}
and
\begin{equation}\label{eq:spec_map_res_4}
\ran(\eta-R(\lambda,A)) = \ran\Bigl(\bigl(\lambda-\frac{1}{\eta}\bigr)-A\Bigr).
\end{equation}

\ref{it:spec_map_res_2} From \eqref{eq:spec_map_res_4} we see that $\eta-R(\lambda,A)$ is not surjective if and only if 
$(\lambda-\frac{1}{\eta})-A$ is not surjective. 
So, $\eta-R(\lambda,A)$ not being surjective, implies that $\mu\coloneqq\lambda-\frac{1}{\eta}\in\sigma_{\alg}(A)$ and 
$\eta=\frac{1}{\lambda-\mu}$. On the other hand, if $\mu\in\sigma_{\alg}(A)$ is such that $\mu-A$ is not 
surjective, then we have with $\eta\coloneqq \frac{1}{\lambda-\mu}$ that $\mu=\lambda-\frac{1}{\eta}$ and 
so $\eta-R(\lambda,A)$ is not surjective. In combination with part \ref{it:spec_map_res_3} this yields that 
statement \ref{it:spec_map_res_2} holds. 

\ref{it:spec_map_res_4} ``$\subseteq$'' Let $\eta\in\sigma_{\ap}(R(\lambda,A))\setminus\{0\}$. Then
there exists a net $(x_i)_{i\in I}$ in $X$ which does not converge to $0$ and fulfils
$\lim_{i\in I}(R(\lambda,A)-\eta )x_i=0$. Then the net $(y_i)_{i\in I}$ defined by 
$y_i\coloneqq \eta R(\lambda,A)x_i \in D(A)$ fulfils $\lim_{i\in I}((\lambda-\frac{1}{\eta})-A)y_i=0$ by 
\eqref{eq:spec_map_res_1}. Suppose that $\lim_{i\in I}y_i=0$. Then 
\[
x_i=\frac{1}{\eta}(\lambda-A)y_i=\frac{1}{\eta}\Bigl(\bigl(\lambda-\frac{1}{\eta}\bigr)-A\Bigr)y_i+\frac{1}{\eta^2}y_i
\to 0,
\]
which is a contradiction. Thus $\mu\coloneqq (\lambda-\frac{1}{\eta})\in\sigma_{\ap}(A)$ and 
$\eta=\frac{1}{\lambda-\mu}$.

``$\supseteq$'' Let $\mu\in\sigma_{\ap}(A)$. Then there is a net $(x_i)_{i\in I}$ in $D(A)$ 
which does not converge to $0$ and fulfils $\lim_{i\in I}(A-\mu )x_i=0$. Then we have with 
$\eta\coloneqq \frac{1}{\lambda-\mu}$ that $\mu=\lambda-\frac{1}{\eta}$  and 
$\lim_{i\in I}(\eta-R(\lambda,A))x_i=0$ by \eqref{eq:spec_map_res_2} and the continuity of $R(\lambda,A)$. 
Thus $\frac{1}{\lambda-\mu}=\eta\in \sigma_{\ap}(R(\lambda,A))\setminus\{0\}$.

\ref{it:spec_map_res_4a} This statement follows from the proof of part \ref{it:spec_map_res_4} with $I\coloneqq\N$.

\ref{it:spec_map_res_4b} This statement follows from the proof of part \ref{it:spec_map_res_4} by noting that 
the net $(y_i)_{i\in I}$ in the proof of the inclusion ``$\subseteq$'' is bounded if $(x_i)_{i\in I}$ is bounded because 
$R(\lambda,A)$ is continuous.

\ref{it:spec_map_res_4c} This statement follows from the observation in the proof of part \ref{it:spec_map_res_4b} 
with $I\coloneqq\N$.

\ref{it:spec_map_res_5} This statement follows from part \ref{it:spec_map_res_3} and \eqref{eq:spec_map_res_4}.

\ref{it:spec_map_res_6} This statement follows from \eqref{eq:spec_map_res_4}.

\ref{it:spec_map_res_7} Since 
$\sigma_{\tp}(R(\lambda,A))=\sigma(R(\lambda,A))\setminus\sigma_{\alg}(R(\lambda,A))$ and 
$\sigma_{\tp}(A)=\sigma(A)\setminus\sigma_{\alg}(A)$, this statement follows from parts \ref{it:spec_map_res_1} 
and \ref{it:spec_map_res_2}.
\end{proof}

In the case that $X$ is a Banach space and $(A,D(A))$ closed, \prettyref{thm:spec_map_res} can also be found in 
\cite[Chap.~IV, 1.13 Spectral Mapping Theorem for the Resolvent, p.~243]{engel_nagel2000}.

We close this section with some remarks on the dual operator of a densely defined linear operator. 
We recall that a Hausdorff locally convex space is called a \emph{Schur space} if every weakly convergent sequence 
is convergent (see \cite[p.~81]{martinpeinador2015}).

\begin{prop}\label{prop:dual_spec}
Let $X$ be a Hausdorff locally convex space, $(A,D(A))$ a densely defined linear operator on $X$ and equip 
$X'$ with the topology $\beta(X',X)$ of uniform convergence on bounded subsets of $X$. 
Then the following assertions hold.
\begin{enumerate}[label=\upshape(\alph*), leftmargin=*]
\item \label{it:dual_spec_1} $\sigma(A')\subseteq \sigma(A)$,
\item \label{it:dual_spec_2} $\sigma_{\alg}(A')\subseteq \sigma_{\alg}(A)$ and 
$(\lambda-A')^{-1}=((\lambda-A)^{-1})'$ for all $\lambda\in\rho_{\alg}(A)$,
\item \label{it:dual_spec_3} $\sigma_{\p}(A')=\sigma_{\res}(A)$,
\item \label{it:dual_spec_4} $\sigma^{\ast}(A')\subseteq \sigma^{\ast}(A)$,
\item \label{it:dual_spec_5} $\sigma_{\p}(A)\subseteq \sigma_{\alg}(A')$ if $\rho(A)\neq\varnothing$,
\item \label{it:dual_spec_6} $\sigma_{\ap}^{\seq}(A)\subseteq \sigma_{\alg}(A')$ if $\rho(A)\neq\varnothing$ 
and $X$ is a Schur space.
\end{enumerate}
\end{prop}
\begin{proof}
\ref{it:dual_spec_2} Let $\lambda\in\rho_{\alg}(A)$. Then we have for all $x'\in D(A')$ and $x\in X$ that
\[
 \langle ((\lambda-A)^{-1})'(\lambda-A')x',x\rangle
=\langle (\lambda-A)'x',(\lambda-A)^{-1}x\rangle
=\langle x',(\lambda-A)(\lambda-A)^{-1}x\rangle
=\langle x',x\rangle .
\]
Further, for all $x\in D(A)$ and $x'\in X'$ we have
\[
 \langle (\lambda-A)x,((\lambda-A)^{-1})'x'\rangle
=\langle (\lambda-A)^{-1}(\lambda-A)x,x'\rangle
=\langle x,x'\rangle
\]
and therefore $((\lambda-A)^{-1})'x'\in D(A')$, which implies
\[
 \langle(\lambda-A') ((\lambda-A)^{-1})'x',x\rangle
=\langle ((\lambda-A)^{-1})'x',(\lambda-A)x\rangle
=\langle x',(\lambda-A)^{-1}(\lambda-A)x\rangle
=\langle x',x\rangle
\]
Since $(\lambda-A') ((\lambda-A)^{-1})'x'\in X'$ and $D(A)$ is dense in $X$, we obtain that the identity
$(\lambda-A') ((\lambda-A)^{-1})'x'=x'$ holds on whole $X$. 
Thus $\rho_{\alg}(A)\subseteq \rho_{\alg}(A')$ and 
$(\lambda-A')^{-1}=((\lambda-A)^{-1})'$ for all $\lambda\in\rho_{\alg}(A)$.

\ref{it:dual_spec_1} Let $\lambda\in\rho(A)$. Then $(\lambda-A')^{-1}=R(\lambda,A)'$ by part \ref{it:dual_spec_2}. 
For bounded $M\subseteq X$ and $x'\in X'$ we have
\begin{align*}
 \sup_{x\in M}|\langle (\lambda-A')^{-1}x',x\rangle |
&=\sup_{x\in M}|\langle R(\lambda,A)'x',x\rangle | 
 =\sup_{x\in M}|\langle x',R(\lambda,A)x\rangle |\\ 
&=\sup_{y\in R(\lambda,A)M}|\langle x',y\rangle |, 
\end{align*}
which implies $(\lambda-A')^{-1}\in\cL(X_{\ob}')$ since $R(\lambda,A)M$ is a bounded subset of $X$. 
Hence $\lambda\in\rho(A')$.

\ref{it:dual_spec_4} Let $\lambda\in\rho^{\ast}(A)$. Then there is $\delta>0$ such that 
$B(\lambda,\delta)\subseteq\rho(A)$ and $\{R(\mu,A)\;|\;\mu\in B(\lambda,\delta)\}$ is equicontinuous in $\cL(X)$. 
By parts \ref{it:dual_spec_1} and \ref{it:dual_spec_2} we have $B(\lambda,\delta)\subseteq\rho(A')$ and 
$R(\mu,A')=R(\mu,A)'$ for all $\mu\in B(\lambda,\delta)$. For bounded $M\subseteq X$, $x'\in X'$ and 
$\mu\in B(\lambda,\delta)$ we have 
\[
 \sup_{x\in M}|\langle R(\mu,A')x',x\rangle |
=\sup_{y\in R(\mu,A)M}|\langle x',y\rangle |
\leq \sup_{y\in N}|\langle x',y\rangle |
\]
with $N\coloneqq R(B(\lambda,\delta),A)M$. The set $N$ is bounded in $X$ as 
$\{R(\mu,A)\;|\;\mu\in B(\lambda,\delta)\}$ is equicontinuous in $\cL(X)$. 
Thus $\{R(\mu,A')\;|\;\mu\in B(\lambda,\delta)\}$ is equicontinuous in $\cL(X_{\ob}')$ and so 
$\lambda\in\rho^{\ast}(A')$.

\ref{it:dual_spec_3} By the bipolar theorem $\ran (\lambda-A)$ for $\lambda\in\C$ is not dense in $X$ if and only if 
there is $x'\in X'$, $x'\neq 0$, such that $x'((\lambda-A)x)=0$ for all $x\in D(A)$, i.e.
\[
\langle x',Ax \rangle = \langle  \lambda x', x \rangle
\]
for all $x\in D(A)$, meaning $x'\in D(A')$, $x'\neq 0$, such that $(\lambda-A')x'=0$.

\ref{it:dual_spec_5} Let $\lambda\in\sigma_{\p}(A)$ and choose $\mu\in\rho(A)$.
Then $\mu\in\rho(A')$ and $R(\mu,A')=R(\mu,A)'$ by parts \ref{it:dual_spec_1} and \ref{it:dual_spec_2}. 
We claim that $\frac{1}{\mu-\lambda}-R(\mu,A')$ is not surjective. Suppose the contrary. 
We have $\frac{1}{\mu-\lambda}\in\sigma_{\p}(R(\mu,A))\setminus\{0\}$ by 
\prettyref{thm:spec_map_res} \ref{it:spec_map_res_3}. So, there is $x\in X$, $x\neq 0$, such that 
$R(\mu,A)x=\frac{1}{\mu-\lambda}x$. We observe that for all $x'\in X'$
\[
 0
=\Bigl\langle x', \Bigl(\frac{1}{\mu-\lambda}-R(\mu,A)\Bigr)x \Bigr\rangle
=\Bigl\langle \Bigl(\frac{1}{\mu-\lambda}-R(\mu,A')\Bigr)x', x \Bigr\rangle .
\] 
Since $\frac{1}{\mu-\lambda}-R(\mu,A')$ is surjective, this implies that $\langle y',x\rangle =0$ for all 
$y'\in X'$. By the Hahn--Banach theorem we get $x=0$, which is a contradiction. 
Thus $\frac{1}{\mu-\lambda}-R(\mu,A')$ is not surjective. 
Due to \prettyref{thm:spec_map_res} \ref{it:spec_map_res_2} we obtain that $\lambda\in\sigma_{\alg}(A')$. 

\ref{it:dual_spec_6} Let $\lambda\in\sigma_{\ap}^{\seq}(A)$, choose $\mu\in\rho(A)$ and suppose that 
$\frac{1}{\mu-\lambda}-R(\mu,A')$ is not surjective. Proceeding as in the proof of part \ref{it:dual_spec_5} 
and using \prettyref{thm:spec_map_res} \ref{it:spec_map_res_4a}, there is a sequence $(x_i)_{i\in\N}$ in $X$ 
that does not converge to $0$ and fulfils $\lim_{i\to\infty}(\frac{1}{\mu-\lambda}-R(\mu,A))x_{i}=0$. 
We observe that for all $x'\in X'$ 
\[
 0
=\Bigl\langle x', \lim_{i\to\infty}\Bigl(\frac{1}{\mu-\lambda}-R(\mu,A)\Bigr)x_{i} \Bigr\rangle
= \lim_{i\to\infty}\Bigl\langle \Bigl(\frac{1}{\mu-\lambda}-R(\mu,A')\Bigr)x', x_{i} \Bigr\rangle .
\] 
Since $\frac{1}{\mu-\lambda}-R(\mu,A')$ is surjective, this implies that $\lim_{i\to\infty}\langle y',x_i\rangle =0$ 
for all $y'\in X'$. Hence $(x_i)_{i\in\N}$ weakly converges to $0$. 
We deduce that $(x_i)_{i\in\N}$ converges to $0$ in $X$ as $X$ is a Schur space, which is contradiction. 
Thus $\frac{1}{\mu-\lambda}-R(\mu,A')$ is not surjective and we obtain that $\lambda\in\sigma_{\alg}(A')$ as in 
the proof of part \ref{it:dual_spec_5}. 
\end{proof}

\prettyref{prop:dual_spec} \ref{it:dual_spec_3} also follows from \cite[Proposition 3.6 (i), p.~924]{albanese2016b}. 
Further, it generalises \cite[Chap.~IV, 1.12 Proposition, p.~243]{engel_nagel2000} 
where $X$ is a Banach space and $(A,D(A))$ in addition assumed as closed. 
In this setting, so if $X$ is a Banach space and $(A,D(A))$ a densely defined closed linear operator on $X$, 
we have even equalities in \prettyref{prop:dual_spec} \ref{it:dual_spec_1}, \ref{it:dual_spec_2} 
and \ref{it:dual_spec_4} by e.g.~\cite[Lemma 1.4.1, p.~9]{vanneerven1992}. 
Moreover, we note that the idea how to prove \prettyref{prop:dual_spec} \ref{it:dual_spec_5} comes from 
\cite[Theorem 1.5.3, p.~18]{farkas2003}.

\section{Periodic semigroups}
\label{sect:periodic_sg}

In this section we study periodic semigroups which turn out to be quite useful in proving 
the spectral mapping theorem for the point spectrum of strongly continuous semigroups in our final section.

\begin{defn}
A semigroup $(T(t))_{t\geq 0}$ on a Hausdorff locally convex space $X$ is called \emph{periodic} if there is $t_0>0$ 
such that $T(t_0)=\id$. The (minimal) \emph{period} $\rho$ of $(T(t))_{t\geq 0}$ is given by 
\[
\rho\coloneqq\inf\{t_0>0\;|\;T(t_0)=\id\}
\]
\end{defn}

This is a generalisation of the notion of a periodic semigroup given in 
\cite[Chap.~IV, 2.23 Definition, p.~266]{engel_nagel2000} where $(T(t))_{t\geq 0}$ is a strongly continuous semigroup 
on a Banach space $X$.

\begin{rem}\label{rem:periodic_strongly}
Let $(T(t))_{t\geq 0}$ be a periodic semigroup on a Hausdorff locally convex space $X$ with period $\rho$.
\begin{enumerate}[label=\upshape(\alph*), leftmargin=*]
\item \label{it:period_strong_1} Since there is $t_0>0$ with $T(t_0)=\id$, the semigroup $(T(t))_{t\geq 0}$ 
can be extended to a group on $\R$ by setting $T(t-nt_0)\coloneqq T(t)$ for all $t\geq 0$ and $n\in\N$. 
\item \label{it:period_strong_2} If $(T(t))_{t\geq 0}$ is strongly continuous, then there is a sequence $(t_n)_{n\in\N}$ in $(0,\infty)$ 
which converges to $\rho$ such that $T(t_n)x=x$ for all $n\in\N$ and $x\in X$, implying 
$T(\rho)x=\lim_{n\to\infty}T(t_n)x=x$ for all $x\in X$ by the strong continuity.
\item \label{it:period_strong_3} If $(T(t))_{t\geq 0}$ is locally equicontinuous, then it is equicontinuous. 
Indeed, if $(T(t))_{t\geq 0}$ is periodic, then there is $t_0>0$ such that $T(t_0)=\id$ 
and we have $\sup_{t\geq 0}q(T(t)x)=\sup_{t\in [0,t_0]}q(T(t)x)$ for all $x\in X$ and $q\in\Gamma_X$ by the 
semigroup property.
\end{enumerate}
\end{rem}

\begin{prop}\label{prop:natural_period}
Let $X$ be a Hausdorff locally convex space and $(T(t))_{t\geq 0}$ a strongly continuous 
semigroup on $X$. If there is $t_0>0$ such that $T(t_0)=\id$, then $(T(t))_{t\geq 0}$ is periodic 
and its period $\rho$ fulfils $\rho\in\{0\}\cup\{\frac{t_0}{k}\;|\;k\in\N\}$.
\end{prop}
\begin{proof}
$(T(t))_{t\geq 0}$ is periodic by definition and its period $\rho\geq 0$ fulfils $T(\rho)=\id$ 
by \prettyref{rem:periodic_strongly} \ref{it:period_strong_2}. 
Let $\rho\notin\{\frac{t_0}{k}\;|\;k\in\N\}$. We show that $\rho=0$ in this case. 
We claim that $\rho\in [0,\frac{t_0}{k})$ for all $k\in\N$. 
Indeed, for $k=1$ this is true by assumption. Let $k\in\N$ such that $\rho\in [0,\frac{t_0}{k})$. 
Suppose that $\rho\in (\frac{t_0}{k+1},\frac{t_0}{k})$. Then $0<k(\frac{t_0}{k}-\rho)<\frac{t_0}{k+1}<\rho$ and 
\[
 T\Bigl(k\Bigl(\frac{t_0}{k}-\rho\Bigr)\Bigr)
=T(t_0-k\rho)
=T(t_0-k\rho)\id
=T(t_0-k\rho)T(k\rho)
=T(t_0)
=\id ,
\]
which is a contradiction to the minimality of the period $\rho$. This verifies our claim. 
Since $\rho\in [0,\frac{t_0}{k})$ for all $k\in\N$, we conclude that $\rho=0$.
\end{proof}

\begin{prop}\label{prop:period_zero}
Let $X$ be a sequentially complete Hausdorff locally convex space and $(T(t))_{t\geq 0}$ a strongly continuous 
semigroup on $X$ with generator $(A,D(A))$. Then $(T(t))_{t\geq 0}$ is periodic with period $\rho=0$ if and only if 
$T(t)=\id$ for all $t\geq 0$.
\end{prop}
\begin{proof}
The implication ``$\Leftarrow$'' is clear so let us turn to the implication ``$\Rightarrow$''. 
Since $\rho=0$, there is a sequence $(t_n)_{n\in\N}$ in $(0,\infty)$ 
which converges to $0$ such that $T(t_n)x=x$ for all $n\in\N$ and $x\in X$. It follows that 
\[
 Ax
=\lim_{t\to 0\rlim}\frac{T(t)x-x}{t}
=\lim_{n\to \infty}\frac{T(t_n)x-x}{t_n}
=0
\]
for all $x\in D(A)$. Hence $T(t)=\id$ for all $t\geq 0$ by Equation \eqref{eq:rescale_iden_1} from \prettyref{prop:rescale_identities} with $\lambda\coloneqq 0$.
\end{proof}

Next, we transfer \cite[Chap.~IV, 2.24 Lemma, p.~266]{engel_nagel2000} to the setting of 
strongly continuous semigroups on Hausdorff locally convex spaces.

\begin{prop}\label{prop:suff_cond_periodic}
Let $X$ be a Hausdorff locally convex space and $(T(t))_{t\geq 0}$ a strongly continuous 
semigroup on $X$ with generator $(A,D(A))$. If 
\begin{enumerate}[label=\upshape(\roman*), leftmargin=*, widest=ii]
\item \label{it:period_1} $\sigma_{\p}(A)\subseteq 2\pi \mathsf{i}\alpha\Z$ for some $\alpha>0$, and
\item \label{it:period_2} the corresponding eigenvectors span a dense subspace of $X$,
\end{enumerate}
then $(T(t))_{t\geq 0}$ is periodic with period $\rho\leq \frac{1}{\alpha}$.
\end{prop}
\begin{proof}
First, we note that condition \ref{it:period_2} guarantees that $\sigma_{\p}(A)\neq\varnothing$. 
Let $\lambda\in\sigma_{\p}(A)$. Then there is $x\in D(A)$, $x\neq 0$, such that $Ax=\lambda x$, and by 
\ref{it:period_1} there is $n\in\Z$ with $\lambda=2\pi \mathsf{i}\alpha n$. For $t>0$ we set $f\colon [0,t]\to X$, 
$f(s)\coloneqq \e^{2\pi \mathsf{i}\alpha n(t-s)}T(s)x$. Due to \cite[Proposition 1.2 (1), p.~260]{komura1968} $f$ is 
continuously differentiable and 
\begin{align*}
  f'(s)
&=-2\pi \mathsf{i} \alpha n\e^{2\pi \mathsf{i}\alpha n(t-s)}T(s)x+\e^{2\pi \mathsf{i}\alpha n(t-s)}T(s)Ax\\
&=\e^{2\pi \mathsf{i}\alpha n(t-s)}(-T(s)(2\pi \mathsf{i} \alpha n x)+T(s)(2\pi \mathsf{i} \alpha n x))
 =0
\end{align*}
for all $s\in [0,t]$, $f(0)=\e^{2\pi \mathsf{i}\alpha nt}x$, and $f(t)=T(t)x$. 
Thus $T(t)x=\e^{2\pi \mathsf{i}\alpha nt}x$ for all 
$t\geq 0$ and so $T(\frac{1}{\alpha})x=x$. We deduce that $(T(t))_{t\geq 0}$ is periodic with period 
$\rho\leq \frac{1}{\alpha}$ from \ref{it:period_2} and the continuity of $T(\frac{1}{\alpha})$.
\end{proof}

Next, we generalise some results from \cite[Chap.~IV, 1.17 Isolated Singularities, p.~246--247]{engel_nagel2000}, \cite[Chap.~IV, 2.25 Lemma, p.~266]{engel_nagel2000}, \cite[Chap.~IV, Eq.~(2.9), p.~267]{engel_nagel2000} 
and \cite[Chap.~8, Sect.~8, Theorem 3, p.~229]{yosida1968}. 
This allows us to describe the (point) spectrum and the eigenspaces of a periodic strongly continuous 
locally equicontinuous semigroup by studying the properties of its resolvent.

\begin{prop}\label{prop:resolvent_periodic}
Let $X$ be a sequentially complete Hausdorff locally convex space and $(T(t))_{t\geq 0}$ a periodic strongly continuous 
locally equicontinuous semigroup on $X$ with period $\rho>0$ and generator $(A,D(A))$. 
Then the following assertions hold.
\begin{enumerate}[label=\upshape(\alph*), leftmargin=*]
\item \label{it:res_per_1} $\sigma^{\ast}(A)\subseteq \frac{2\pi \mathsf{i}}{\rho}\Z$ and
\[
R(\mu,A)x=(1-\e^{-\mu \rho})^{-1}\int_{0}^{\rho}\e^{-\mu s}T(s)x\d s 
\]
for all $\mu\notin \frac{2\pi \mathsf{i}}{\rho}\Z$ and $x\in X$.
\item \label{it:res_per_2} The map $R(\cd,A)\colon \C\setminus  \frac{2\pi \mathsf{i}}{\rho}\Z\to \cL_{\ob}(X)$ 
is holomorphic and with $\mu_n\coloneqq \frac{2\pi \mathsf{i}n}{\rho}$ the limit 
$P_n\coloneqq \lim_{\mu\to\mu_n}(\mu-\mu_n)R(\mu,A)$ exists in $\cL_{\ob}(X)$ and
\[
P_n x=\frac{1}{\rho}\int_{0}^{\rho}\e^{-\mu_n s}T(s)x\d s 
\]
for all $n\in\Z$ and $x\in X$.
\item \label{it:res_per_3} For all $n\in\Z$ and $x\in X$ we have the Laurent series expansion
\[ 
R(\mu,A)x=\sum_{k=-1}^{\infty}a_{k,n}(x)(\mu-\mu_n)^k,\quad 0<|\mu-\mu_n |<\frac{2\pi}{\rho},
\]
which converges locally uniformly in $\mu$ in the topology of $X$ and where 
\begin{align*}
a_{k,n}(x)&=\frac{1}{2\pi \mathsf{i}}\int_{\partial B(\mu_n ,r)}\frac{R(\lambda,A)x}{(\lambda-\mu_n)^{k+1}}\d\lambda\in X,
\quad k\in\N_0 ,\\
a_{-1,n}(x)&=\frac{1}{2\pi \mathsf{i}}\int_{\partial B(\mu_n ,r)}R(\lambda,A)x\d\lambda=P_n x
\end{align*}
for all $0<r<\frac{2\pi}{\rho}$ and the boundary $\partial B(\mu_n ,r)=\{z\in\C\;|\;|z-\mu_n|=r\}$ 
is positively oriented. Further, $a_{k,n}\in\cL(X)$ for all $k\in\N_{0}\cup\{-1\}$ and $n\in\Z$. 

If in addition $\cL_{\ob}(X)$ is sequentially complete, then the Laurent series and the Riemann integral for the coefficients 
$a_{k,n}$ can be taken w.r.t.~the topology of $\cL_{\ob}(X)$.
\item \label{it:res_per_4} $\ran (P_n)=\ker (\mu_n - A)$ and $P_n$ is a projection for all $n\in\Z$.
\item \label{it:res_per_5} $\sigma^{\ast}(A)=\sigma_{\p}(A)\subseteq \frac{2\pi \mathsf{i}}{\rho}\Z$.
\end{enumerate}
\end{prop}
\begin{proof}
\ref{it:res_per_1} Choosing $\lambda\coloneqq \mu\notin \frac{2\pi \mathsf{i}}{\rho}\Z$ and $t\coloneqq \rho$, 
we get $1-\e^{-\mu \rho}\neq 0$ and by \prettyref{prop:rescale_identities} that
\[
 (1-\e^{-\mu \rho})x
=-\bigl(\e^{-\mu \rho}T(\rho)x-x\bigr)
\underset{\mathclap{\eqref{eq:rescale_iden_1}}}{=}-(A-\mu)\int_{0}^{\rho}\e^{-\mu s}T(s)x \d s
\]
for all $x\in X$, and
\[
(1-\e^{-\mu \rho})x\underset{\mathclap{\eqref{eq:rescale_iden_2}}}{=}-\int_{0}^{\rho}\e^{-\mu s}T(s)(A-\mu)x \d s
\]
for all $x\in D(A)$. This implies 
\[
x=(\mu-A)(1-\e^{-\mu \rho})^{-1}\int_{0}^{\rho}\e^{-\mu s}T(s)x \d s
\]
for all $x\in X$, and 
\[
x = (1-\e^{-\mu \rho})^{-1}\int_{0}^{\rho}\e^{-\mu s}T(s)(\mu-A)x \d s .
\]
for all $x\in D(A)$. 
By the first equality above $\mu-A \colon D(A)\to X$ is surjective and by the second injective. Thus we have 
\[
(\mu-A)^{-1}x=(1-\e^{-\mu \rho})^{-1}\int_{0}^{\rho}\e^{-\mu s}T(s)x \d s
\]
for all $x\in X$. Let $q\in\Gamma_X$. By the local equicontinuity there are $p\in\Gamma_X$ and $C\geq 0$ such that 
for all $x\in X$
\begin{align}\label{eq:res_per_0}
      q((\mu-A)^{-1}x)
&\leq |1-\e^{-\mu \rho}|^{-1}\rho\e^{|\re(\mu)|\rho}\sup_{s\in [0,\rho]}q(T(s)x)\nonumber\\
&\leq C|1-\e^{-\mu \rho}|^{-1}\rho\e^{|\re(\mu)|\rho}p(x).
\end{align}
Thus $(\mu-A)^{-1}\in\mathcal{L}(X)$ for all $\mu\in\C\setminus\frac{2\pi\mathsf{i}}{\rho}\Z$. 
Since $\C\setminus\frac{2\pi\mathsf{i}}{\rho}\Z$ is open, there is $\delta>0$ such that 
$\overline{B(\mu,\delta)}\subset \C\setminus\frac{2\pi\mathsf{i}}{\rho}\Z$ and 
\[
 q((\eta-A)^{-1}x)
\underset{\mathclap{\eqref{eq:res_per_0}}}{\leq} C|1-\e^{-\eta \rho}|^{-1}\rho\e^{|\re(\eta)|\rho}p(x)
\leq C\rho\sup_{s\in \overline{B(\mu,\delta)}}|1-\e^{-s \rho}|^{-1}\e^{|\re(s)|\rho}p(x)
\]
for all $x\in X$. Hence $\{R(\eta,A)\;|\;\eta\in B(\mu,\delta)\}$ is equicontinuous in $\cL(X)$, 
which finishes the proof of part \ref{it:res_per_1}.

\ref{it:res_per_2} First, we show that $R(\cd,A)$ is holomorphic. 
We note that $\C\setminus\frac{2\pi\mathsf{i}}{\rho}\Z$ 
is open and that it suffices to show that the map $f\colon \C\to \cL_{\ob}(X)$ 
defined by 
$f(\mu)x\coloneqq \int_{0}^{\rho}\e^{-\mu s}T(s)x\d s$, $x\in X$, is holomorphic. 
Let $\mu,\eta\in\C$. Let $q\in\Gamma_X$ and $M\subseteq X$ be bounded. By the local equicontinuity of the semigroup and 
the boundedness of $M$ there are $p\in\Gamma_X$ and $C_0,C_1\geq 0$ such that $q(T(s)x)\leq C_0 p(x)$ for all 
$s\in [0,\rho]$ and $x\in X$, and $\sup_{x\in M} p(x)\leq C_1$. We deduce that
\begin{align*}
 \sup_{x\in M}q(f(\mu)x-f(\eta)x)
&=\sup_{x\in M}q\Bigl(\int_{0}^{\rho}(\e^{-\mu s}-\e^{-\eta s})T(s)x\d s \Bigr)\\
&\leq \rho \sup_{x\in M} \sup_{s\in [0,\rho]}|\e^{-\mu s}-\e^{-\eta s}|q(T(s)x)\\
&\leq C_0 \rho \sup_{s\in [0,\rho]}|\e^{-\mu s}-\e^{-\eta s}|\sup_{x\in M}p(x)
 \leq C_0 C_1 \rho\sup_{s\in [0,\rho]}|\e^{-\mu s}-\e^{-\eta s}| .
\end{align*}
This implies that $f$ is continuous. 

Now, let $h\in\C$ such that $0<|h|<\frac{2\pi}{\rho}$. 
Then we have
\[
 \frac{f(\mu+h)x-f(\mu)x}{h}-\int_{0}^{\rho}-s\e^{-\mu s}T(s)x\d s 
=\int_{0}^{\rho}\Bigl(\frac{\e^{-(\mu+h) s}-\e^{-\mu s}}{h}+s\e^{-\mu s}\Bigr)T(s)x\d s 
\]
for all $x\in X$. Again the local equicontinuity of the semigroup and 
the boundedness of $M$ yield that
\[
    q\Bigl(\int_{0}^{\rho}-s\e^{-\mu s}T(s)x\d s\Bigr)
\leq \rho^2 \e^{|\re(\mu)|\rho}\sup_{s\in [0,\rho]}q(T(s)x)
\leq C_0\rho^2 \e^{|\re(\mu)|\rho}p(x)
\]
for all $x\in X$ and 
\begin{flalign*}
&\phantom{\leq}\sup_{x\in M}q\left(\int_{0}^{\rho}\Bigl(\frac{\e^{-(\mu+h) s}-\e^{-\mu s}}{h}+s\e^{-\mu s}\Bigr)
 T(s)x\d s \right)\\
&\leq\rho \sup_{s\in [0,\rho]}\Bigl|\frac{\e^{-(\mu+h) s}-\e^{-\mu s}}{h}+s\e^{-\mu s}\Bigr|
       \sup_{x\in M}\sup_{s\in [0,\rho]}q(T(s)x) \\
&\leq C_0\rho\sup_{s\in [0,\rho]}\Bigl|\e^{-\mu s}\Bigr|\Bigl|\frac{\e^{-h s}-1}{h}+s\Bigr|
       \sup_{x\in M}p(x) \\
&\leq C_0 C_1 \rho\e^{|\re(\mu)|\rho}\sup_{s\in [0,\rho]}\frac{1}{|h|}\Bigl|\sum_{k=2}^{\infty}\frac{1}{k!}(-hs)^k\Bigr| 
 \leq C_0 C_1 \rho\e^{|\re(\mu)|\rho}\sup_{s\in [0,\rho]}|h||s|^2\sum_{k=0}^{\infty}\frac{1}{k!}|hs|^{k}\\
&\leq C_0 C_1 \rho^3 \e^{|\re(\mu)|\rho+2\pi} |h|. 
\end{flalign*}
Letting $h\to 0$, this yields that $f$ is holomorphic, $f'(\mu)x=\int_{0}^{\rho}-s\e^{-\mu s}T(s)x\d s$ for 
all $x\in X$ and $f'(\mu)\in\cL(X)$. 
Thus, $R(\cd,A)$ is holomorphic as well. Further, by l'H\^opital's rule and the continuity of $f$ on $\C$ we get 
for all $n\in\Z$ that
\[
 P_n
=\lim_{\mu\to\mu_n}(\mu-\mu_n)R(\mu,A)
\underset{\mathclap{\text{\ref{it:res_per_1}}}}{=}\lim_{\mu\to\mu_n}\frac{\mu-\mu_n}{1-\e^{-\mu \rho}}f(\mu)
=\frac{1}{\rho \e^{-\mu_n \rho}}f(\mu_n)
=\frac{1}{\rho}f(\mu_n)
\]
where the limit is taken in $\cL_{\ob}(X)$. Further, this implies for all $x\in X$ that
\[
P_n x = \frac{1}{\rho}f(\mu_n)x = \frac{1}{\rho}\int_{0}^{\rho}\e^{-\mu_n s}T(s)x\d s.
\]
\ref{it:res_per_3} Let $x\in X$. By part \ref{it:res_per_2} the map $
R(\cd,A)x\colon\C\setminus\frac{2\pi\mathsf{i}}{\rho}\Z\to X$ is holomorphic and has poles of order at most $1$ 
at each $\mu_n$ for $n\in\Z$. This implies the first part of our statement by \cite[p.~274--275]{jorda2005} 
(cf.~\cite[p.~243]{kaballo2014} in the case that $X$ is quasi-complete). 
In particular, the formula for the coefficients $a_{k,n}(x)$, $k\in\N_{0}\cup\{-1\}$, follows from 
\cite[Eq.~(3), p.~275]{jorda2005} where we note that the integrals in our formulas for the 
$a_{k,n}(x)$ are Riemann integrals in $X$, existing by \cite[Theorem 10, p.~317]{albanese2012}, which coincide with the 
Pettis integrals in \cite[Eq.~(3), p.~275]{jorda2005}. The identity $a_{-1,n}(x)=P_n x$ is a consequence of 
part \ref{it:res_per_2} and the Laurent series expansion. Furthermore, for $q\in\Gamma_X$ there are $p\in\Gamma_X$ 
and $C\geq 0$ by \eqref{eq:res_per_0} such that for $0<r<\frac{2\pi}{\rho}$  
\begin{align*}
  q(a_{k,n}(x))
&=\frac{1}{2\pi}q\Bigl(\int_{\partial B(\mu_n ,r)}\frac{R(\lambda,A)x}{(\lambda-\mu_n)^{k+1}}\d\lambda\Bigr)
 \leq \frac{1}{r^k}\sup_{\lambda\in \partial B(\mu_n ,r)}q(R(\lambda,A)x)\\
&\underset{\mathclap{\eqref{eq:res_per_0}}}{\leq} \frac{C}{r^k}\sup_{\lambda\in \partial B(\mu_n ,r)}|1-\e^{-\lambda \rho}|^{-1}\rho\e^{|\re(\lambda)|\rho}p(x)
\end{align*}
for all $x\in X$. Thus $a_{k,n}\in\cL(X)$.

The addendum in the case that $\cL_{\ob}(X)$ is sequentially complete also follows from part \ref{it:res_per_2} 
and \cite[p.~274--275]{jorda2005} where the sequential completeness of $\cL_{\ob}(X)$ guarantees the convergence 
of the Laurent series and the existence of the Riemann integrals in our formulas for the $a_{k,n}$ w.r.t.~the topology 
of $\cL_{\ob}(X)$. 

\ref{it:res_per_4} Let $n\in\Z$ and $0<r<\frac{2\pi}{\rho}$. First, we remark that 
$R(\lambda,A)x\in D(A)=D(A-\mu_n)$ and 
\begin{align}\label{eq:res_per_1}
 (A-\mu_n)R(\lambda,A)x
&=(A-\lambda+\lambda-\mu_n)R(\lambda,A)x\nonumber\\
&=-x+(\lambda-\mu_n)R(\lambda,A)x\\
&=R(\lambda,A)(A-\lambda+\lambda-\mu_n)x
 =R(\lambda,A)(A-\mu_n)x\nonumber
\end{align}
for all $\lambda\in\rho(A)$ and $x\in X$. This implies
\begin{align}\label{eq:res_per_2}
  (A-\mu_n)a_{0,n}(x)
&=(A-\mu_n)\frac{1}{2\pi \mathsf{i}}\int_{\partial B(\mu_n ,r)}\frac{R(\lambda,A)x}{\lambda-\mu_n}\d\lambda\nonumber\\
&=\frac{1}{2\pi \mathsf{i}}\int_{\partial B(\mu_n ,r)}
  (A-\mu_n)\frac{R(\lambda,A)x}{\lambda-\mu_n}\d\lambda
 \underset{\mathclap{\eqref{eq:res_per_1}}}{=}-x+\frac{1}{2\pi \mathsf{i}}\int_{\partial B(\mu_n ,r)} 
  R(\lambda,A)x\d\lambda\nonumber\\
&=-x+P_n x
\end{align}
for all $x\in X$ by \cite[Remark 3.7, p.~51--52]{kruse_schwenninger2025} and part \ref{it:res_per_3}. 
Moreover, we get from \eqref{eq:res_per_1} and \eqref{eq:res_per_2} that
\begin{equation}\label{eq:res_per_3}
(A-\mu_n)a_{0,n}(x)=a_{0,n}((A-\mu_n)x)
\end{equation}
for all $x\in X$. Now, we prove $\ran (P_n)=\ker (\mu_n - A)$.

``$\subseteq$'' Let $x\in X$. Then we have $P_n x\in D(A)$ by \prettyref{prop:rescale_identities}, 
and by part \ref{it:res_per_2} and \prettyref{rem:periodic_strongly} \ref{it:period_strong_2} that
\begin{equation}\label{eq:res_per_4}
 (A-\mu_n)P_n x
=(A-\mu_n)\frac{1}{\rho}\int_{0}^{\rho}\e^{-\mu_n s}T(s)x\d s 
\underset{\mathclap{\eqref{eq:rescale_iden_1}}}{=}\frac{1}{\rho}\Bigl(\e^{-\mu_n \rho}T(\rho)x-x\Bigr)
=\frac{1}{\rho}(x-x)
=0,
\end{equation}
implying $P_n x\in \ker (A-\mu_n)$.

``$\supseteq$'' Let $x\in \ker (A-\mu_n)$. Then we have 
\[
0=a_{0,n}(0)
 =a_{0,n}((A-\mu_n)x)
 \underset{\mathclap{\eqref{eq:res_per_3}}}{=}(A-\mu_n)a_{0,n}(x)
 \underset{\mathclap{\eqref{eq:res_per_2}}}{=}-x+P_n x .
\]
We conclude that $x=P_n x\in\ran (P_n)$. Thus $\ker (A-\mu_n)=\ran(P_n)$ and in particular $\id=P_n$ on $\ran(P_n)$, 
so $P_n$ is a projection.

\ref{it:res_per_5} We only need to show that $\sigma^{\ast}(A)\subseteq\sigma_{\p}(A)$. Due to part \ref{it:res_per_1} 
we know that $\sigma^{\ast}(A)\subseteq\frac{2\pi\mathsf{i}}{\rho}\Z$. Let $\lambda\in\sigma^{\ast}(A)$. 
Then there is $n\in\Z$ such that $\lambda=\frac{2\pi\mathsf{i}n}{\rho}=\mu_n$. 
Suppose that $\mu_n\notin\sigma_{\p}(A)$. 
By part \ref{it:res_per_4} this means that $\ran (P_n)=\ker (\mu_n - A)=\{0\}$. Hence we have 
\[
 a_{0,n}((\mu_n -A)x)
\underset{\mathclap{\eqref{eq:res_per_3}}}{=}(\mu_n-A)a_{0,n}(x)
\underset{\mathclap{\eqref{eq:res_per_2}}}{=}x-P_n x
=x-0
=x 
\]
for all $x\in X$. By part \ref{it:res_per_3} we know that $a_{0,n}\in\cL(X)$. 
We deduce that $\mu_n\in\rho(A)$ and $R(\mu_n,A)=a_{0,n}$. 
Since $P_n x= 0$, the map $R(\cd,A)x\colon\C\setminus\frac{2\pi \mathsf{i}}{\rho}\Z\to X$ has a holomorphic extension 
at $\mu=\mu_n$ with value $a_{0,n}(x)=R(\mu_n,A)x$ for all $x\in X$ by parts \ref{it:res_per_2} and \ref{it:res_per_3}.  
So, $R(\cd,A)x$ is holomorphic 
on $B(\mu_n,\frac{2\pi}{\rho})$ for all $x\in X$. Let $q\in \Gamma_X$ and $0<r<\frac{2\pi}{\rho}$. 
Due to Cauchy's inequality (see e.g.~\cite[Proposition 2.5, p.~57]{dineen1981}) and \eqref{eq:res_per_0}
there are $p\in\Gamma_X$ and $C\geq 0$ such that for all $\mu\in B(\mu_n,r)$ and all $x\in X$
\[
q(R(\mu,A)x)\leq \max_{\eta\in\partial B(\mu_n, r)} q(R(\eta,A)x)
\underset{\mathclap{\eqref{eq:res_per_0}}}{\leq} C\rho\max_{\eta\in\partial B(\mu_n, r)}|1-\e^{-\eta \rho}|^{-1}
 \e^{|\re(\eta)|\rho}p(x).
\]
We deduce that $\{R(\mu,A)\;|\;\mu\in B(\mu_n,r)\}$ is equicontinuous in $\cL(X)$. 
This implies $\lambda=\mu_n\in\rho^{\ast}(A)$, which is a contradiction. Thus $\lambda\in\sigma_{\p}(A)$.
\end{proof}

The following remark is relevant in view of the addendum of \prettyref{prop:resolvent_periodic} \ref{it:res_per_3} 
and \ref{it:res_per_5}.

\begin{rem}\label{rem:space_cont_lin_seq_compl}
Let $X$ be a Hausdorff locally convex space. 
\begin{enumerate}[label=\upshape(\alph*), leftmargin=*]
\item \label{it:cont_lin_seq_compl_1} Let $X$ be sequentially complete. Then $ \cL_{\ob}(X)$ is sequentially complete 
if $X$ is barrelled (see \cite[(1.8) Proposition, (1.9), p.~164--165]{dierolf1985}, 
cf.~\cite[Remark 3.5 (ii), p.~262]{albanese2013}), or if $X$ is a gDF space (see the comment directly after 
\cite[12.4.2 Theorem, p.~258]{jarchow1981}). 
\item \label{it:cont_lin_seq_compl_2} If $\cL_{\ob}(X)$ is sequentially complete, 
then $X$ is sequentially complete. 
Indeed, by \cite[Chap.~8, \S 39.1, (2'), p.~132]{koethe1979} $X$ is topologically isomorphic to a complemented subspace 
of $\cL_{\ob}(X)$. Since complemented subspaces are closed by \cite[p.~77]{jarchow1981}, $X$ can be considered as 
a closed subspace of the sequentially complete space $\cL_{\ob}(X)$. Therefore $X$ is also sequentially complete.
\item \label{it:cont_lin_seq_compl_3} If $X$ is sequentially complete and 
$(T(t))_{t\geq 0}$ a periodic strongly continuous locally equicontinuous semigroup on $X$ with generator 
$(A,D(A))$, then 
\[
\sigma^{\ast}(A)=\sigma(A)=\sigma_{\apr}(A)=\sigma_{\ap}(A)=\sigma_{\ap}^{\seq}(A)=\sigma_{\bap}(A)
=\sigma_{\bap}^{\seq}(A)=\sigma_{\p}(A)
\]
by \prettyref{rem:point_spec_included} \ref{it:p_spec_included_1} and 
\prettyref{prop:resolvent_periodic} \ref{it:res_per_5}. 
\end{enumerate}
\end{rem}

We note that we adjusted the proof of \cite[Chap.~8, \S 39.6, (2a), p.~143]{koethe1979} 
to prove \prettyref{rem:space_cont_lin_seq_compl} \ref{it:cont_lin_seq_compl_2}.

\begin{exa}\label{ex:periodic_sg}
We denote by $\mathcal{H}(\D)$ the space of $\C$-valued holomorphic functions on the open unit disc 
$\D\coloneqq \{z\in\C\;|\;|z|<1\}$ and define the \emph{Hardy space} of bounded holomorphic functions by 
\[
H^{\infty}\coloneqq\{f\in\mathcal{H}(\D)\;|\;\|f\|_{\infty}\coloneqq \sup_{z\in\D}|f(z)|<\infty\}.
\]
Further, we denote by $\tau_{\operatorname{co}}$ the \emph{compact-open topology} on $H^{\infty}$, i.e.~the topology 
of uniform convergence on compact subsets of $\D$. By \cite[I.1.27 Remark, p.~19]{cooper1978} and 
\cite[V.1.1 Proposition 1), 4), p.~226--227]{cooper1978} 
the triple $(H^{\infty},\|\cd\|_{\infty},\tau_{\operatorname{co}})$ is a Saks space and 
$(H^{\infty},\gamma(\|\cd\|_{\infty},\tau_{\operatorname{co}}))$ is a complete semi-Montel gDF space. 
In particular, $\cL_{\ob}(H^{\infty},\gamma(\|\cd\|_{\infty},\tau_{\operatorname{co}}))$ is sequentially complete 
by \prettyref{rem:space_cont_lin_seq_compl} \ref{it:cont_lin_seq_compl_1}.
Furthermore, the system of seminorms $(|\cd|_{\nu})_{\nu\in\mathcal{C}_{0}(\D)}$ given by
\[
|f|_{\nu}\coloneqq\sup_{z\in\D}|f(z)\nu(z)|,\quad f\in H^{\infty},
\]
for $\nu\in\mathcal{C}_{0}(\D)$ induces $\gamma(\|\cd\|_{\infty},\tau_{\operatorname{co}})$ 
by \cite[p.~227]{cooper1978}. Here, $\mathcal{C}_{0}(\D)$ denotes the space of $\C$-valued continuous functions 
on $\D$ that vanish at infinity. 
 
The composition semigroup $(T(t))_{t\geq 0}$ given by 
\[
T(t)f(z)\coloneqq f(\e^{\mathsf{i} t}z),\quad t\geq 0,\,f\in H^{\infty},\,z\in\D ,
\]
is strongly continuous and locally equicontinuous w.r.t~$\gamma(\|\cd\|_{\infty},\tau_{\operatorname{co}})$ 
by \cite[7.3 Corollary, p.~42]{kruse2024b} and clearly $2\pi$-periodic. 
We note that this semigroup is not strongly continuous w.r.t.~$\|\cd\|_{\infty}$ by 
\cite[Theorem 1.1, p.~844]{anderson2017}. Due to \cite[7.4 Theorem (c), p.~42--43]{kruse2024b} 
its generator $(A,D(A))$ fulfils $Af(z)=\mathsf{i}zf'(z)$, $z\in\D$, for $f\in D(A)$ and 
\[
D(A)=\{f\in H^{\infty}\;|\;(z\mapsto \mathsf{i}zf'(z))\in H^{\infty}\}.
\]
We claim that $\sigma_{\p}(A)=\mathsf{i}\N_{0}$. First, we observe that the monomials $f_n\colon \D\to \C$, 
$f_n(z)\coloneqq z^n$, belong to $D(A)$ for every $n\in\N_0$. Then we have $Af_{n}(z)=\mathsf{i}nf_{n}(z)$ for all 
$z\in\D$, so $\mathsf{i} n\in \sigma_{\p}(A)$ for every $n\in\N_0$. Second, let us take a look at the 
converse inclusion ``$\subseteq$''. Let $\lambda\in\sigma_{\p}(A)$ and $f\in D(A)$, $f\neq 0$, such that $Af=\lambda f$. 
Thus $\mathsf{i}zf'(z)=\lambda f(z)$ for all $z\in\D$. By taking derivatives on both sides and induction 
we get $\mathsf{i}zf^{(k)}(z)=(\lambda-(k-1)\mathsf{i})f^{(k-1)}(z)$ for all $z\in\D$ and $k\in\N$. 
Evaluating this equation at $z=0$, we obtain that $\lambda=(k-1)\mathsf{i}$ for some $k\in\N$ or 
$f^{(k-1)}(0)=0$ for all $k\in\N$. However, the latter case implies that $f=0$ on $\D$ since $f$ is holomorphic, 
which is a contradiction. Hence there is some $k\in\N$ such that $\lambda=(k-1)\mathsf{i}$, proving our claim.
By \prettyref{prop:resolvent_periodic} \ref{it:res_per_5} 
we get that $\sigma^{\ast}(A)=\sigma_{\p}(A)=\mathsf{i}\N_{0}$. 
The eigenspaces of $A$ are $\ker(\mathsf{i}n-A)=\operatorname{span}\{f_n\}$ for all $n\in\N_0$, 
so in particular one-dimensional. Indeed, we already know that the inclusion ``$\supseteq$'' holds. We turn 
to the converse inclusion ``$\subseteq$''. Let $n\in\N_{0}$ and $f\in \ker(\mathsf{i}n-A)$. 
By \prettyref{prop:resolvent_periodic} \ref{it:res_per_4} we have $\ker(\mathsf{i}n-A)=\ran(P_n)$ 
with $P_n$ from \prettyref{prop:resolvent_periodic} \ref{it:res_per_2}. Thus there 
is $g\in H^{\infty}$ such that $f=P_n g$.
Since $g$ is holmorphic on $\D$, it has the locally uniformly convergent power series representation 
\[
g(z)=\sum_{k=0}^{\infty}\frac{g^{(k)}(0)}{k!}z^{k},\quad z\in\D.
\]
 This implies by \prettyref{prop:resolvent_periodic} \ref{it:res_per_2} that 
\begin{align*}
  2\pi P_n g(z) 
&=\int_{0}^{2\pi}\e^{-\mathsf{i}n s}g(\e^{\mathsf{i}s}z)\d s
 =\int_{0}^{2\pi}\sum_{k=0}^{\infty}\frac{g^{(k)}(0)}{k!}\e^{-\mathsf{i}n s}\e^{\mathsf{i}ks}z^{k}\d s \\
&=\sum_{k=0}^{\infty}\frac{g^{(k)}(0)}{k!}\int_{0}^{2\pi}\e^{\mathsf{i}(k-n)s}\d s\,z^{k}
 =2\pi \frac{g^{(n)}(0)}{n!} z^{n}
\end{align*}
for all $z\in \D$ where the swap of the integral and the series in the third equation is justified by 
the locally uniform convergence of the power series of $g$. Hence we have $f=P_n g= \frac{g^{(n)}(0)}{n!} f_n$, 
proving our claim.
\end{exa}
 
The preceding results allow us now to fully characterise periodic strongly continuous locally equicontinuous 
semigroups. The proof of this characterisation is a modification of the one of 
\cite[Chap.~IV, 2.26 Theorem, p.~267]{engel_nagel2000}, which covers the case of Banach spaces.

\begin{thm}\label{thm:nec_suff_cond_periodic}
Let $X$ be a sequentially complete Hausdorff locally convex space 
and $(T(t))_{t\geq 0}$ a strongly continuous locally equicontinuous semigroup on $X$ with generator $(A,D(A))$. 
Then the following assertions are equivalent.
\begin{enumerate}[label=\upshape(\alph*), leftmargin=*]
\item \label{it:period_equiv_1} $(T(t))_{t\geq 0}$ is periodic. 
\item \label{it:period_equiv_2} $\sigma^{\ast}(A)=\sigma_{\p}(A)\subseteq 2\pi \mathsf{i}\alpha\Z$ for some $\alpha>0$, and if $X\neq\{0\}$, then the corresponding eigenvectors span a dense subspace of $X$.
\end{enumerate}
\end{thm}
\begin{proof}
``\ref{it:period_equiv_2}$\Rightarrow$\ref{it:period_equiv_1}'' This implication follows 
from \prettyref{prop:suff_cond_periodic}.

``\ref{it:period_equiv_1}$\Rightarrow$\ref{it:period_equiv_2}'' Let  $(T(t))_{t\geq 0}$ be periodic with period 
$\rho\geq 0$. If $\rho=0$, then we have by \prettyref{prop:period_zero} that 
$T(t)=\id$ for all $t\geq 0$, $A=0$ with $D(A)=X$ and $\ker(A)=X$ as well as 
$\sigma^{\ast}(A)=\sigma_{\p}(A)=\{0\}$ if $X\neq\{0\}$, and all spectra are empty if $X=\{0\}$. 
Thus \ref{it:period_equiv_2} holds for any $\alpha>0$. 
Now, let us consider the case $\rho>0$. In particular, this yields that $X\neq\{0\}$. 
Due to \prettyref{prop:resolvent_periodic} \ref{it:res_per_4} 
and \ref{it:res_per_5} we can choose $\alpha\coloneqq \frac{1}{\rho}$ and it is only left to show that 
\[
\overline{\operatorname{span}}\left(\bigcup_{n\in\Z}P_n X\right)=X
\]
with $P_n$ from \prettyref{prop:resolvent_periodic} \ref{it:res_per_2}. 
Suppose that the span above is not dense in $X$. Then there is $x'\in X'$, $x'\neq 0$, such that $x'(P_n x)=0$ 
for all $x\in X$ and $n\in\Z$ by the bipolar theorem. W.l.o.g~we may assume that $\rho=2\pi$ 
(otherwise we consider the $2\pi$-periodic semigroup $(T(\frac{\rho}{2\pi}t))_{t\geq 0}$ instead, 
see \prettyref{prop:rescale_sg} with $\lambda\coloneqq 0$ and $c\coloneqq\frac{\rho}{2\pi}$, 
which fulfils $\sigma_{\p}(cA)=c\sigma_{\p}(A)$ and $\ker(c\mu-cA)=\ker(\mu-A)$ for all $\mu\in\sigma_{\p}(A)$). 
Using \prettyref{rem:periodic_strongly} \ref{it:period_strong_1} and \ref{it:period_strong_2}, 
we define for $x\in D(A)$ the $2\pi$-periodic function $f_{x,x'}\colon\R\to\C$, $f_{x,x'}(t)\coloneqq x'(T(t)x)$, 
which is continuously differentiable by \cite[Proposition 1.2 (1), p.~260]{komura1968}. 
Then $f_{x,x'}$ coincides with its Fourier series (see e.g.~\cite[\S 23, Satz 3, p.~321]{forster2016}) and we have by 
\prettyref{prop:resolvent_periodic} \ref{it:res_per_2} that
\[
 f_{x,x'}(t)
=\sum_{n\in\Z}\biggl(\frac{1}{2\pi}\int_{0}^{2\pi}\e^{-\mathsf{i}ns}x'(T(s)x)\d s\biggr)\e^{\mathsf{i}nt}
=\sum_{n\in\Z}x'(P_n x) \e^{\mathsf{i}nt}
=0
\]
for all $t\in\R$ and the series converges uniformly. Since $D(A)$ is dense in $X$ by 
\cite[Proposition 1.3, p.~261]{komura1968} and $x'\neq 0$, there is $\widetilde{x}\in D(A)$ such that 
$x'(\widetilde{x})\neq 0$. However, this implies that 
\[
0=f_{\widetilde{x},x'}(0)=x'(\widetilde{x})\neq 0,
\]
which is a contradiction.
\end{proof}

Next, \prettyref{prop:resolvent_periodic} and \prettyref{thm:nec_suff_cond_periodic} enable us to lift 
\cite[Chap.~IV, 2.27 Theorem, p.~267]{engel_nagel2000} from Banach spaces to quasi-complete 
Hausdorff locally convex spaces.

\begin{thm}\label{thm:fourier_expansion}
Let $X$ be a quasi-complete Hausdorff locally convex space and $(T(t))_{t\geq 0}$ a periodic strongly continuous 
locally equicontinuous semigroup on $X$ with period $\rho>0$ and generator $(A,D(A))$. 
Then we have for every $x\in D(A)$ that the sequence $(P_n x)_{n\in\Z}$ is summable to $x$, so $x=\sum_{n\in\Z}P_n x$, 
with $\mu_n =\frac{2\pi \mathsf{i}n}{\rho}$ and
\[
P_n x=\frac{1}{\rho}\int_{0}^{\rho}\e^{-\mu_n s}T(s)x\d s, \quad n\in\Z,\,x\in X, 
\]
from \prettyref{prop:resolvent_periodic} \ref{it:res_per_2}. In particular, we have
\begin{align}
T(t)x&=\sum_{n\in\Z}\e^{\mu_n t}P_n x, && x\in D(A),\,t\geq 0,\label{eq:fourier_1}\\
Ax&=\sum_{n\in\Z}\mu_n P_n x, && x\in D(A^2).\label{eq:fourier_2}
\end{align}
\end{thm}
\begin{proof}
W.l.o.g.~we may assume that $\rho=2\pi$ (see the proof of \prettyref{thm:nec_suff_cond_periodic}). 
Let $x\in D(A)$ and set $y\coloneqq Ax$. We start with showing that $(P_n x)_{n\in\Z}$ is summable to $x$. 
Due to Equations \eqref{eq:rescale_iden_2} and \eqref{eq:res_per_4} from \prettyref{prop:rescale_identities} 
and from the proof of \prettyref{prop:resolvent_periodic} \ref{it:res_per_4}, respectively, we observe that 
\begin{equation}\label{eq:fourier_exp_1}
 P_n y
=P_n Ax
\underset{\mathclap{\eqref{eq:rescale_iden_2}}}{=}AP_n x
\underset{\mathclap{\eqref{eq:res_per_4}}}{=}\,\mu_n P_n x
=\mathsf{i}n P_n x.
\end{equation}
Furthermore, if $\mu_n\in\sigma_{\p}(A)\subseteq\frac{2\pi\mathsf{i}}{\rho}\Z$, 
then it follows from \prettyref{prop:resolvent_periodic} \ref{it:res_per_4} 
and the proof of \prettyref{prop:suff_cond_periodic} with $\alpha\coloneqq\frac{1}{\rho}$ that
\begin{equation}\label{eq:fourier_exp_2} 
T(t)P_n w=\e^{\mu_n t}P_n w=\e^{\mathsf{i}n t}P_n w
\end{equation}
for all $w\in X$. The same equality holds if $\mu_n\not\in \sigma_{\p}(A)$ because then $P_n w=0$ for all $w\in X$ 
by \prettyref{prop:resolvent_periodic} \ref{it:res_per_4}.

Let $M\subset\Z\setminus\{0\}$ be finite and $q\in\Gamma_X$. We set $U_{q}\coloneqq\{x\in X\;|\;q(x)<1\}$ and 
denote by $U_{q}^{\circ}\subseteq X'$ the polar of $U_q$. For all $x'\in U_{q}^{\circ}$ we have by the 
Cauchy--Schwarz inequality that
\begin{align*}
 \bigl|x'\bigl(\sum_{n\in M}P_n x\bigr)\bigr|
&=\bigl|\sum_{n\in M}x'(P_n x)\bigr|
\,\underset{\mathclap{\eqref{eq:fourier_exp_1}}}{=}\,\bigl|\sum_{n\in M}(\mathsf{i}n)^{-1}x'(P_n y)\bigr|\\
&\leq \bigl(\sum_{n\in M}n^{-2}\bigr)^{\frac{1}{2}}\bigl(\sum_{n\in M}|x'(P_n y)|^{2}\bigr)^{\frac{1}{2}}.
\end{align*}
Let us turn to the second factor on the right-hand side. By applying the Bessel inequality to the 
$2\pi$-periodic continuous function $f\colon\R\to\C$, $f(t)\coloneqq x'(T(t)y)$, whose Fourier 
coefficients are given by $x'(P_n y)$, $n\in\Z$, due to \prettyref{prop:resolvent_periodic} \ref{it:res_per_2}, 
we obtain
\[
    \sum_{n\in M}|x'(P_n y)|^{2}
\leq\frac{1}{2\pi}\int_{0}^{2\pi}|x'(T(s)y)|^2\d s
\leq \sup_{s\in[0,2\pi]}|x'(T(s)y)|^2 .
\]
We deduce that
\begin{align*}
  q\bigl(\sum_{n\in M}P_n x\bigr)
&=\sup_{x'\in U_{q}^{\circ}}\bigl|x'\bigl(\sum_{n\in M}P_n x\bigr)\bigr|
 \leq\bigl(\sum_{n\in M}n^{-2}\bigr)^{\frac{1}{2}}\sup_{x'\in U_{q}^{\circ}}\sup_{s\in[0,2\pi]}|x'(T(s)y)|\\
&=\bigl(\sum_{n\in M}n^{-2}\bigr)^{\frac{1}{2}}\sup_{s\in[0,2\pi]}q(T(s)y)
\end{align*}
where we used \cite[Proposition 22.14, p.~256]{meisevogt1997} in the first and last equation to get from $q$ to 
$\sup_{x'\in U_{q}^{\circ}}$ and back. Denoting by $F(\Z)$ the family of finite subsets of $\Z$, 
this estimate implies that the net $(\sum_{n\in M}P_n x)_{M\in F(\Z)}$ is a bounded Cauchy net in $X$ and so 
convergent since $X$ is quasi-complete. Therefore $(P_n x)_{n\in\Z}$ is summable (see \cite[p.~120]{schaefer1971})
and so $z\coloneqq \sum_{n\in\Z}P_n x \in X$.

Next, we show that $z=x$. Let $x'\in X'$. By the proof of \prettyref{thm:nec_suff_cond_periodic}
the $2\pi$-periodic continuously differentiable function $f_{x,x'}\colon\R\to\C$, $f_{x,x'}(t)\coloneqq x'(T(t)x)$, 
coincides with its Fourier series and its Fourier coefficients are given by $x'(P_n x)$ for $n\in\Z$. 
Further, for the $2\pi$-periodic continuous function $f_{z,x'}\colon\R\to\C$, $f_{z,x'}(t)\coloneqq x'(T(t)z)$, 
its Fourier coefficients $c_n$, $n\in\Z$, fulfil
\begin{align*}
 c_n
&=\frac{1}{2\pi}\int_{0}^{2\pi}\e^{-\mathsf{i}ns}x'(T(s)z)\d s
 =\frac{1}{2\pi}\int_{0}^{2\pi}\e^{-\mathsf{i}ns}\sum_{k\in\Z}x'(T(s)P_k x)\d s\\
&\underset{\mathclap{\eqref{eq:fourier_exp_2}}}{=}\,\sum_{k\in\Z}\frac{1}{2\pi}\int_{0}^{2\pi}\e^{\mathsf{i}(k-n)s}x'(P_k x)\d s
=x'(P_n x) .
\end{align*}
Hence the functions $f_{z,x'}$ and $f_{x,x'}$ have the same Fourier coefficients and by Carleson's theorem we get 
$f_{z,x'}(t)=f_{x,x'}(t)$ for Lebesgue-almost every $t\in\R$. Since both functions are continuous, 
they actually coincide for every $t\in\R$, which implies 
\[
x'(z)=f_{z,x'}(0)=f_{x,x'}(0)=x'(x)
\]
for all $x'\in X'$. Thus we have $z=x$ by the Hahn--Banach theorem, which means 
\begin{equation}\label{eq:fourier_exp_3} 
x=\sum_{n\in\Z}P_n x . 
\end{equation}
Noting that $T(t)x\in D(A)$ by \cite[Proposition 1.2 (1), p.~260]{komura1968} if $x\in D(A)$, 
and $Ax\in D(A)$ if $x\in D(A^2)$, we obtain the identities \eqref{eq:fourier_1} and \eqref{eq:fourier_2} 
by replacing $x$ by $T(t)x$ and $Ax$ in \prettyref{eq:fourier_exp_3}, respectively, 
and using \eqref{eq:fourier_exp_2} and \eqref{eq:fourier_exp_1}. 
\end{proof}

\prettyref{thm:fourier_expansion} allows us to generalise and refine 
\cite[Chap.~IV, 2.28 Corollary, p.~269]{engel_nagel2000} next.

\begin{cor}\label{cor:fourier_expansion}
Let $X$ be a quasi-complete Hausdorff locally convex space, $t_0>0$ and 
$(T(t))_{t\geq 0}$ a family of maps from $X$ to $X$. Then the following two assertions are equivalent.
\begin{enumerate}[label=\upshape(\alph*), leftmargin=*]
\item \label{it:fourier_1} $(T(t))_{t\geq 0}$ is a periodic strongly continuous locally equicontinuous semigroup on $X$ with generator 
$(A,D(A))$ such that $D(A)=X$, $\sigma_{\p}(A)$ is bounded and $(T(t))_{t\geq 0}$ has period $\rho=\frac{t_0}{k}$ for some $k\in\N$.
\item \label{it:fourier_2} There are $m\in\N$ and projections $P_{n}\in\cL(X)$, $-m\leq n\leq m$, such that
\begin{enumerate}[label=\upshape(\roman*), leftmargin=*, widest=iii]
\item \label{it:fourier_20} $P_{n}P_{j}=0$ for all $-m\leq n,j\leq m$, $j\neq n$,
\item \label{it:fourier_21} $P_{-m}\neq 0$ or $P_{m}\neq 0$,
\item \label{it:fourier_22} $\sum_{n=-m}^{m}P_{n}=\id$, and
\item \label{it:fourier_23} $T(t)=\sum_{n=-m}^{m}\e^{\frac{2\pi\mathsf{i}nt}{t_0}}P_n$ for all $t\geq 0$.
\end{enumerate}
\end{enumerate}
If one of the two equivalent assertions is fulfilled, then $A =\sum_{n=-m}^{m}\frac{2\pi\mathsf{i}n}{t_0}P_n \in\cL(X)$.
\end{cor}
\begin{proof} 
``\ref{it:fourier_1}$\Rightarrow$\ref{it:fourier_2}'' Let $P_n\in\cL(X)$ for $n\in\Z$ be the map from 
\prettyref{prop:resolvent_periodic} \ref{it:res_per_2}. 
By \prettyref{prop:resolvent_periodic} \ref{it:res_per_4} $P_n$ is a projection 
for all $n\in\Z$, and we have for all $n,j\in\Z$, $n\neq j$, and $x\in X$ by Equation \eqref{eq:fourier_exp_2} 
from the proof of \prettyref{thm:fourier_expansion} that 
\[
 P_{n}P_{j}x
=\frac{1}{\rho}\int_{0}^{\rho}\e^{-\mu_n s}T(s)P_j x\d s
\underset{\mathclap{\eqref{eq:fourier_exp_2}}}{=}\;\frac{1}{\rho}\int_{0}^{\rho}\e^{(\mu_j-\mu_n) s}\d s\,P_j x
=0.
\]
Now, the implication follows from \prettyref{thm:fourier_expansion} in combination with
\prettyref{prop:resolvent_periodic} \ref{it:res_per_4}, the assumption that $\sigma_{\p}(A)$ is bounded and 
$D(A)=X$. In particular, $D(A)=X$ and \eqref{eq:fourier_2} imply that there is $i\in\Z$, $i\neq 0$, 
such that $P_{i}\neq 0$ (if $i=0$ were the only $i\in\Z$ with $P_{i}\neq 0$, then $\rho=0$ 
by \prettyref{prop:period_zero}, which is contradiction). 

``\ref{it:fourier_2}$\Rightarrow$\ref{it:fourier_1}'' Let there be $m\in\N$ and projections $P_{n}\in\cL(X)$, 
$-m\leq n\leq m$, such that conditions \ref{it:fourier_20}--\ref{it:fourier_23} are fulfilled. 
Due to \ref{it:fourier_20}, \ref{it:fourier_22} and \ref{it:fourier_23} $(T(t))_{t\geq 0}$ is a periodic strongly 
continuous locally equicontinuous semigroup on $X$ with $T(t_0)=\id$ whose generator $(A,D(A))$ fulfils $D(A)=X$ 
and $A =\sum_{n=-m}^{m}\frac{2\pi\mathsf{i}n}{t_0}P_n$.  
Moreover, we note that it follows from \ref{it:fourier_21} that there is $x\in X$ with $P_{-m}x\neq 0$ or 
$P_{m}x\neq 0$. W.l.o.g.~$P_{m}x\neq 0$. Then it holds that
\[
 AP_{m}x
=\sum_{n=-m}^{m}\frac{2\pi\mathsf{i}n}{t_0}P_n P_m x
\underset{\mathclap{\text{\ref{it:fourier_20}}}}{=}\frac{2\pi\mathsf{i}m}{t_0}P_m^2 x
=\frac{2\pi\mathsf{i}m}{t_0}P_m x\neq 0
\]
as $m\neq 0$, yielding that $(T(t))_{t\geq 0}$ has period $\rho>0$ by \prettyref{prop:period_zero}. 
Thus there is $k\in\N$ such that 
$\rho=\frac{t_0}{k}$ by \prettyref{prop:natural_period}. Suppose that there are $\lambda\in\C$, 
$\lambda\neq \frac{2\pi\mathsf{i}n}{t_0}$ for all $-m\leq n\leq m$, and $x\in X$ such that $Ax=\lambda x$. 
Then we have by \ref{it:fourier_22} that
\[
 0
=Ax-\lambda x
=\sum_{n=-m}^{m}\Bigl(\frac{2\pi\mathsf{i}n}{t_0}-\lambda\Bigr)P_n x,
\]
which implies 
\[
 0
=P_{j}\sum_{n=-m}^{m}\Bigl(\frac{2\pi\mathsf{i}n}{t_0}-\lambda\Bigr)P_n x
\underset{\mathclap{\text{\ref{it:fourier_20}}}}{=}\Bigl(\frac{2\pi\mathsf{i}j}{t_0}-\lambda\Bigr)P_j^2 x
=\Bigl(\frac{2\pi\mathsf{i}j}{t_0}-\lambda\Bigr)P_j x
\]
for all $-m\leq j\leq m$. Hence $P_j x=0$ for all $-m\leq j\leq m$ and so $x=0$ by \ref{it:fourier_22}. 
We conclude that $\lambda\notin\sigma_{\p}(A)$ and 
$\sigma_{\p}(A)\subseteq \{\frac{2\pi\mathsf{i}n}{t_0}\;|\;-m\leq n\leq m\}$. 
In particular, $\sigma_{\p}(A)$ is bounded. 
\end{proof}

\section{Spectral inclusion and mapping theorems}
\label{sect:spectral_mapping}

We begin our final section with spectral inclusion theorems that cover the known case \cite[Chap.~IV, 3.6 Spectral Inclusion Theorem, p.~276]{engel_nagel2000} in the setting of Banach spaces. 

\begin{thm}\label{thm:spec_incl}
Let $X$ be a sequentially complete Hausdorff locally convex space and $(T(t))_{t\geq 0}$ a strongly continuous 
semigroup on $X$ with generator $(A,D(A))$. 
Then the following assertions hold for all $t\geq 0$.
\begin{enumerate}[label=\upshape(\alph*), leftmargin=*]
\item \label{it:spec_incl_2} $\e^{t\sigma_{\alg}(A)}\subseteq \sigma_{\alg} (T(t))$,
\item \label{it:spec_incl_3} $\e^{t\sigma_{\p}(A)}\subseteq \sigma_{\p} (T(t))$,
\item \label{it:spec_incl_5} $\e^{t\sigma_{\res}(A)}\subseteq \sigma_{\res} (T(t))$.
\end{enumerate}
If in addition $(T(t))_{t\geq 0}$ is locally equicontinuous, then the following assertions hold for all $t\geq 0$.
\begin{enumerate}[label=\upshape(\alph*), leftmargin=*]\setcounter{enumi}{3}
\item \label{it:spec_incl_1} $\e^{t\sigma(A)}\subseteq \sigma (T(t))$,
\item \label{it:spec_incl_4} $\e^{t\sigma_{\ap}(A)}\subseteq \sigma_{\ap} (T(t))$,
\item \label{it:spec_incl_4a} $\e^{t\sigma_{\ap}^{\seq}(A)}\subseteq \sigma_{\ap}^{\seq} (T(t))$,
\item \label{it:spec_incl_4b} $\e^{t\sigma_{\bap}(A)}\subseteq \sigma_{\bap} (T(t))$,
\item \label{it:spec_incl_4c} $\e^{t\sigma_{\bap}^{\seq}(A)}\subseteq \sigma_{\bap}^{\seq} (T(t))$,
\item \label{it:spec_incl_6}  $\e^{t\sigma^{\ast}(A)}\subseteq \sigma^{\ast} (T(t))$.
\end{enumerate}
\end{thm}
\begin{proof} 
Let $t\geq 0$. We start with two observations. First, if $(T(s))_{s\geq 0}$ is locally equicontinuous, 
then for every $q\in\Gamma_X$ there are $p\in\Gamma_X$ and $C\geq 0$ such that for all $x\in X$ we have
\begin{equation}\label{eq:spectral_incl_1}
     q\Bigl( \int_{0}^{t}\e^{\lambda (t-s)}T(s)x \d s\Bigr)
\leq t\e^{|\re(\lambda)|t}\sup_{s\in [0,t]}q(T(s)x)
\leq Ct\e^{|\re(\lambda)|t}p(x).
\end{equation}
Second, if $\lambda\in\C$ is such that $\lambda-A$ and $\e^{\lambda t}-T(t)$ are bijective, then 
it follows for all $x\in X$ by Equation \eqref{eq:rescale_iden_2} from \prettyref{prop:rescale_identities} that
\begin{align*}
  \int_{0}^{t}\e^{\lambda (t-s)}T(s)x \d s
&=\int_{0}^{t}\e^{\lambda (t-s)}T(s)(\lambda - A)(\lambda - A)^{-1}x \d s \\
&\underset{\mathclap{\eqref{eq:rescale_iden_2}}}{=}\e^{\lambda t}(\lambda-A)^{-1}x-T(t)(\lambda-A)^{-1}x
=(\e^{\lambda t}-T(t))(\lambda-A)^{-1}x
\end{align*}
and thus
\begin{equation}\label{eq:spectral_incl_2}
 (\lambda-A)^{-1}x
=(\e^{\lambda t}-T(t))^{-1} \int_{0}^{t}\e^{\lambda (t-s)}T(s)x \d s.
\end{equation}
Now, let us turn to the proofs of the listed statements. 
 
\ref{it:spec_incl_3} Let $\lambda\in\sigma_{\p}(A)$. So, $\lambda-A$ is not injective. 
Then there is $x\in D(A)$, $x\neq 0$, such that 
$(\lambda-A)x=0$. By multiplying \eqref{eq:rescale_iden_2} with $\e^{\lambda t}$ this implies 
$(T(t)-\e^{\lambda t})x=0$ and so $\e^{\lambda t}-T(t)$ is not injective.

\ref{it:spec_incl_2} Let $\lambda\in\sigma_{\alg}(A)$ such that $\lambda-A$ is not surjective. 
Then there is $y\in X$ such that for all 
$z\in D(A)$ it holds that $(\lambda-A)z\neq y$. Since $\int_{0}^{t}\e^{\lambda (t-s)}T(s)x \d s\in D(A)$ for all 
$x\in X$ by \prettyref{prop:rescale_identities}, this yields that there is no 
$x\in X$ such that $(T(t)-\e^{\lambda t})x=y$ by multiplying \eqref{eq:rescale_iden_1} with 
$\e^{\lambda t}$. Thus $\e^{\lambda t}-T(t)$ is not surjective. Together with part \ref{it:spec_incl_3} this 
proves statement \ref{it:spec_incl_2}. 

\ref{it:spec_incl_1} Let $\lambda\in\sigma(A)$ such that $\lambda-A$ is bijective but 
$(\lambda-A)^{-1}\notin\mathcal{L}(X)$. We only need to consider the case that $\e^{\lambda t}-T(t)$ is bijective. 
Since $(\lambda-A)^{-1}\notin\mathcal{L}(X)$, there is a net $(x_{i})_{i\in I}$ in $X$ converging to $0$ such 
that $((\lambda-A)^{-1}x_{i})_{i\in I}$ does not converge to $0$. 
By the local equicontinuity of the semigroup and \eqref{eq:spectral_incl_1} the net 
$(\int_{0}^{t}\e^{\lambda (t-s)}T(s)x_{i} \d s)_{i\in I}$ converges to $0$ in $X$. 
Suppose that $(\e^{\lambda t}-T(t))^{-1}\in \mathcal{L}(X)$. Then $((\lambda-A)^{-1}x_{i})_{i\in I}$ 
converges to $0$ by \eqref{eq:spectral_incl_2}, 
which is a contradiction. Hence $(\e^{\lambda t}-T(t))^{-1}\notin \mathcal{L}(X)$. 
In combination with part \ref{it:spec_incl_2} this yields our statement.

\ref{it:spec_incl_5} Let $\lambda\in\sigma_{\res}(A)$. By multiplying Equation \eqref{eq:rescale_iden_1} 
from \prettyref{prop:rescale_identities} with $\e^{\lambda t}$, 
we see that 
\begin{equation}\label{eq:spectral_incl_3}
\ran(\e^{\lambda t}-T(t))\subseteq \ran(\lambda-A).
\end{equation}
Thus $\ran(\e^{\lambda t}-T(t))$ cannot be dense in $X$, which means that $\e^{\lambda t}\in\sigma_{\res}(T(t))$. 

\ref{it:spec_incl_4} Let $\lambda\in\sigma_{\ap}(A)$. Then there is a net $(x_i)_{i\in I}$ in $D(A)$ 
which does not converge to $0$ and fulfils $\lim_{i\in I}(A-\lambda)x_i=0$. 
Due to Equation \eqref{eq:rescale_iden_2} from \prettyref{prop:rescale_identities} we observe that
\[
 (\e^{\lambda t}-T(t))x_i
\underset{\mathclap{\eqref{eq:rescale_iden_2}}}{=}\int_{0}^{t}\e^{\lambda (t-s)}T(s)(\lambda-A)x_i \d s
\]
for $i\in I$. Let $q\in\Gamma_X$. By the local equicontinuity of the semigroup there are $p\in\Gamma_X$ and $C\geq 0$ 
such that for all $i\in I$
\[
 q((\e^{\lambda t}-T(t))x_i)\leq Ct\e^{|\re(\lambda)|t}p((\lambda-A)x_i)
\]
which we obtain from \eqref{eq:spectral_incl_1} by replacing $x$ by $(\lambda-A)x_i$. 
Hence $(T(t)x_i-\e^{\lambda t}x_i)_{i\in I}$ converges to $0$, implying $\e^{\lambda t}\in\sigma_{\ap}(T(t))$.

\ref{it:spec_incl_4a}, \ref{it:spec_incl_4b} and \ref{it:spec_incl_4c} These statements follow from the proof 
of part \ref{it:spec_incl_4}.

\ref{it:spec_incl_6} Let $\lambda\in\sigma^{\ast}(A)$. If $\lambda\in\sigma(A)$, then statement \ref{it:spec_incl_6} 
is covered by part \ref{it:spec_incl_1}. So, let us consider the case that 
$\lambda\in\sigma^{\ast}(A)\setminus\sigma(A)$. 
This means that for all $\delta>0$ such that $B(\lambda,\delta)\subseteq\rho(A)$ it holds that 
$\{R(\mu,A)\;|\;\mu\in B(\lambda,\delta)\}$ is not equicontinuous in $\cL(X)$. Suppose that 
$\e^{\lambda t}\in\rho^{\ast}(T(t))$. Then there is $\varepsilon >0$ such that 
$B(\e^{\lambda t},\varepsilon)\subseteq\rho(T(t))$ and
$\{R(\mu,T(t))\;|\;\mu\in B(\e^{\lambda t},\varepsilon)\}$ is equicontinuous in $\cL(X)$. 
Since the map $f\colon\C\to\C$, $f(z)\coloneqq\e^{zt}$, is continuous there is $\delta_\varepsilon >0$ such that 
for all $\mu\in B(\lambda,\delta_\varepsilon )$ it holds that $\e^{\mu t}\in B(\e^{\lambda t},\varepsilon)$. 
Let $\mu\in B(\lambda,\delta_\varepsilon )$. It follows from multiplying Equation \eqref{eq:rescale_iden_1} 
from \prettyref{prop:rescale_identities} with $\e^{\mu t}$ 
and replacing $\lambda$ by $\mu$ that 
\[
 (\e^{\mu t}-T(t))x
=(\mu-A)\int_{0}^{t}\e^{\mu (t-s)}T(s)x \d s
\]
for all $x\in X$. Since $\e^{\mu t}\in B(\e^{\lambda t},\varepsilon)\subseteq\rho(T(t))$, we know that $\e^{\mu t}-T(t)$ 
is invertible. By replacing in the equality above $x$ by $R(\e^{\mu t},T(t))x=(\e^{\mu t}-T(t))^{-1}x$, this yields
\[
 x
=(\mu-A)\int_{0}^{t}\e^{\mu (t-s)}T(s)R(\e^{\mu t},T(t))x \d s
\]
for all $x\in X$, so $\mu-A$ is surjective. The injectivity of $\e^{\mu t}-T(t)$ in 
combination with Equation \eqref{eq:rescale_iden_2} from \prettyref{prop:rescale_identities} 
implies that $\mu-A$ is also injective, so it is bijective. 
Thus we have 
\begin{equation}\label{eq:spectral_incl_4}
 (\mu-A)^{-1}x
\underset{\mathclap{\eqref{eq:spectral_incl_2}}}{=}R(\e^{\mu t},T(t))\int_{0}^{t}\e^{\mu (t-s)}T(s)x \d s
\end{equation}
for all $x\in X$. Hence $(\mu-A)^{-1}\in\cL(X)$ by \eqref{eq:spectral_incl_1} and so 
$B(\lambda,\delta_\varepsilon )\subseteq\rho(A)$. 
Further, the fact that $\e^{\mu t}\in B(\e^{\lambda t},\varepsilon)$ for all $\mu\in B(\lambda,\delta_\varepsilon )$, 
the equicontinuity of $\{R(\eta,T(t))\;|\;\eta\in B(\e^{\lambda t},\varepsilon)\}$ in $\cL(X)$ 
and \eqref{eq:spectral_incl_1} in combination with \eqref{eq:spectral_incl_4} imply that 
$\{R(\mu,A)\;|\;\mu\in B(\lambda,\delta_\varepsilon )\}$ is equicontinuous in $\cL(X)$, which is a contradiction. 
We conclude that $\e^{\lambda t}\in\sigma^{\ast}(T(t))$.
\end{proof}

We make the following little observation about eigenvectors and eigenvalues of the generator of a 
strongly continuous semigroup.

\begin{rem}\label{rem:eigenvectors_0}
Let $X$ be a sequentially complete Hausdorff locally convex space and $(T(t))_{t\geq 0}$ a strongly continuous 
semigroup on $X$ with generator $(A,D(A))$. Then the following assertions are equivalent for $\lambda\in\C$ and 
$x\in X$.
\begin{enumerate}[label=\upshape(\alph*), leftmargin=*]
\item\label{it:eigenvec_01} $x\in D(A)$ and $Ax=\lambda x$.
\item\label{it:eigenvec_02} $T(t)x=\e^{\lambda t}x$ for all $t\in[0,\infty)$.
\item\label{it:eigenvec_03} There is $t_0>0$ such that $T(t)x=\e^{\lambda t}x$ for all $t\in[0,t_0]$.
\end{enumerate}
Indeed, for $x=0$ this is clearly true. 
If $x\neq 0$, then the implication ``\ref{it:eigenvec_01}$\Rightarrow$\ref{it:eigenvec_02}'' follows from the proof 
of \prettyref{thm:spec_incl} \ref{it:spec_incl_3}, and the implication ``\ref{it:eigenvec_02}$\Rightarrow$\ref{it:eigenvec_03}'' is obvious. Let us turn to ``\ref{it:eigenvec_03}$\Rightarrow$\ref{it:eigenvec_01}''. 
Let $t_0>0$ be such that $T(t)x=\e^{\lambda t}x$ for all $t\in[0,t_0]$. Then we have 
\[
 \lim_{t\to 0\rlim}\frac{T(t)x-x}{t}
=\lim_{t\to 0\rlim}\frac{\e^{\lambda t}x-x}{t}
=\lim_{t\to 0\rlim}\frac{\e^{\lambda t}-1}{t}x
=\lambda x,
\]
so $x\in D(A)$ and $Ax=\lambda x$.
\end{rem}

\begin{rem}\label{rem:eigenvectors}
Let $X$ be a sequentially complete Hausdorff locally convex space and $(T(t))_{t\geq 0}$ a strongly continuous 
locally equicontinuous semigroup on $X$ with generator $(A,D(A))$. 
\begin{enumerate}[label=\upshape(\alph*), leftmargin=*]
\item \label{it:eigenvec_1} Let $\lambda\in\sigma_{\ap}(A)$ and $(x_i)_{i\in I}$ be an (bounded, sequential) 
approximate eigenvector of $A$ corresponding to $\lambda$. Looking at the proof of 
\prettyref{thm:spec_incl} \ref{it:spec_incl_4}, 
we note that $(x_i)_{i\in I}$ is also an (bounded, sequential) approximate eigenvector of $T(t)$ corresponding 
to $\e^{\lambda t}$ for all $t\geq 0$. 
\item In general, we do not know whether $\e^{t\sigma_{\apr}(A)}\subseteq \sigma_{\apr} (T(t))$ holds for all $t\geq 0$. 
It is not clear how to infer from \eqref{eq:spectral_incl_3} alone that 
$\ran(\e^{\lambda t}-T(t))$ cannot be closed if $\ran(\lambda-A)$ is not closed for some $\lambda\in\C$. However, 
we know that the spectral inclusion theorem for the approximate spectra holds 
by the closedness of the generator $(A,D(A))$, \prettyref{prop:approx_spec} \ref{it:ap_spec_3} 
and \prettyref{thm:spec_incl} \ref{it:spec_incl_4} if $X$ is a Fr\'echet space. 
If $X$ is only complete, then the best that we can say is that 
$\e^{t\sigma_{\apr}(A)}\subseteq \sigma_{\ap} (T(t))$ by \prettyref{prop:approx_spec} \ref{it:ap_spec_1} 
and \prettyref{thm:spec_incl} \ref{it:spec_incl_4}.
\item \label{it:spec_top_incl} In general, we also do not know whether $\e^{t\sigma_{\tp}(A)}\subseteq \sigma_{\tp} (T(t))$ holds for all 
$t\geq 0$. It is not clear how to see that $\e^{\lambda t}-T(t)$ is bijective and 
$(\e^{\lambda t}-T(t))^{-1}\not\in\cL(X)$ if $\lambda-A$ is bijective and $(\lambda-A)^{-1}\notin\mathcal{L}(X)$ for 
some $\lambda\in\C$. However, if $\sigma_{\tp}(A)=\varnothing$, then $\e^{t\sigma_{\tp}(A)}=\varnothing$ and 
the spectral inclusion theorem for the topological spectra holds trivially. This is the case for the spaces $X$ 
listed in \prettyref{rem:alg_resolvent}. 

On the other hand, let $X$ be a space such that any bijective continuous 
linear map $S\colon X\to X$ has a continuous inverse. Then $\sigma_{\tp}(T(t))=\varnothing$ 
since $S_{t}\coloneqq \lambda-T(t)\in\cL(X)$ has continuous inverse for all $\lambda\in\C$ and $t\geq 0$ 
such that $\lambda-T(t)$ is bijective. For instance, such spaces $X$ are the ones listed 
in \prettyref{rem:alg_resolvent} since the continuous map $T(t)$ is clearly closed for all $t\geq 0$.
\end{enumerate}
\end{rem}

We have now everything at hand that we need to prove the spectral mapping theorem for the point spectrum of 
strongly continuous locally equicontinuous semigroups on sequentially complete Hausdorff locally convex spaces, 
which generalises one part of 
\cite[Chap.~IV, 3.7 Spectral Mapping Theorem for Point and Residual Spectrum, p.~277]{engel_nagel2000}. 

\begin{thm}\label{thm:spec_point}
Let $X$ be a sequentially complete Hausdorff locally convex space and $(T(t))_{t\geq 0}$ a strongly continuous 
locally equicontinuous semigroup on $X$ with generator $(A,D(A))$. 
Then 
\[
\sigma_{\p}(T(t))\setminus\{0\}=\e^{t\sigma_{\p}(A)}
\]
holds for all $t\geq 0$.
\end{thm}
\begin{proof}
Due to \prettyref{thm:spec_incl} \ref{it:spec_incl_3} we only need to prove the inclusion 
``$\subseteq$''. 
If $X=\{0\}$, then the point spectra are empty and so the inclusion ``$\subseteq$'' 
trivially holds. Let $X\neq\{0\}$.
If $t=0$, then $T(0)=\id$, $A=0$ and $\sigma_{\p}(T(0))=\{1\}$ and $\sigma_{\p}(A)=\{0\}$. 
So, the inclusion ``$\subseteq$'' holds. Let $t>0$ and $\lambda\in\sigma_{\p}(T(t))$, $\lambda\neq 0$. 
Then there is $\theta\in [0,2\pi)$ such that $\lambda=|\lambda|\e^{\mathsf{i}\theta}$. 
As in the proof of \prettyref{thm:nec_suff_cond_periodic} we may use \prettyref{prop:rescale_sg} and 
consider the rescaled semigroup $(S(s))_{s\geq 0}$ given by 
$S(s)\coloneqq \e^{-s(\ln(|\lambda|)+\mathsf{i}\theta)}T(ts)$ for $s\geq 0$ with generator $(B,D(B))$ such that 
$B=tA-\ln(|\lambda|)-\mathsf{i}\theta$ and $D(B)=D(A)$. Since $S(1)=\frac{1}{\lambda} T(t)$, it follows that $S(1)$ 
has eigenvalue $1$. Hence we may assume w.l.o.g.~that $t=1$ and $\lambda=1$ from the start. 
Let us consider the corresponding non-trivial eigenspace
\[
Y\coloneqq \ker(1-T(1))=\{y\in X\;|\;T(1)y=y\},
\]  
which is a $(T(s))_{s\geq 0}$-invariant closed subspace of $X$. 
The restricted semigroup $(T(s)_{\mid Y})_{s\geq 0}$ is strongly continuous and locally equicontinuous 
by \prettyref{prop:restricted_sg} \ref{it:rest_sg_1} and \ref{it:rest_sg_2}. 
Its generator is $(A_{\mid Y},D(A_{\mid Y}))$ fulfilling $D(A_{\mid Y})=D(A)\cap Y$ 
by \prettyref{prop:restricted_sg} \ref{it:rest_sg_3}.
Further, we have $T(1)_{\mid Y}=\id$ on $Y$ and so $(T(s)_{\mid Y})_{s\geq 0}$ is periodic with period 
$\rho=0$ or $\rho=\frac{1}{k}$ for some $k\in\N$ by \prettyref{prop:natural_period}. 
In particular, it holds that $\sigma_{\p}(A_{\mid Y})\subseteq \sigma_{\p}(A)$. 
If $\rho=0$, then we have by \prettyref{prop:period_zero} that $T(s)_{\mid Y}=\id$ on $Y$ and $A_{\mid Y}=0$ on $D(A_{\mid Y})=Y$. 
This implies that $0\in\sigma_{\p}(A)$ and so $\lambda=1=\e^{0}\in \e^{\sigma_{\p}(A)}$ since $Y$ is non-trivial. 
If $\rho=\frac{1}{k}$ for some $k\in\N$, then $\sigma_{\p}(A_{\mid Y})\subseteq 2\pi\mathsf{i}k\Z$ 
by \prettyref{prop:resolvent_periodic} \ref{it:res_per_1}. 
Moreover, $\sigma_{\p}(A_{\mid Y})\neq\varnothing$ by \prettyref{thm:nec_suff_cond_periodic} since 
$Y$ is sequentially complete as a closed subspace of the sequentially complete space $X$. 
So, for $\mu\in\sigma_{\p}(A_{\mid Y})$ there is $m\in\Z$ such that $\mu=2\pi\mathsf{i}km$. 
Since $\sigma_{\p}(A_{\mid Y})\subseteq \sigma_{\p}(A)$, we get 
\[
\lambda=1=\e^{2\pi\mathsf{i}km}=\e^{\mu}\in\e^{\sigma_{\p}(A)}.
\]
\end{proof}

We also have the following relation between the eigenspaces of $A$ and $T(t)$, 
which is observed in the case of Banach spaces in \cite[Chap.~IV, 3.8 Corollary, p.~278]{engel_nagel2000}.

\begin{cor}\label{cor:eigenspace}
Let $X$ be a sequentially complete Hausdorff locally convex space and $(T(t))_{t\geq 0}$ a strongly continuous 
semigroup on $X$ with generator $(A,D(A))$. 
Then the following assertions hold for all $\lambda\in\C$.
\begin{enumerate}[label=\upshape(\alph*), leftmargin=*]
\item \label{it:eigensp_1} $\ker (\lambda-A)=\bigcap_{t\geq 0}\ker (\e^{\lambda t}-T(t))$,
\item \label{it:eigensp_2} $\ker (\e^{\lambda t}-T(t))=\overline{\operatorname{span}}\left(\bigcup_{n\in\Z}
\ker (\lambda+\frac{2\pi\mathsf{i}n}{t}-A)\right)$ for all $t>0$ if $(T(s))_{s\geq 0}$ is locally equicontinuous.
\end{enumerate}
\end{cor}
\begin{proof} 

\ref{it:eigensp_1} This directly follows from \prettyref{rem:eigenvectors_0}.

\ref{it:eigensp_2} By \prettyref{thm:spec_point} we have $\sigma_{\p}(T(t))\setminus\{0\}=\e^{t\sigma_{\p}(A)}$ for all 
$t\geq 0$. Let $t>0$ and $n\in\Z$. If $\e^{\lambda t}=\e^{t(\lambda+\frac{2\pi\mathsf{i}n}{t})}\notin 
\sigma_{\p}(T(t))\setminus\{0\}$, then $\lambda+\frac{2\pi\mathsf{i}n}{t}\notin \sigma_{\p}(A)$. 
Thus $\ker (\e^{\lambda t}-T(t))=\{0\}=\overline{\operatorname{span}}\left(\bigcup_{n\in\Z}
\ker (\lambda+\frac{2\pi\mathsf{i}n}{t}-A)\right)$. 
Let $\e^{\lambda t}\in \sigma_{\p}(T(t))\setminus\{0\}$. 

``$\supseteq$'' Let $x\in \ker (\lambda+\frac{2\pi\mathsf{i}n}{t}-A)$. 
Then $T(t)x=\e^{t(\lambda+\frac{2\pi\mathsf{i}n}{t})}x=\e^{\lambda t}x$ by \prettyref{rem:eigenvectors_0}. 
Since $\ker (\e^{\lambda t}-T(t))$ is a closed linear subspace 
of $X$, this proves the inclusion ``$\supseteq$''.

``$\subseteq$'' Proceeding as in the proof of \prettyref{thm:spec_point}, we may w.l.o.g.~assume that $\lambda=0$ 
and $t=1$, and we set $Y\coloneqq \ker(1-T(1))$.
From the proof of \prettyref{thm:spec_point} we recall that the restricted semigroup $(T(s)_{\mid Y})_{s\geq 0}$ 
is strongly continuous, locally equicontinuous and periodic with period $\rho=0$ or $\rho=\frac{1}{k}$ 
for some $k\in\N$. Its generator is $(A_{\mid Y},D(A_{\mid Y}))$ whose domain fulfils $D(A_{\mid Y})=D(A)\cap Y$. 
If $\rho=0$, then we have by \prettyref{prop:period_zero} that $A_{\mid Y}=0$ and for $n=0$ we get 
\[
 \ker (\e^{\lambda t}-T(t))
=\ker(1-T(1))
=Y
=\ker\Bigl(\lambda+\frac{2\pi\mathsf{i}n}{t}-A_{\mid Y}\Bigr)
\subseteq \ker\Bigl(\lambda+\frac{2\pi\mathsf{i}n}{t}-A\Bigr),
\]
so the inclusion ``$\subseteq$'' holds in this case. Now, let us consider the case that 
$\rho=\frac{1}{k}$ for some $k\in\N$. Then we have by (the proof of) \prettyref{thm:nec_suff_cond_periodic} 
\begin{align*}
 \ker (\e^{\lambda t}-T(t))
&=\ker(1-T(1))
 =Y
 =\overline{\operatorname{span}}\left(\bigcup_{n\in\Z}\ker\Bigl(\frac{2\pi\mathsf{i}n}{\rho}-A_{\mid Y}\Bigr)\right)\\
&=\overline{\operatorname{span}}\left(\bigcup_{n\in\Z}\ker(2\pi\mathsf{i}nk-A_{\mid Y})\right)
\subseteq \overline{\operatorname{span}}\left(\bigcup_{n\in\Z}\ker(2\pi\mathsf{i}n-A)\right)\\
&=\overline{\operatorname{span}}\left(\bigcup_{n\in\Z}\ker\Bigl(\lambda+\frac{2\pi\mathsf{i}n}{t}-A\Bigr)\right)
\end{align*}
Hence the inclusion ``$\subseteq$'' also holds in this case.
\end{proof}

Now, we turn to proving the spectral mapping theorem for the residual spectrum of 
strongly continuous locally equicontinuous semigroups on Hausdorff locally convex spaces $X$. 
We start with a generalisation of some of the results given in 
\cite[Chap.~IV, 2.18 Proposition (i), (vi), p.~262]{engel_nagel2000} where $X$ is a Banach space. 
However, we will modify the proof of these results given in \cite[p.~28--29]{vanneerven1996}. 

\begin{prop}\label{prop:spec_point_sun_dual}
Let $X$ be a sequentially complete Hausdorff locally convex space such that $X_{\ob}'$ is sequentially complete 
and $(T(t))_{t\geq 0}$ a strongly continuous semigroup on $X$ with generator $(A,D(A))$. 
Then the following assertions hold.
\begin{enumerate}[label=\upshape(\alph*), leftmargin=*]
\item \label{it:spec_point_dual_1} $\sigma_{\p}(A')=\sigma_{\p}(A^{\odot})$,
\item \label{it:p_spec_point_dual_2} $\sigma_{\p}(T'(t))=\sigma_{\p}(T^{\odot}(t))$ for all $t\geq 0$ 
if $\rho_{\alg}(A')\neq\varnothing$.
\end{enumerate}
\end{prop}
\begin{proof}
The proof is based on the results in \prettyref{thm:dual_str_cont}.

\ref{it:spec_point_dual_1} ``$\subseteq$'' Let $\sigma_{\p}(A')$. Then there is $x'\in D(A')$, $x'\neq 0$, 
such that $A'x'=\lambda x'$. Due to the inclusion $D(A')\subseteq X^{\odot}$ we obtain that $x'\in X^{\odot}$. 
It follows from \cite[Corollary, p.~261]{komura1968} that
\begin{align*}
  \langle T^{\odot}(t)x'-x',x \rangle
&=\langle T'(t)x'-x',x \rangle
 =\langle x',T(t)x -x\rangle
 =\langle x',A\int_{0}^{t}T(s)x\d s\rangle\\
&=\langle A'x',\int_{0}^{t}T(s)x\d s\rangle
 =\langle \lambda x',\int_{0}^{t}T(s)x\d s\rangle
 =\lambda \int_{0}^{t}\langle x',T(s)x\rangle\d s 
\end{align*}
for all $t\geq 0$ and $x\in X$.
Hence we get for all $t>0$ and $x\in X$ that
\[
 \langle \frac{1}{t}( T^{\odot}(t)x'-x')-\lambda x',x \rangle
=\frac{\lambda}{t}\int_{0}^{t}\langle x',T(s)x-x\rangle\d s 
=\frac{\lambda}{t}\int_{0}^{t}\langle T^{\odot}(s)x'-x',x\rangle\d s .
\]
For bounded $M\subseteq X$ we deduce that 
\[
    \sup_{x\in M}|\langle \frac{1}{t}( T^{\odot}(t)x'-x')-\lambda x',x \rangle|
\leq |\lambda|\sup_{s\in [0,t]}\sup_{x\in M}|\langle T^{\odot}(s)x'-x',x\rangle|.
\]  
Letting $t\to 0\rlim$, the $\beta(X',X)$-strong continuity of $(T^{\odot}(t))_{t\geq 0}$ on $X^{\odot}$ 
implies that $x'\in D(A^{\odot})$ and $A^{\odot}x'=\lambda x'$. Thus $\lambda\in \sigma_{\p}(A^{\odot})$.

``$\supseteq$'' Let $\lambda\in \sigma_{\p}(A^{\odot})$. Then there is $x^{\odot}\in D(A^{\odot})$, $x^{\odot}\neq 0$, 
such that $A^{\odot}x^{\odot}=\lambda x^{\odot}$. Hence we have $x^{\odot}\in D(A^{\odot})\subseteq D(A')$ and 
\[
\lambda x^{\odot}=A^{\odot}x^{\odot}=A_{\mid X^{\odot}}'x^{\odot}=A'x^{\odot},
\]
yielding that $\lambda\in\sigma_{\p}(A')$.

\ref{it:p_spec_point_dual_2} ``$\subseteq$'' Let $t\geq 0$ and $\lambda\in \sigma_{\p}(T'(t))$. 
Then there is $x'\in X'$, $x'\neq 0$, such that $T'(t)x'=\lambda x'$. Let $\mu\in\rho_{\alg}(A')$. 
It follows that $(\mu-A')^{-1}x'\in D(A')\subseteq X^{\odot}$ and 
\[
 T^{\odot}(t)(\mu-A')^{-1}x'
=T'(t)(\mu-A')^{-1}x'
=(\mu-A')^{-1}T'(t)x'
=\lambda (\mu-A')^{-1}x'.
\]
Since $x'\neq 0$, we also have that $(\mu-A')^{-1}x'\neq 0$ and so $(\mu-A')^{-1}x'$ is an eigenvector of 
$T^{\odot}(t)$ corresponding to $\lambda$. Thus $\lambda\in \sigma_{\p}(T^{\odot}(t))$. 

``$\supseteq$'' This inclusion is clear since $T^{\odot}(t)$ is a restriction of $T'(t)$ for all $t\geq 0$.
\end{proof}

The preceding result allows us to prove the spectral mapping theorem for the residual spectrum. 

\begin{thm}\label{thm:spec_res}
Let $X$ be a sequentially complete Hausdorff locally convex space such that $X_{\ob}'$ is sequentially complete 
and $(T(t))_{t\geq 0}$ a strongly continuous locally equicontinuous semigroup on $X$ with generator $(A,D(A))$ 
such that $\rho_{\alg}(A)\neq\varnothing$. 
Then 
\[
\sigma_{\res}(T(t))\setminus\{0\}=\e^{t\sigma_{\res}(A)}
\]
holds for all $t\geq 0$.
\end{thm}
\begin{proof}
Let $t\geq 0$. First, we note that $\rho_{\alg}(A')\neq\varnothing$ by \prettyref{prop:dual_spec} \ref{it:dual_spec_2} 
because $\rho_{\alg}(A)\neq\varnothing$. Due to \prettyref{prop:dual_spec} \ref{it:dual_spec_3} and 
\prettyref{prop:spec_point_sun_dual} we have 
$\sigma_{\res}(A)=\sigma_{\p}(A')=\sigma_{\p}(A^{\odot})$ and 
$\sigma_{\res}(T(t))=\sigma_{\p}(T'(t))=\sigma_{\p}(T^{\odot}(t))$. 
Next, we remark that $X^{\odot}$ is sequentially complete as a $\beta(X',X)$-closed subspace 
of the sequentially complete space $X_{\ob}'$ by \prettyref{thm:dual_str_cont}. 
Applying \prettyref{thm:spec_point} to the $\beta(X',X)$-strongly continuous locally 
$\beta(X',X)$-equicontinuous semigroup $(T^{\odot}(s))_{s\geq 0}$ on $X^{\odot}$ (see \prettyref{thm:dual_str_cont}), 
we conclude our statement. 
\end{proof}

\prettyref{thm:spec_res} generalises one part of \cite[Chap.~IV, 3.7 Spectral Mapping Theorem for Point and Residual Spectrum, p.~277]{engel_nagel2000} (cf.~\cite[Theorem 2.1.3, p.~30]{vanneerven1996}) where $X$ is a Banach space.

\begin{rem}\label{rem:non_empty_alg_res}
Let $X$ be a Hausdorff locally convex space.
\begin{enumerate}[label=\upshape(\alph*), leftmargin=*]
\item \label{it:non_empty_alg_res} With regard to the condition $\rho_{\alg}(A)\neq\varnothing$ 
in \prettyref{thm:spec_res} we note that it is fulfilled by \cite[Lemma 5.2, p.~275]{albanese2013} 
if $(T(t))_{t\geq 0}$ is a strongly continuous locally equicontinuous semigroup on sequentially complete $X$ 
such that there is $a\geq 0$ with
\begin{equation}\label{eq:laplace_resolvent}
\Bigl\{\e^{-at}\int_{0}^{t}T(s)(\cd)\d s\;|\;t\geq 0\Bigr\}\text{ is equicontinuous in }\cL(X).
\end{equation}
Indeed, in this case $\{\lambda\in\C\;|\;\re(\lambda)>a\}\subseteq\rho(A)$ and we have for all $\lambda\in\C$ 
with $\re(\lambda)>a$ that
\[
R(\lambda,A)x=\int_{0}^{\infty}\e^{-\lambda s}T(s)x\d s
\]
for all $x\in X$ where the integral above is an improper Riemann integral. 
In particular, if $(T(t))_{t\geq 0}$ is quasi-equicontinu\-ous, thus locally equicontinuous, 
then \eqref{eq:laplace_resolvent} is fulfilled by \cite[Remark 5.3, p.~276]{albanese2013}.
\item \label{it:strong_dual_compl} $X_{\ob}'$ is quasi-complete, so in particular sequentially complete, 
by \cite[11.2.4 Proposition, p.~222]{jarchow1981} and \cite[12.4.2 Theorem, p.~258]{jarchow1981} 
if $X$ is quasi-barrelled or a gDF space. 
\end{enumerate}
\end{rem}

Looking at the spectral decompositions \prettyref{prop:decomp_spec} \ref{it:decomp_spec_2} 
and \ref{it:decomp_spec_3} and having a spectral mapping theorem for the residual spectrum, 
we would like to obtain a spectral mapping theorem for the approximate point spectrum (under suitable conditions). 
However, we will only manage to get the spectral mapping theorem for the bounded (sequential) approximate 
point spectrum. The remaining part of our final section is dedicated to this spectral mapping theorem. 
Let us recall that this kind of spectral mapping theorem already does not hold 
for general strongly continuous semigroups on Banach spaces (see \cite[p.~270--275]{engel_nagel2000}). 
Therefore we have to impose more properties on the semigroups, namely eventual uniform continuity, 
or different properties on the spaces, namely consider generalised Schwartz spaces. 
We start with the following generalisation of (one implication of) \cite[Chap.~IV, 3.9 Lemma, p.~279]{engel_nagel2000}.

\begin{prop}\label{prop:bounded_ap}
Let $X$ be a sequentially complete Hausdorff locally convex space and $(T(t))_{t\geq 0}$ a strongly continuous 
locally equicontinuous semigroup on $X$ with generator $(A,D(A))$. 
Let $t\geq 0$, $\lambda\in\C$, $\lambda\neq 0$, and $(x_i)_{i\in I}$ a net in $X$ such that
\begin{enumerate}[label=\upshape(\roman*), leftmargin=*, widest=iii]
\item \label{it:bounded_ap_11} $(x_i)_{i\in I}$ does not converge to $0$, 
\item \label{it:bounded_ap_12} $\lim_{i\in I} T(t)x_i -\lambda x_i=0$, and
\item \label{it:bounded_ap_13} $\lim_{s\to 0\rlim}\sup_{i\in I} q(T(s)x_i-x_i)=0$ for all $q\in\Gamma_X$.
\end{enumerate}
Then the following assertions hold.
\begin{enumerate}[label=\upshape(\alph*), leftmargin=*]
\item \label{it:bounded_ap} If $(x_i)_{i\in I}$ is bounded, then there is $\mu\in\sigma_{\bap}(A)$ ($\mu\in\sigma_{\bap}^{\seq}(A)$ if $I=\N$) such that $\lambda=\e^{\mu t}$.
\item \label{it:quasi_ap} If $(T(s))_{s\geq 0}$ is quasi-equicontinuous, then there is $\mu\in\sigma_{\ap}(A)$ ($\mu\in\sigma_{\ap}^{\seq}(A)$ if $I=\N$) such that $\lambda=\e^{\mu t}$.
\end{enumerate}
\end{prop}
\begin{proof}
W.l.o.g.~we may assume $\lambda=1$ and $t=1$ by \prettyref{prop:rescale_sg}. 
Let $(a_i)_{i\in I}$ be a bounded net in $\C$.
We define $f_{i}\colon [0,1]\to X$, $f_{i}(s)\coloneqq T(s)a_i x_i$, for $i\in I$. 
Let $q\in\Gamma_X$ and $\eta>0$. By the local equicontinuity of the semigroup there are 
$p_0\in\Gamma_X$ and $C_0\geq 0$ such that $q(T(s)x)\leq C_0p_0(x)$ for all $s\in [0,1]$ 
and $x\in X$. Hence we obtain for all $s,r\in [0,1]$ and $i\in I$
\begin{align*}
  q(T(r)a_i x_i-T(s)a_i x_i)
&=q(T(\min\{r,s\})(T(|s-r|)a_i x_i-a_i x_i))\\
&\leq C_0 p_0(T(|s-r|)a_i x_i-a_i x_i).
\end{align*}
By \ref{it:bounded_ap_13} and the boundedness of $(a_i)_{i\in I}$ there is $\delta>0$ such that for all 
$s,r\in [0,1]$ with $|s-r|<\delta$ it holds that $p_0(T(|s-r|)a_i x_i-a_i x_i)\leq \eta$, implying
\[
q(T(r)a_i x_i-T(s)a_i x_i)\leq C_0\eta .
\] 
We deduce that the family $(f_{i})_{i\in I}$ is (uniformly) equicontinuous. 

In case \ref{it:quasi_ap} the semigroup is quasi-equicontinuous and we choose 
$\omega\in\R$ from \prettyref{defn:semigroup} \ref{it:quasi-equi}. For $q\in \Gamma_X$ we set
\[
\widetilde{q}(x)\coloneqq \sup_{s\geq 0}q(\e^{-\omega s}T(s)x),\quad x\in X.
\]
By \prettyref{prop:rescale_sg} and \cite[Remark 2.2 (i), p.~256]{albanese2013} 
(cf.~\cite[Lemma 2.2, p.~802]{wegner2016}) $\widetilde{\Gamma}_X\coloneqq \{\widetilde{q}\;|\;q\in\Gamma_X\}$ 
defines a fundamental system of seminorms inducing the topology of $X$ which fulfils 
$\widetilde{q}(T(s)x)\leq \e^{\omega s}\widetilde{q}(x)$ for all $s\geq 0$ and $x\in X$.
Due to \ref{it:bounded_ap_11} there are $\varepsilon >0$ and $p\in\Gamma_{X}$ in case \ref{it:bounded_ap} and 
$p\in\widetilde{\Gamma}_X$ in case \ref{it:quasi_ap}, respectively, such that for all 
$j\in I$ there is $i\in I$, $i\geq j$, with $p(x_i)\geq\varepsilon$. 
By passing to a subnet (a subsequence in the case $I=\N$)\footnote{The construction of the subsequence 
in the case $I=\N$ is quite obvious and so we omit it. The subnet in the sense of \cite[p.~188]{munkres2000} 
is constructed as follows in the general case. 
For $j\in I$ let $M_{j}\coloneqq\{i\in I\;|\;i\geq j,\,p(x_i)\geq\varepsilon\}$. 
Then $M_{j}\neq\varnothing$ for all $j\in I$ and we set $M\coloneqq\{(i,j)\;|\;j\in I,\,i\in M_{j}\}$. 
We define a preorder $\leq_{M}$ on $M$ by $(i_0,j_0)\leq_{M}(i_1,j_1)$ $\colon\Leftrightarrow$ $i_0\leq i_1$. 
Equipped with this preorder $M$ becomes a directed set. Indeed, for $(i_0,j_0),(i_1,j_1)\in M$ there is 
$j_2\in I$ such that $i_0 , i_1\leq j_2$ since $I$ is directed. Then there is $i_2\in I$, $i_2\geq j_2$, such that 
$p(x_{i_2})\geq\varepsilon$. Thus $(i_0,j_0),(i_1,j_1)\leq_{M} (i_2,j_2)$. Further, the function $f\colon M\to I$, 
$(i,j)\mapsto i$, is monotone by the definition of $\leq_{M}$, and $f(M)$ is also cofinal since for every $j\in I$ 
there is $i\in I$, $i\geq j$, with $p(x_{i})\geq\varepsilon$, which implies that $(i,j)\in M$ and $f(i,j)=i\geq j$. 
Hence $(z_{i,j})_{(i,j)\in M}$ defined by $z_{i,j}\coloneqq x_{f(i,j)}=x_i$ for $(i,j)\in M$ is a subnet 
of $(x_i)_{i\in I}$.}, which we still denote by $(x_i)_{i\in I}$, 
we may assume that $p(x_i)\geq\varepsilon$ for all $i\in I$.
Due to the Hahn--Banach theorem there are $x_i'\in X'$ such that 
$x_i'(x_i)=p(x_i)\geq\varepsilon$ and $|x_i'(x)|\leq p(x)$ for all $x\in X$ and $i\in I$. 
Now, we set $g_{i}\colon [0,1]\to \C$, $g_i(s)\coloneqq x_i'( T(s)a_i x_i)$.  
By the local equicontinuity of the semigroup there are $p_{1}\in\Gamma_X$ and $C_{1}\geq 0$ such that 
$p(T(s)x)\leq C_{1}p_{1}(x)$ for all $s\in [0,1]$ and $x\in X$. 
In combination with the uniform equicontinuity of $(f_{i})_{i\in I}$ this implies that $(g_i)_{i\in I}$ 
is also uniformly equicontinuous. 

In case \ref{it:bounded_ap} we choose $a_i\coloneqq 1$ for $i\in I$. 
The boundedness of $(x_i)_{i\in I}$ implies that
there is $C_2\geq 0$ such that $p_1(x_i)\leq C_2$ for all $i\in I$.
Therefore we have 
\begin{align*}
 \sup_{s\in [0,1]}\sup_{i\in I}|g_i(s)|
&=\sup_{s\in [0,1]}\sup_{i\in I}|x_i'(T(s)x_i)|
 \leq \sup_{s\in [0,1]}\sup_{i\in I}p(T(s)x_i)
 \leq C_1\sup_{i\in I}p_1(x_i)\\
&\leq C_1 C_2 .
\end{align*}

In case \ref{it:quasi_ap} we choose $a_i\coloneqq \frac{1}{p(x_i)}$ for $i\in I$ and observe that 
$|a_i |\leq \frac{1}{\varepsilon}$ for all $i\in I$. Thus $(a_i)_{i\in I}$ is bounded. 
Since $p\in\widetilde{\Gamma}_X$ in case \ref{it:quasi_ap}, we have
\begin{align*}
 \sup_{s\in [0,1]}\sup_{i\in I}|g_i(s)|
&=\sup_{s\in [0,1]}\sup_{i\in I}\Bigl|x_i'\Bigl(T(s)\frac{x_i}{p(x_i)}\Bigr)\Bigr|
 \leq \sup_{s\in [0,1]}\sup_{i\in I}p\Bigl(T(s)\frac{x_i}{p(x_i)}\Bigr)\\
&\leq \sup_{s\in [0,1]}\sup_{i\in I}\e^{\omega s}p\Bigl(\frac{x_i}{p(x_i)}\Bigr)
 \leq \e^{|\omega|}.
\end{align*}
Hence $(g_i)_{i\in I}$ is uniformly bounded in both cases. 

The Arzel\`{a}--Ascoli theorem yields that $(g_i)_{i\in I}$ is relatively compact in the Banach space 
$\mathcal{C}([0,1])$ of $\C$-valued continuous functions on $[0,1]$. 
Thus there is a subnet (a subsequence in the case $I=\N$), 
which we still denote by $(g_i)_{i\in I}$, that converges to some $g\in\mathcal{C}([0,1])$. 
Since 
\[
 g(0)
=\lim_{i\in I}g_i(0)
=\lim_{i\in I}x_i'(T(0)a_i x_i)
=\lim_{i\in I}x_i'(a_i x_i)
\geq \min(\varepsilon, 1)
> 0,
\]
the continuous function $g$ on $[0,1]$, which we can extend to a $1$-periodic function on $\R$ 
(replacing the value of $g$ at $s=1$ by $g(0)$ if needed), has a 
non-zero Fourier coefficient by Carleson's theorem. So, there is $k\in\Z$ such that 
\[
\int_{0}^{1}\e^{-2\pi\mathsf{i}k s}g(s)\d s \neq 0.
\]
We set $\mu_k\coloneqq 2\pi\mathsf{i} k$ and $z_{i}\coloneqq \int_{0}^{1}\e^{-\mu_k s}T(s)x_i\d s$ for $i\in I$. 
We note that $z_i\in D(A)$ for all $i\in I$ by \prettyref{prop:rescale_identities} and 
\[
 \lim_{i\in I}(A-\mu_k)z_i
\underset{\mathclap{\eqref{eq:rescale_iden_1}}}{=}\lim_{i\in I}(\e^{-\mu_{k}}T(1)-1)x_i
=\lim_{i\in I}(T(1)-1)x_i
\underset{\mathclap{\text{\ref{it:bounded_ap_12}}}}{=}0 .
\]
In addition, we observe that
\[
 \Bigl|x_i'(a_i z_i)-\int_{0}^{1}\e^{-2\pi\mathsf{i}k s}g(s)\d s\Bigr|
=\Bigl|\int_{0}^{1}\e^{-2\pi\mathsf{i}k s}(g_i-g)(s)\d s\Bigr|
\leq \sup_{s\in [0,1]}|g_i(s)-g(s)|
\]
for all $i\in I$ and thus $\lim_{i\in I}x_i'(a_i z_i)=\int_{0}^{1}\e^{-2\pi\mathsf{i}k s}g(s)\d s$ 
because $(g_i)_{i\in I}$ converges to $g$ in $\mathcal{C}([0,1])$.
Next, we claim that there are $j\in I$ and $\delta>0$ such that $|x_i'(a_i z_i)|\geq\delta$ for all $i\in I$ 
with $i\geq j$. 
Indeed, suppose to the contrary that for all $j\in I$ and $\delta>0$ there is $i_j\in I$, $i_j\geq j$, 
such that $|x_{i_j}'(a_{i_j}z_{i_j})|<\delta $. 
Since $\lim_{i\in I}x_i'(a_i z_i)=\int_{0}^{1}\e^{-2\pi\mathsf{i}k s}g(s)\d s$, 
there is $j_\delta \in I$ such that for every $i\in I$, $i\geq j_\delta$, it holds that 
$|x_i'(a_i z_i)-\int_{0}^{1}\e^{-2\pi\mathsf{i}k s}g(s)\d s|\leq \delta$. 
Thus we obtain
\begin{align*}
      \Bigl|\int_{0}^{1}\e^{-2\pi\mathsf{i}k s}g(s)\d s \Bigr|
&\leq |x_{i_{j_\delta}}'(a_{i_{j_\delta}}z_{i_{j_\delta}})|
       +\Bigl|x_{i_{j_\delta}}'(a_{i_{j_\delta}}z_{i_{j_\delta}})
       -\int_{0}^{1}\e^{-2\pi\mathsf{i}k s}g(s)\d s \Bigr|
< 2\delta .
\end{align*}
As $\delta>0$ is arbitrary, we get $\int_{0}^{1}\e^{-2\pi\mathsf{i}k s}g(s)\d s=0$, which is a contradiction. 
We derive that
\[
p(z_i)\geq |x_i'(z_i)|=\frac{1}{|a_i |}|x_i'(a_i z_i)|\geq\min(\varepsilon, 1)\delta>0
\]
for all $i\in I$ with $i\geq j$. Hence $(z_i)_{i\in I}$ does not converge to $0$. 
We deduce that $\mu_k\in\sigma_{\ap}(A)$ ($\mu_k\in\sigma_{\ap}^{\seq}(A)$ if $I=\N$) with $\lambda=1=\e^{\mu_k t}$ 
in both cases.

Furthermore, in case \ref{it:bounded_ap} for $q\in\Gamma_X$ we choose $p_{0}\in\Gamma_X$ and $C_0 \geq 0$ 
as above by the local equicontinuity of the semigroup. 
Then there is $C_3\geq 0$ such that $p_0(x_i)\leq C_3$ for all $i\in I$ since $(x_i)_{i\in I}$ 
is bounded. This implies that
\[
 q(z_i)
\leq \sup_{s\in [0,1]} q(T(s)x_i)
\leq C_0 p_0(x_i)
\leq C_0 C_3
\]
for all $i\in I$, which means that $(z_i)_{i\in I}$ is bounded. 
We conclude that $\mu_{k}\in\sigma_{\bap}(A)$ ($\mu_k\in\sigma_{\bap}^{\seq}(A)$ if $I=\N$) in case \ref{it:bounded_ap}
\end{proof}

We will focus on \prettyref{prop:bounded_ap} \ref{it:bounded_ap} for the remaining part of this section since 
we do not know how to tackle condition \ref{it:bounded_ap_13} in case \ref{it:quasi_ap} without assuming that 
$(x_i)_{i\in I}$ is bounded. We recall that a Hausdorff locally convex space $X$ is a \emph{generalised Schwartz space} 
if every bounded subset of $X$ is already precompact (see \cite[5.2.50 Definition, p.~93]{kruse2023}). 
In particular, Schwartz spaces and semi-Montel spaces are generalised 
Schwartz spaces but infinite-dimensional Banach spaces are not. 

\begin{rem}\label{rem:gen_Schwartz}
Let $X$ be a Hausdorff locally convex space and $(T(t))_{t\geq 0}$ a strongly continuous 
locally equicontinuous semigroup on $X$. 
If $X$ is a generalised Schwartz space, then condition \ref{it:bounded_ap_13} of 
\prettyref{prop:bounded_ap} is fulfilled for all bounded nets $(x_i)_{i\in I}$ in $X$. 
Indeed, by the strong continuity of the semigroup we have $\lim_{s\to 0\rlim}T(s)x-x=0$ for all $x\in X$.
Since the set $\{x_{i}\;|\;i\in I\}$ is bounded in $X$, it is precompact as $X$ is a generalised Schwartz space. 
Hence we obtain
\[
\lim_{s\to 0\rlim}\sup_{i\in I} q(T(s)x_i-x_i)=0
\]
for all $q\in\Gamma_X$ by \cite[8.5.1 Theorem (b), p.~156]{jarchow1981} and the local equicontinuity of the semigroup. 
\end{rem}

Restricting to sequentially complete generalised Schwartz spaces, we get the following spectral mapping theorems 
for the bounded (sequential) approximate point spectrum of strongly continuous locally equicontinuous semigroups. 

\begin{cor}\label{cor:spec_bounded_ap}
Let $X$ be a sequentially complete generalised Schwartz space and $(T(t))_{t\geq 0}$ a strongly 
continuous locally equicontinuous semigroup on $X$ with generator $(A,D(A))$.
Then
\[
\sigma_{\bap}(T(t))\setminus\{0\}=\e^{t \sigma_{\bap}(A)}\quad \text{and}\quad \sigma_{\bap}^{\seq}(T(t))\setminus\{0\}=\e^{t \sigma_{\bap}^{\seq}(A)}
\] 
hold for all $t\geq 0$.
\end{cor}
\begin{proof}
Due to \prettyref{thm:spec_incl} \ref{it:spec_incl_4b} and \ref{it:spec_incl_4c} we only need to prove the 
inclusions ``$\subseteq$''.  
Let $t\geq 0$ and $\lambda\in\C$, $\lambda\neq 0$, such that there is a bounded net $(x_i)_{i\in I}$ 
(sequence if $I=\N$) in $X$ that does not converge to $0$ and fulfils $\lim_{i\in I} T(t)x_i -\lambda x_i=0$. 
Due to \prettyref{prop:bounded_ap} \ref{it:bounded_ap} and \prettyref{rem:gen_Schwartz} there is 
$\mu\in\sigma_{\bap}(A)$ ($\mu\in\sigma_{\bap}^{\seq}(A)$ if $I=\N$) such that $\lambda=\e^{\mu t}$, 
finishing the proof.
\end{proof}

If we want to avoid the restriction to generalised Schwartz spaces, then we need to impose stronger conditions on 
the semigroup, namely eventual uniform continuity. This is our next goal. However, we start with a somewhat converse 
of \prettyref{prop:bounded_ap} \ref{it:bounded_ap}.

\begin{prop}\label{prop:bounded_ap_reverse}
Let $X$ be a sequentially complete Hausdorff locally convex space and $(T(t))_{t\geq 0}$ a strongly continuous 
locally equicontinuous semigroup on $X$ with generator $(A,D(A))$. 
If there are $\mu\in\C$ and a bounded net $(x_i)_{i\in I}$ in $D(A)$ such that 
$(Ax_i-\mu x_i)_{i\in I}$ is bounded, then
\[
\lim_{s\to 0\rlim}\sup_{i\in I} q(T(s)x_i-x_i)=0
\]
for all $q\in\Gamma_X$. 
\end{prop}
\begin{proof}
We observe that $T(s)x_i-x_i=\int_{0}^{s}T(r)Ax_i\d r$ for all $s\geq 0$ and $i\in I$ 
by Equation \eqref{eq:rescale_iden_2} from \prettyref{prop:rescale_identities}. 
Let $q\in\Gamma_X$. By the local equicontinuity of the semigroup there are 
$p\in\Gamma_X$ and $C\geq 0$ such that for all  $s\geq 0$ and $i\in I$
\[
     q(T(s)x_i-x_i)
\leq s\sup_{r\in [0,s]}q(T(r)Ax_i)
\leq C s p(Ax_i)
\leq C s ( p(Ax_i-\mu x_i)+|\mu|p(x_i)).
\]
By the boundedness of $(x_i)_{i\in I}$ and $(Ax_i-\mu x_i)_{i\in I}$ there is $C_1\geq 0$ such that 
\[
q(T(s)x_i-x_i)\leq CC_1 (1+|\mu|)s
\]
for all $s\geq 0$ and $i\in I$, implying $\lim_{s\to 0\rlim}\sup_{i\in I} q(T(s)x_i-x_i)=0$.
\end{proof}

This result enables us to fully generalise \cite[Chap.~IV, 3.9 Lemma, p.~279]{engel_nagel2000} 
in the case of sequences, i.e.~$I=\N$, next.

\begin{cor}\label{cor:bounded_seq_ap}
Let $X$ be a sequentially complete Hausdorff locally convex space and $(T(t))_{t\geq 0}$ a strongly continuous 
locally equicontinuous semigroup on $X$ with generator $(A,D(A))$. Then the following two assertions are equivalent 
for $t\geq 0$ and $\lambda\in\C$, $\lambda\neq 0$.
\begin{enumerate}[label=\upshape(\alph*), leftmargin=*]
\item \label{it:bounded_ap_seq_1} There is a bounded sequence $(x_i)_{i\in \N}$ in $X$ such that 
\begin{enumerate}[label=\upshape(\roman*), leftmargin=*, widest=iii]
\item \label{it:bounded_ap_seq_11} $(x_i)_{i\in\N}$ does not converge to $0$, 
\item \label{it:bounded_ap_seq_12} $\lim_{i\to\infty} T(t)x_i -\lambda x_i=0$, and 
\item \label{it:bounded_ap_seq_13} $\lim_{s\to 0\rlim}\sup_{i\in\N} q(T(s)x_i-x_i)=0$ for all $q\in\Gamma_X$. 
\end{enumerate}
\item \label{it:bounded_ap_seq_2} There is $\mu\in\sigma_{\bap}^{\seq}(A)$ such that $\lambda=\e^{\mu t}$.
\end{enumerate}
\end{cor}
\begin{proof}
``\ref{it:bounded_ap_seq_1}$\Rightarrow$\ref{it:bounded_ap_seq_2}'' This implication follows from \prettyref{prop:bounded_ap} \ref{it:bounded_ap} with $I\coloneqq\N$.

``\ref{it:bounded_ap_seq_2}$\Rightarrow$\ref{it:bounded_ap_seq_1}'' Since $\mu\in\sigma_{\bap}^{\seq}(A)$ such that 
$\lambda=\e^{\mu t}$, there is a bounded sequence $(x_i)_{i\in \N}$ in $D(A)$ that does not converge to $0$ 
and fulfils $\lim_{i\to\infty }Ax_i-\mu x_i=0$. Hence $(Ax_i-\mu x_i)_{i\in\N}$ is bounded as a convergent sequence.
An application of \prettyref{rem:eigenvectors} \ref{it:eigenvec_1} and \prettyref{prop:bounded_ap_reverse} 
with $I\coloneqq\N$ finishes the proof. 
\end{proof}

\begin{prop}\label{prop:bounded_ap_uniform}
Let $X$ be a Hausdorff locally convex space and $(T(t))_{t\geq 0}$ an eventually uniformly continuous 
locally equicontinuous semigroup on $X$. 
If there are a bounded net $(x_i)_{i\in I}$ in $X$, $t_0\geq t_{\ev}$ and $c\in\C$, $c\neq 0$, such that 
$\{T(t_0)x_i-c x_i\;|\;i\in I\}$ is precompact in $X$, then 
\[
\lim_{s\to 0\rlim}\sup_{i\in I} q(T(s)x_i-x_i)=0
\]
for all $q\in\Gamma_X$.
\end{prop}
\begin{proof}
Let $r,s\geq 0$ and $q\in\Gamma_X$. Then we have $r+t_0,s+t_0\geq t_{\ev}$ and 
\[
 \sup_{i\in I}q(T(r)T(t_0)x_i-T(s)T(t_0)x_i)
=\sup_{i\in I}q((T(r+t_0)-T(s+t_0))x_i).
\]
Hence $(T(s))_{s\geq 0}$ is uniformly continuous on $\{T(t_0)x_i\;|\;i\in I\}$ as $(x_i)_{i\in I}$ is bounded 
and the semigroup eventually uniformly continuous on $X$. Further, the strong continuity of the semigroup implies that 
\[
\lim_{s\to 0\rlim}T(s)(T(t_0)x_i-c x_i)-(T(t_0)x_i-c x_i)=0
\]
for all $i\in I$. Since $\{T(t_0)x_i-c x_i\;|\;i\in I\}$ is precompact in $X$, we obtain
\[
\lim_{s\to 0\rlim}\sup_{i\in I} q(T(s)(T(t_0)x_i-c x_i)-(T(t_0)x_i-c x_i))=0
\]
for all $q\in\Gamma_X$ by \cite[8.5.1 Theorem (b), p.~156]{jarchow1981} and the local equicontinuity of the semigroup, 
so $(T(s))_{s\geq 0}$ is uniformly continuous on $\{T(t_0)x_i-c x_i\;|\;i\in I\}$.
Therefore $(T(s))_{s\geq 0}$ is uniformly continuous on $\{x_{i}\;|\;i\in I\}$ because 
$x_i=\frac{1}{c}T(t_0)x_i-\frac{1}{c}(T(t_0)x_i-cx_i)$ for all $i\in I$.
\end{proof}

\begin{rem}
Let $X$ be a Hausdorff locally convex space.
\begin{enumerate}[label=\upshape(\alph*), leftmargin=*]
\item If $X$ is quasi-complete, barrelled and has the Grothendieck and Dunford--Pettis properties 
(see \cite[p.~147--148]{albanese2010b}), then every strongly 
continuous semigroup on $X$ is already uniformly continuous by \cite[Theorem 7, p.~313]{albanese2012}. 
For instance, this is fulfilled by \cite[Corollary 3.8, p.~155]{albanese2010b} and 
\cite[11.5.2 Proposition, p.~230]{jarchow1981} if $X$ is a Montel space. 
\item An operator $S\in\cL(X)$ is called \emph{Montel} 
if it maps bounded sets to relatively compact sets (see \cite[p.~268]{albanese2013}). 
Clearly, any $S\in\cL(X)$ is Montel if $X$ is a semi-Montel space. 
If $(T(t))_{t\geq 0}$ is a strongly continuous locally equicontinuous semigroup on $X$ and there is $\widetilde{t}>0$ 
such that $T\left(\widetilde{t}\,\right)$ is Montel, then $(T(t))_{t\geq 0}$ is eventually uniformly continuous 
by \cite[Lemma 4.3, p.~270]{albanese2013}.
\end{enumerate}
\end{rem}

Now, we are able to prove the spectral mapping theorem for the bounded sequential point spectrum of 
eventually uniformly continuous locally equicontinuous semigroups which generalises 
\cite[Chap.~IV, 3.10 Spectral Mapping Theorem for Eventually Norm-Continuous Semigroups, p.~280]{engel_nagel2000} 
(cf.~\cite[Theorem 2.3.2, p.~37]{vanneerven1996}).

\begin{cor}\label{cor:spec_bounded_seq_ap}
Let $X$ be a sequentially complete Hausdorff locally convex space and $(T(t))_{t\geq 0}$ an 
eventually uniformly continuous locally equicontinuous semigroup on $X$ with generator $(A,D(A))$.
Then
\[
\sigma_{\bap}^{\seq}(T(t))\setminus\{0\}=\e^{t \sigma_{\bap}^{\seq}(A)}
\] 
holds for all $t\geq 0$.
\end{cor}
\begin{proof}
Due to \prettyref{thm:spec_incl} \ref{it:spec_incl_4c} we only need to prove the inclusion ``$\subseteq$''.  
Let $t\geq 0$ and $\lambda\in\sigma_{\bap}^{\seq}(T(t))$, $\lambda\neq 0$. 
If $t=0$, then $T(t)=\id$ and both sides of the inclusion ``$\subseteq$'' are equal to $\{1\}$ if $X\neq\{0\}$, or 
the empty set if $X=\{0\}$. Let $t>0$. Then there is a bounded 
sequence $(x_i)_{i\in\N}$ that does not converge to $0$ and fulfils $\lim_{i\to\infty} T(t)x_i -\lambda x_i=0$. 
Since $t>0$, there is $k\in\N$ such that $kt\geq t_{\ev}$. We claim that 
$\lim_{i\to\infty}T(kt)x_i -\lambda^{k} x_i=0$. If $k=1$, then this is clearly fulfilled. 
If $k>1$, then it follows from writing 
\[
 T(kt)x_i-\lambda^{k}x_i
=T(t)\bigl(T((k-1)t)x_i-\lambda^{k-1}x_i\bigr)+\lambda^{k-1}(T(t)x_i-\lambda x_i) 
\]
for all $i\in \N$. Since the sequence $(T(kt)x_i -\lambda^{k} x_i)_{i\in\N}$ converges in $X$, the set 
$\{T(kt)x_i -\lambda^{k} x_i\;|\;i\in\N\}$ is precompact in $X$. 
Due to \prettyref{prop:bounded_ap_uniform} with $I\coloneqq \N$, $t_0\coloneqq kt$ and $c\coloneqq \lambda^{k}$ 
we get that 
\[
\lim_{s\to 0\rlim}\sup_{i\in \N} q(T(s)x_i-x_i)=0
\]
for all $q\in\Gamma_X$. Applying \prettyref{cor:bounded_seq_ap}, we conclude our statement.  
\end{proof}

\prettyref{cor:spec_bounded_seq_ap} in combination with \prettyref{rem:alg_resolvent} \ref{it:cl_gr_1}, \prettyref{prop:seq_ap_bounded_ap}, \prettyref{prop:decomp_spec} \ref{it:decomp_spec_1} and \prettyref{thm:spec_res} 
implies the spectral mapping theorem 
\eqref{eq:full_spec_map} for eventually uniformly continuous semigroups on Banach spaces. Further, the aforementioned 
results together with \prettyref{prop:decomp_spec} \ref{it:decomp_spec_2} and \ref{it:decomp_spec_3}, 
\prettyref{thm:spec_incl} \ref{it:spec_incl_4} and \ref{it:spec_incl_4a} and 
\prettyref{cor:spec_bounded_ap} also yield the following observation.

\begin{rem}\label{rem:spectral_bound}
Let $X$ be a Hausdorff locally convex space and $(T(t))_{t\geq 0}$ a strongly continuous 
semigroup on $X$ with generator $(A,D(A))$.
\begin{enumerate}[label=\upshape(\alph*), leftmargin=*]
\item Let $X_{\ob}'$ be sequentially complete, $\rho_{\alg}(A)\neq\varnothing$ and 
$(T(t))_{t\geq 0}$ locally equicontinuous. 
If $X$ is one of the spaces listed in \prettyref{rem:alg_resolvent} and 
\begin{enumerate}[label=\upshape(\roman*), leftmargin=*, widest=iii]
\item \label{it:full_spec_1} $X$ is a complete generalised Schwartz space, or 
\item \label{it:full_spec_2} $X$ is a sequentially complete sequential generalised Schwartz space, or 
\item \label{it:full_spec_3} $X$ is sequentially complete, sequential and $(T(t))_{t\geq 0}$ eventually 
uniformly continuous, 
\end{enumerate}
and $\sigma_{\ap}(A)=\sigma_{\bap}(A)$ and $\sigma_{\ap}(T(t))=\sigma_{\bap}(T(t))$ 
for all $t\geq 0$ in case \ref{it:full_spec_1}, and 
$\sigma_{\ap}^{\seq}(A)=\sigma_{\bap}^{\seq}(A)$ and $\sigma_{\ap}^{\seq}(T(t))=\sigma_{\bap}^{\seq}(T(t))$ 
for all $t\geq 0$ in cases \ref{it:full_spec_2}--\ref{it:full_spec_3}, respectively, 
then \eqref{eq:full_spec_map} holds, i.e. 
\begin{equation}\label{eq:spectral_bound_1}
\sigma(T(t))\setminus\{0\}=\e^{t\sigma(A)},\quad t\geq 0.
\end{equation}
Unfortunately, we do not know e.g.~nice sufficient conditions when $\sigma_{\ap}(B)=\sigma_{\bap}(B)$ 
or $\sigma_{\ap}^{\seq}(B)=\sigma_{\bap}^{\seq}(B)$ holds for all closed linear operators $(B,D(B))$ on $X$ 
apart from the case that $X$ is Banach space. 
\item\label{it:spectral_bound_2} Let $(T(t))_{t\geq 0}$ be quasi-equicontinuous. Then 
\[
\omega_{0}(T)\coloneqq \inf\{\omega\in\R\;|\;(\mathrm{e}^{-\omega t}T(t))_{t\geq 0}\text{ is equicontinuous}\}
\]
is called the \emph{growth bound} of $(T(t))_{t\geq 0}$ by \cite[Definition 2.1, p.~802]{wegner2016}. 
If $\omega_0(T)<0$, we directly get from this definition the asymptotic behaviour of the mild solution of the abstract Cauchy problem from the introduction, namely, then there is $\omega<0$ such that $\lim_{t\to\infty}\mathrm{e}^{-\omega t}T(t)=0$ in $\cL_{\ob}(X)$ 
(cf.~\cite[Chap.~V, 1.7 Proposition, p.~299]{engel_nagel2000}). 
Since often only the generator is given and the semigroup is not explicitly known, we 
are interested in a sufficient criterion relying on the knowledge of $A$ alone that gives us $\omega_0(T)<0$. 
We call 
\[
s(A)\coloneqq\sup\{\re(\lambda)\;|\;\lambda\in\sigma(A)\}
\]
the \emph{spectral bound} of $A$ (see \cite[Chap.~II, 1.12 Definition, p.~57]{engel_nagel2000} in the case where $X$ 
is a Banach space). We note that this definition of the growth bound differs from the one given in 
\cite[Definition 3.1, p.~804--805]{wegner2016}. However, we still have 
\[
-\infty\leq s(A)\leq\omega_0(T)<\infty
\]
by the remark directly after the proof of \cite[Theorem 4.1, p.~806]{wegner2016}.\footnote{In \cite{wegner2016} it is assumed throughout the paper that $X$ is a Fr\'echet space (see \cite[p.~802]{wegner2016}) but this is not relevant for the proof of \cite[Theorem 4.1, p.~806]{wegner2016}.} 
Suppose that $s(A)=\omega_{0}(T)$. If we know that $s(A)<0$, then there is $\omega<0$ such that 
$\lim_{t\to\infty}\mathrm{e}^{-\omega t}T(t)=0$ in $\cL_{\ob}(X)$ by our consideration above. 
Hence, we are interested in the question when $s(A)=\omega_{0}(T)$ holds. 
In the case that $X$ is a Banach space it is known that $s(A)=\omega_{0}(T)$ is true by 
\cite[Chap.~V, 1.9 Lemma, p.~301]{engel_nagel2000} if the spectral mapping theorem \eqref{eq:spectral_bound_1} holds. 
So, the open question is whether in the locally convex case the spectral mapping theorem \eqref{eq:spectral_bound_1} 
(or some version of it) also implies $s(A)=\omega_{0}(T)$. It seems that this is in general not known even for 
uniformly continuous semigroups on Fr\'echet spaces (cf.~the comments below \cite[Example 5.1, p.~807--809]{wegner2016}).
\end{enumerate}

\end{rem}

\bibliography{biblio_Spectral_rev2}
\bibliographystyle{plainnat}
\end{document}